\documentclass[12pt]{article}
\usepackage[english]{babel}
\usepackage[latin1]{inputenc}
\usepackage{amsmath}
\usepackage{amssymb}
\usepackage{amsthm}
\usepackage{mathtools}
\usepackage{latexsym}
\usepackage{verbatim}
\usepackage{enumitem}   
\usepackage{pinlabel}
\usepackage{pifont}
\usepackage{graphics}
\usepackage{multirow}
\usepackage{amsfonts}
\usepackage{graphicx}
\usepackage{epstopdf}
\usepackage{pdflscape}
\usepackage{multicol}
\usepackage{framed}
\usepackage[pdftex,%
    nomarginpar,%
    lmargin=1in,%
    rmargin=1in,%
    tmargin=1in,%
    bmargin=1in,%
]{geometry}

\usepackage[margin=1cm]{caption}

\def\ra{\rightarrow}

\def\Z{\mathbb{Z}}
\def\N{\mathbb{N}}

\newcommand{\Pcal}{\mathcal{P}}
\newcommand{\Lcal}{\mathcal{L}}
\newcommand{\Dcal}{\mathcal{D}}
\newcommand{\Tcal}{\mathcal{T}}
\newcommand{\mc}[1]{\mathcal{#1}}
\DeclareMathOperator{\interior}{int}

\usepackage[dvipsnames]{xcolor}
\usepackage{hyperref}
\hypersetup{
    colorlinks=true,
    linkcolor=NavyBlue,
    citecolor=NavyBlue,
}
\definecolor{shadecolor}{HTML}{EDF8FC}

\usepackage{csquotes}
\usepackage[backend=biber,style=alphabetic,sorting=nyt,url=false,doi=false,isbn=false,backref=true,maxbibnames=10]{biblatex}
\DefineBibliographyStrings{english}{
  backrefpage={$\uparrow$},
  backrefpages={$\uparrow$}
}
\addto\extrasenglish{%
}

\newtheorem{theorem}{Theorem}[section]
\newtheorem{corollary}{Corollary}[section]
\newtheorem{proposition}{Proposition}[section]
\newtheorem{lemma}{Lemma}[section]

\newtheorem{conjecture}{Conjecture}[section]
\newtheorem{claim}{Claim}[section]
\newtheorem{question}{Question}[section]
\newtheorem{problem}{Problem}[section]

\theoremstyle{definition}
\newtheorem{definition}{Definition}[section]
\newtheorem{example}{Example}[section]
\newtheorem{remark}{Remark}[section]

\makeatletter
\let\c@conjecture=\c@theorem
\let\c@corollary=\c@theorem
\let\c@proposition=\c@theorem
\let\c@lemma=\c@theorem
\let\c@definition=\c@theorem
\let\c@problem=\c@theorem
\let\c@example=\c@theorem
\let\c@remark=\c@theorem
\let\c@equation\c@theorem
\let\c@question\c@theorem
\makeatother

\def\makeautorefname#1#2{\expandafter\def\csname#1autorefname\endcsname{#2}}
\makeautorefname{proposition}{Proposition}
\makeautorefname{lemma}{Lemma}
\makeautorefname{definition}{Definition}
\makeautorefname{problem}{Problem}
\makeautorefname{example}{Example}
\makeautorefname{question}{Question}
\makeautorefname{conjecture}{Conjecture}
\makeautorefname{section}{Section}
\makeautorefname{claim}{Claim}
\makeautorefname{equation}{Equation}
\makeautorefname{remark}{Remark}
\makeautorefname{corollary}{Corollary}

\newcounter{step}

\makeatletter
\def\@tocline#1#2#3#4#5#6#7{\relax
  \ifnum #1>\c@tocdepth 
  \else
    \par \addpenalty\@secpenalty\addvspace{#2}%
    \begingroup \hyphenpenalty\@M
    \@ifempty{#4}{%
      \@tempdima\csname r@tocindent\number#1\endcsname\relax
    }{%
      \@tempdima#4\relax
    }%
    \parindent\z@ \leftskip#3\relax \advance\leftskip\@tempdima\relax
    \rightskip\@pnumwidth plus4em \parfillskip-\@pnumwidth
    #5\leavevmode\hskip-\@tempdima
      \ifcase #1
       \or\or \hskip 1em \or \hskip 2em \else \hskip 3em \fi%
      #6\nobreak\relax
    \dotfill\hbox to\@pnumwidth{\@tocpagenum{#7}}\par
    \nobreak
    \endgroup
  \fi}
\makeatother

\newcommand\T{\rule{0pt}{2.6ex}}       
\newcommand\B{\rule[-1.2ex]{0pt}{0pt}} 

\title{Pants distances of knotted surfaces in $4$-manifolds}

\author{Rom\'an Aranda, Sarah Blackwell, Devashi Gulati, Homayun Karimi,\\ Geunyoung Kim, Nicholas Paul Meyer, and Puttipong Pongtanapaisan}

\date{}

\begin{document}
\maketitle
\begin{abstract}
We define a pants distance for knotted surfaces in 4-manifolds, which generalizes the complexity studied by Blair-Campisi-Taylor-Tomova for surfaces in the 4-sphere. We determine that if the distance computed on a given diagram does not surpass a theoretical bound in terms of the multisection genus, then the pair $(X, F)$ admits a standard form (i.e., has simple topology). Furthermore, we calculate the exact values of our invariants for many new examples, such as the spun lens spaces. We provide a characterization of genus two quadrisections with distance at most six.
\end{abstract}

\section{Introduction} \label{sec:intro}

The study of 3-manifold topology via complexes of curves has been an active and fruitful area of research in low-dimensional topology. 
In 2001, Hempel used the curve complex of a surface to define an integer complexity for Heegaard splittings \cite{Hempel01}, which measures how ``far apart'' the two sides of the Heegaard splitting are by computing the distance between compressing disks on each side. Hempel then used this notion of distance to obstruct the existence of essential tori in the given 3-manifold.  
In the last twenty years, versions of Hempel's distance have been used to make progress on important problems in low-dimensional topology \cite{johnson2010mapping,blair2017exceptional,blair2020distortion}.

A trisection of a closed 4-manifold is a decomposition into three 4-dimensional handlebodies \cite{gay2016trisecting}, and can be thought of as a 4-dimensional analog of Heegaard splittings. The data of a trisection is determined by three 3-dimensional handlebodies which can be described by various collections of curves on a closed surface. 
In 2018, Kirby and Thompson used trisections to introduce an invariant of closed 4-manifolds in the spirit of Hempel's distance for Heegaard splittings
\cite{kirby2018new}. Their invariant, denoted by $\Lcal(X)$, uses the cut complex of a surface to measure how far apart the three handlebodies of a trisection are.

In the following years, versions of the $\Lcal$-invariant for other 4-dimensional structures have appeared in the literature. Castro, Islambouli, Miller, and Tomova used the arc-cut complex and relative trisections to define $\Lcal$-invariants of compact 4-manifolds with boundary \cite{castro2022relative}. Blair, Campisi, Taylor, and Tomova used the pants complex and bridge trisections to define the $\Lcal$-invariant of knotted surfaces in $S^4$ \cite{blair2020kirby}.
Aranda, Pongtanapaisan, and Zhang used the dual curve complex and bridge trisections to define a new $\Lcal^*$-invariant for knotted surfaces in $S^4$ \cite{aranda2022bounds}. 
Besides their motivation, all $\Lcal$-invariants above share a common trait: they all detect the simpler topology. For example, it has been proved that the \hbox{$\Lcal$-invariant} of a surface in $S^4$ can be used to detect when a knotted surface is not prime or smoothly unknotted \cite{blair2020kirby}. These results indicate that $\Lcal$ may offer new insights into the smooth topology in four dimensions. 

The aim of this paper is to introduce generalizations of the $\Lcal$-invariants of \cite{blair2020kirby} and \cite{aranda2022bounds} for knotted surfaces in arbitrary 4-manifolds. 
Our new invariants (see \autoref{sec:invariants}), denoted by $\Lcal_n^*$ and $\Lcal_n^\Pcal$, 
are reminiscent of the bridge and pants complexities of knots in $3$-manifolds studied by Johnson and Zupan \cite{johnson2006heegaard,zupan2013bridge}. 
In practice, they are invariants of bridge $n$-sections of (4-manifold, surface) pairs \hbox{(see \autoref{subsec:multisections}).} 

\subsection{Comparison to previous \texorpdfstring{$\Lcal$}{L}-invariants}

Here we discuss more precisely how our invariants fit into the existing literature. When the 4-manifold is $S^4$, our invariants $\Lcal_3^\Pcal$ and $\Lcal_3^*$ agree with the invariants $\Lcal$ and $\Lcal^*$ of a knotted surface from \cite{blair2020kirby} and \cite{aranda2022bounds}, respectively.

When the surface is not present, we obtain a complexity measure of a closed $4$-manifold independent from the one originally defined by Kirby and Thompson. 
In particular, \autoref{cor:ComparingKT1} and \autoref{cor:ComparingKT2} show that there are two infinite families of $4$-manifolds $X_p$ and $Y_m$ such that $\Lcal(X_p) \to \infty$ as $p \to \infty$ while $\Lcal^*_3(X_p)$ remains bounded, and that $\Lcal^*_3(Y_m) \to \infty$ as $m\to \infty$ while $\Lcal(Y_m)$ remains bounded. We consider the $n=3$ version of $\Lcal^*_n$ here in order to compare with Kirby and Thompson's $\Lcal$, which was defined for trisections. We do \textit{not} consider the $\Lcal_n^\Pcal$ invariant here, for reasons we will see shortly.

The distinction between our invariants and the original Kirby-Thompson $\Lcal$-invariant is the metric space in which we compute distance. The Kirby-Thompson $\Lcal$-invariant uses the cut complex, whereas the invariants $\Lcal^*_n$ and $\Lcal^{\Pcal}_n$ use the dual curve complex and pants complex, respectively (see \autoref{sec:invariants} for details). 
As the pants complex is a subcomplex of the dual curve complex, we always have that $\Lcal_n^*\leq \Lcal_n^{\mc P}$ (see \autoref{remark:compare}). 
Because of this, many of our results are stated (solely) in terms of the $\Lcal_n^*$ version.

\subsection{Detection results} \label{subsec:detection}
A good feature of Kirby-Thompson's $\mathcal{L}$-invariant is that it detects the 4-sphere: if $X$ is a homology $4$-sphere with $\mathcal{L}(X)\leq 1$, then $X$ is diffeomorphic to the 4-sphere.  
Similar results hold 
for the $\Lcal$-invariants of 4-manifolds with boundary \cite{castro2022relative} and surfaces in $S^4$ \cite{blair2020kirby}. 
In this paper, we present an incarnation of the now classic detection result for $\Lcal_n^*$-invariants of (4-manifold, surface) pairs.  

\begin{theorem} \label{thm:leq1}
Let $F$ be an embedded surface in a closed $4$-manifold $X$ with no $S^1\times S^3$ summands. If $\Lcal_n^*(X,F)\leq 1$ for some $n\geq 3$, then $F$ is smoothly isotopic to an unknotted surface in $X\cong\#^i S^2\times S^2\#^j \mathbb{CP}^2\#^k\overline{\mathbb{CP}^2}$. 
\end{theorem}

\begin{proof}
Suppose that $(X,F)$ admits a bridge multisection $\Tcal$ with $\Lcal_n^*(\Tcal)\leq 1$. \autoref{thm:ogawaV1} implies that $\Tcal$ is completely decomposable (see \autoref{subsec:new_from_old}). \autoref{lem:completely_decomposable} states that for such pairs $F$ is unknotted and $X$ is standard. 
\end{proof}

This bound is the best possible as \autoref{exmp:s2xpt} exhibits a complex curve in \hbox{$\mathbb{CP}^1\times \mathbb{CP}^1$} of bi-degree (1,0) with $\Lcal_4^*=2$. We have proved stronger results: in \autoref{sec:ogawa} we classify multisections with complexity at most two.
Recently, Ogawa improved the detection result for Kirby-Thompson's $\Lcal$-invariant, showing that in fact, $\mathcal{L}(X)\leq 2$ implies that $X$ is diffeomorphic to the $4$-sphere \cite{ogawa2021trisections}. We prove that the same result holds for $\Lcal_3^*$. The following is a stronger version of \cite[Theorem~3.5]{aranda2022bounds}. 
\begin{theorem} 
Let $F$ be an embedded surface in a closed $4$-manifold $X$ with no $S^1\times S^3$ summands. If $\Lcal_3^*(X,F)\leq 2$, then $F$ is smoothly isotopic to an unknotted surface in \hbox{$X\cong\#^i S^2\times S^2\#^j \mathbb{CP}^2\#^k\overline{\mathbb{CP}^2}$.} 
\end{theorem} \label{thm:leq2}
\begin{proof}
Let $\Tcal$ be a bridge trisection of the pair $(X,F)$ with $\Lcal_3^*(\Tcal)=2$. \autoref{thm:ogawaV2a} implies that $\Tcal$ is a c-connected sum (see \autoref{subsec:new_from_old}) of bridge trisections with complexities $(g,b)\in \{(0,1)$, $(0,2)$, $(0,3)$, $(1,0)$, $(1,1)$, $(2,0)\}$ and $\Lcal_3^*\leq 2$. All such $(g,b)$-bridge trisections have been classified: 
manifolds with $(g,b)=(g,0)$ and $g\leq 2$ are connected sums of copies of $S^1\times S^3$, $\pm\mathbb{CP}^2$, and $S^2\times S^2$~\cite{MZgenustwo}. Surfaces with $(g,b)=(0,b)$ with $b\leq 3$ are unknotted inside $S^4$~\cite[Theorem~1.8]{meier2017bridgeS4}. If $(g,b)=(1,1)$, then the trisected surface is either unknotted inside a 4-manifold of the form \hbox{$X\cong\#^i S^2\times S^2\#^j \mathbb{CP}^2\#^k\overline{\mathbb{CP}^2}$}, or is a complex curve of degree one or two in $\mathbb{CP}^2$~\cite[Proposition~3.2]{meier2018bridge4manifolds}. 
In \autoref{exam:deg12_cp2}, we showed that $\Lcal_3^*=3$ for complex curves $\mathcal{C}_d$ of degree $d=1,2$, so $(\mathbb{CP}^2,\mathcal{C}_d)$ is not a connected summand for $(X,F)$. In conclusion, as $X$ has no $S^1\times S^3$ summands, $(X,F)$ is the c-connected sum of pairs of the form $(S^4,F')$ and $(X',\emptyset)$ with $F'$ an unknotted surface and $X'$ a copy of $\pm \mathbb{CP}^2$ or $S^2\times S^2$.
\end{proof}

As with other meaningful invariants, it is natural to ask whether our complexity can be arbitrarily large. We give a positive answer to this question in the following theorem, and remark that when $X$ is diffeomorphic to $S^4,$ we recover \cite[Theorem~1.1]{aranda2022bounds}. Recall that a 4-manifold is \textit{prime} if it does not decompose as a connected sum unless one summand is $S^4.$ In the same spirit, a knotted surface is \textit{prime} if it is nontrivial and it cannot decompose as a connected sum unless one summand is a trivial 2-knot. A pair $(X,F)$ is prime if both $X$ and $F$ are prime.

\begin{theorem} \label{thm:gb}
Let $(X,F)$ be a prime ($4$-manifold, surface) pair. Then, 
$$\mathcal{L}_3^*(X,F) \geq 7(g-1)-\chi(X)+4b+\chi(F),$$
where $(g,b)$ are the smallest trisection genus and bridge number of the pair $(X,F)$.
\end{theorem}

\begin{proof}
Let $\Tcal$ be a minimal bridge trisection for $(X,F)$ with $\mathcal{L}_3^*(\Tcal) < 7(g-1)-\chi(X)+4b+\chi(F)$. According to \autoref{thm:lower_bound_1_V2}, $\mathcal{T}$ is c-reducible. As $\Tcal$ is minimal, it must decompose $(X,F)$ into two nontrivial summands. This contradicts the primality of $(X,F)$.
\end{proof}

\subsection{Classification of genus two quadrisections}
Genus two trisections were shown to be standard by Meier and Zupan \cite{MZgenustwo}, but the classification of genus three trisections is still open.
A complete list of manifolds admitting genus three trisections was proposed by Meier \cite{Meier}. Progress towards this problem was made by the first author and Moeller \cite{aranda2019diagrams}, where various trisection diagrams were shown to belong to Meier's list.

An advantage of multisections is that generally speaking, a 4-manifold may be represented by a central surface of a lower genus \cite[Theorem~8.4]{islambouli2020multisections}. In their work, Naylor and Islambouli showed that \hbox{4-manifolds} admitting genus two quadrisections form a subfamily of all 4-manifolds admitting genus three trisections. 
This leads to the following problem.
\begin{problem}
Classify $4$-manifolds admitting genus two quadrisections. \label{ques:classify4sec}
\end{problem}

In \autoref{sec:genus_two_quadrisec} we solve \autoref{ques:classify4sec} for 4-manifolds with $\Lcal_4^*$-invariant at most six. 

\begin{theorem}
Let $X$ be a $4$-manifold with a $(2,1)$-quadrisection $\Tcal$. If $\Lcal_4^*(\Tcal)\leq 6$, then $X$ is diffeomorphic to the spin of a lens space, the twisted-spin of a lens space, $\#^2 S^1\times S^3$, or $\#^h S^1\times S^3\#^iS^2\times S^2\#^j\mathbb{CP}^2\#^k \overline{\mathbb{CP}^2}$ where $h=\{0,1\}$.
\end{theorem}

\begin{proof}
\autoref{thm:L>=6} states that if $\Lcal_4^*(\Tcal)<6$, then $X$ is diffeomorphic to $S^4$, $\#^2 S^1\times S^3$, or is a connected sum of copies of $S^2\times S^2$, $\pm \mathbb{CP}^2$, and at most one copy of $S^1\times S^3$. If $\Lcal_4^*(\Tcal)=6$, then \autoref{thm:L=6} implies that $X$ is diffeomorphic to (possibly twisted) spun lens space, $S^1\times S^3\# S^2\times S^2$, or $S^1\times S^3\#S^2\widetilde{\times} S^2$. 
\end{proof}

\subsection{Bridge distance in dimension three}
The three-dimensional version of $\Lcal_n^*$ is an invariant $D(\Sigma)$ introduced by Zupan in \cite{zupan2013bridge}. 
$D(\Sigma)$ measures the distance 
between the disk sets in the bridge splitting of a knot in a closed 3-manifold. 
He showed that if $K$ is a knot and $M$ has no $S^1\times S^2$ summands, then \hbox{$(g,b)$-bridge} positions of $(M,K)$ have complexity $D(\Sigma)$ at least $g+b-1$. 
In particular, the distance zero splittings are 1-bridge unknots in $S^3$.
The main technical lemma of the present work is a result in the same vein as Zupan's. The following lemma is contained in \autoref{lem:efficient}.

\begin{lemma}
Let $\Sigma$ be a $(g,b)$-bridge splitting of an unlink in $\#^k S^1\times S^2$. Then $D(\Sigma)$ is exactly $g-k+b-c$. 
\end{lemma}

We conjecture that these results can be merged together as follows. 

\begin{conjecture}
Let $K$ be a $c$-component link embedded in a compact 3-manifold $M$ with $k$ $S^1\times S^2$ summands. For any $(g,b)$-bridge position of $(M,K)$, 
\[ D(\Sigma)\geq g-k+b-c.\]
Moreover, equality holds if and only if $K$ is an unlink in $\#^k S^1\times S^2$. 
\end{conjecture}

\subsection*{Acknowledgements}
The authors thank Maggy Tomova and the University of Central Florida for making this research and collaboration possible. SB would additionally like to thank the Max Planck Institute for Mathematics in Bonn for its hospitality and financial support. PP is supported by the Pacific Institute for the Mathematical Sciences (PIMS). The research and findings may not reflect those of the Institute.

\setcounter{tocdepth}{2}
\tableofcontents

\section{Background and preliminaries} \label{sec:definitions}

We work in the smooth category and are interested in embeddings of objects up to isotopy. The symbol $X$ is reserved for 4-manifolds which are smooth, orientable, closed, and connected. $F$ will always be a closed surface embedded in $X$.
We denote by $\Sigma=\Sigma_{g,p}$ an orientable surface of genus $g$ and with $p$ punctures. If we write  $\Sigma_g$ we mean that $\Sigma$ is closed. 

A three-dimensional \textit{handlebody} $H$ is the result of adding 1-handles to a 3-ball. A \textit{trivial tangle} is a collection of properly embedded arcs $T\subset H$ which can be simultaneously pushed (rel. $\partial T$) to the boundary of $H$ as embedded arcs. We will write $(H, T)$ (or just $T$) to refer to a trivial tangle $T\subset H$. Such projections of $T$ onto $\partial H$ are called \textit{shadows}. 
A \textit{Heegaard splitting} is a decomposition of a closed orientable 3-manifold $M$ as the union of two handlebodies. We say that a link $L\subset M$ is in \textit{bridge position} if $(M,L)$ can be decomposed as the union of two trivial tangles $(M,L)=(H_1,T_1)\cup (H_2,T_2)$ glued along their boundaries. The double intersection is a punctured surface $\Sigma$ which will play a key role in this work. If $\Sigma$ has genus $g$ and $2b$ punctures, we say that $L$ is in $(g,b)$-bridge position. See \cite{Schultens3m} for an introduction to Heegaard splittings and bridge decompositions. 

\subsection{Complexes of curves} \label{subsec:complexes} 

Fix $g,b\geq 0$ and $\Sigma=\Sigma_{g,2b}$. 
A \textit{pair of pants} is a planar surface with Euler characteristic equal to $-1$:
a thrice-punctured sphere, a disk with two punctures, a once-punctured annulus, or a sphere with three open disks removed.  
A \textit{pants decomposition} of $\Sigma$ is a collection of simple closed curves that split the surface into pairs of pants. There are $3g+2b-3$ curves in any pants decomposition. We say that $\Sigma$ is an \textit{admissible} surface if $3g+2b-3>0$. 

The \textit{dual curve graph} $\mathcal{C}^*(\Sigma)$ of a surface $\Sigma$ is a simplicial complex where each vertex is a pants decomposition of the surface. 
Two vertices $a,b$ are connected by an edge if their representatives differ by a single curve (see \cite[Figure~3]{zupan2013bridge}). Let $d^*(a,b)$ denote the graph length of a geodesic in $\mathcal{C}^*(\Sigma)$ connecting $a$ and $b$. Given two subsets $A,B\subset \mathcal{C}^*(\Sigma)$, the \textit{distance} $d^{*}(A,B) = \min\lbrace d^{*}(a,b) \mid a\in A, b\in B\rbrace$.

The \textit{pants graph} $\mathcal{P}(\Sigma)$ is a subgraph of $\mc{C}^*(\Sigma)$. Two pants decompositions are connected by an edge if their representatives differ one of two moves: \textit{A-move} or \textit{S-move}. An A-move replaces a curve with another curve that intersects it geometrically twice but algebraically zero, and an S-move replaces a curve with another curve that intersects it once. See \autoref{fig:moves}. As before, let $d^{\mc{P}}(a,b)$ denote the graph length of a geodesic in $\mathcal{P}(\Sigma)$ connecting two vertices $a$ and $b$, and for two subsets $A,B\subset \mathcal{P}(\Sigma)$, let $d^{\mc{P}}(A,B) = \min\lbrace d^{\mc{P}}(a,b) \mid a\in A, b\in B\rbrace$.

\begin{figure}[h]
\centering
\includegraphics[width=16cm]{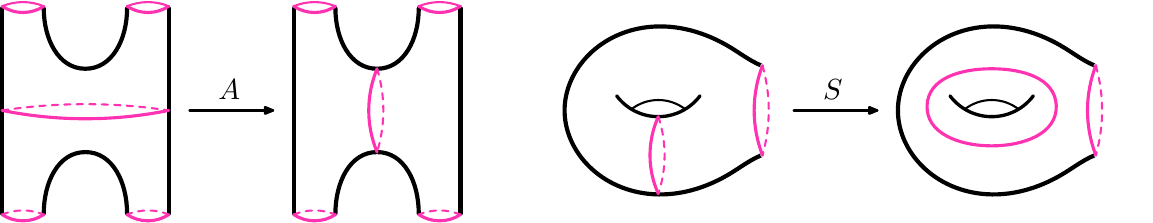}
\caption{Local model of an A-move and S-move.}
\label{fig:moves}
\end{figure}

Let $(H,T)$ be a tangle. The boundary of $H$ is a $2b$ punctured surface $\Sigma$ of genus $g$. A curve in a surface $\Sigma$ is \textit{essential} if it does not bound a disk or once-punctured disk in $\Sigma$. A \textit{c-disk} is an embedded disk in $H$ whose boundary is an essential simple closed curve on $\Sigma$ and whose interior intersects $T$ at most once. If a c-disk does not intersect the tangle we call it a \textit{compressing disk}, and if it intersects the tangle once we call it a \textit{cut disk}. We will sometimes refer to the boundary of a compressing (resp. cut) disk as a compressing curve (resp. cut curve).

We are interested in subcomplexes spanned by curves that bound disks in trivial tangles. We say that a vertex $v$ in the dual curve graph or the pants graph belongs to the \textit{disk set} $\mathcal{D}(H,T)$ if each curve in the pants decomposition bounds a c-disk for $(H,T)$. 

\subsection{Multisections} \label{subsec:multisections} 

Trisections, introduced in 2016 by Gay and Kirby, are decompositions of 4-manifolds into three four-dimensional handlebodies. 
In 2020, Islambouli and Naylor generalized this idea to multisections, where the number of pieces is arbitrary.

\begin{definition}[\cite{gay2016trisecting,islambouli2020multisections}]\label{def:multisection}
Let $X$ be a closed, orientable $4$-manifold. A \textit{multisection} of $X$ is a decomposition $X = X_1 \cup X_2 \cup \cdots \cup X_n$ such that for each $i,j\in\{1,\dots, n\}$, 
\begin{enumerate}
\item $X_i\cong \natural^{k_i}(S^1\times B^3)$,
\item $X_1\cap \cdots \cap X_n=\Sigma$ is a genus $g$ surface, 
\item $X_i \cap X_{i+1}=H_i$ is a three-dimensional handlebody, and
\item $X_i\cap X_j=\Sigma$ whenever $|i-j|>1$.
\end{enumerate}
\end{definition}

The indices in the above definition are treated modulo $n$, so $X_n\cap X_1=H_n$ is also a three-dimensional handlebody. Sometimes we will say \textit{$n$-section} to denote a multisection with $n$ sectors, \textit{trisection} for a multisection with $n=3$, and \textit{quadrisection} for a multisection with $n=4$. We will use $\mathcal{T}$ to refer to multisections, including trisections.

The data $(g,k)$ associated with a multisection $\Tcal$,  where $k$ denotes the vector $(k_1,\dots, k_n)$, is the simplest measure of complexity of $X$. For example, it recovers the Euler characteristic: \[\chi(X)=2+(n-2)g-\sum_i k_i.\]
It follows from the definition that the boundary of each 4-dimensional piece admits a Heegaard splitting $\partial X_i=H_{i-1} \cup_
\Sigma H_{i}$. 
A \textit{multisection diagram} is an ordered collection $(\Sigma; C_1,\dots, C_n)$ where $\Sigma$ is a genus $g$ surface,
$C_1,\dots, C_n$ are cut systems for $\Sigma$, and each triple $(\Sigma; C_{i-1}, C_{i})$ determines a Heegaard splitting for $\#^{k_i}(S^1\times S^2)$.

Meier and Zupan showed that knotted surfaces in trisected $4$-manifolds can be isotoped in such a way that they inherit their own trisection \cite{meier2017bridgeS4,meier2018bridge4manifolds}. This notion was extended to multisections by Islambouli, Karimi, Lambert-Cole, and Meier. 

\begin{definition}[\cite{meier2018bridge4manifolds,islambouli2020multisections}] Let $F$ be an embedded surface in a closed 4-manifold $X$. 
A \textit{$(g,k;b,c)$-bridge multisection} of $(X,F)$, where $k=(k_i)_{i=1}^n$ and $c=(c_i)_{i=1}^n$, is a decomposition \[(X,F)=(X_1,D_1)\cup (X_2,D_2)\cup\cdots \cup (X_n,D_n)\] 
such that $X=\cup^{n}_{i=1} X_i$ is a $(g,k)$-multisection of $X$ satisfying, for each $i,j\in \{1,\dots, n\}$, 
\begin{enumerate}
\item $X_i\cap F=D_i$ is a collection of $c_i$ boundary parallel disks in $X_i$, and 
\item $H_i\cap F=T_i$ is a $b$-strand trivial tangle in $H_i$. 
\end{enumerate}
\end{definition}

The common intersection $\cap_i X_i=\Sigma$ intersects $F$ in $2b$ points. When $\Tcal$ is a bridge multisection for $(X,F)$, we take $\Sigma$ as a $2b$-punctured surface. By definition, $L_i=F\cap \partial X_i$ is $c_i$-component unlink in $(g,b)$-bridge position $(\partial X_i,L_i)=(H_{i-1},T_{i-1})\cup_{\Sigma} (H_{i},T_{i})$. 
As each tangle $(H_i,T_i)$ is trivial, we can find a collection of embedded shadows $s_i\subset \Sigma$ determining the tangle $T_i$. With this new data, we can describe a bridge multisection using a 2-dimensional diagram. 
A \textit{multisection shadow diagram} of $\Tcal$ is a tuple $(\Sigma;(C_i)_i, (s_i)_i)$ where $C_i$ is a cut system for $H_i$ and $s_i$ is a collection of shadows for $T_i$ satisfying $s_i\cap C_i=\emptyset$. It follows that $(\Sigma; C_{i-1}\cup s_{i-1}, C_i\cup s_i)$ represents a $c_i$-component unlink in $\#^{k_i}S^1\times S^2$. 

\begin{remark}
For expository purposes, we will present the arguments in this paper for $\Tcal$ a bridge multisection of a pair $(X,F)$ where $F$ is an embedded surface in $X$. If $F$ is empty, then $\Tcal$ is a multisection of $X$ as in \autoref{def:multisection}. 
\end{remark}

\subsection{New from old: c-reducible multisections}\label{subsec:new_from_old}
Let $(X_1,F_1)$ and $(X_2,F_2)$ be two (4-manifold, surface) pairs. Fix a $(g_i,b_i)$-bridge multisection $\Tcal_i$ for $(X_i,F_i)$ with central surface $\Sigma_i$, for $i=1,2$. 
We describe ways to build new pairs $(X,F)$ from this given data. 
\vspace{1em}

\textbf{Distant sum of two pairs.} This is obtained by taking the connected sum of $X_1$ and $X_2$ along 4-balls disjoint from $F_1$ and $F_2$. One can obtain a multisection for this by choosing the 4-ball to be a small neighborhood of a point in $\Sigma_i$ disjoint from the punctures. The complexities of this new multisection satisfy $g=g_1+g_2$ and $b=b_1+b_2$. 
\vspace{1em}

\textbf{Connected sum of two pairs.} This is obtained when the 4-balls in $X_i$ intersect each $F_i$ in a small disk. The new surface is $F=F_1\#F_2$. A multisection is obtained by choosing the 4-balls to be neighborhoods of punctures in $\Sigma_i$. The new genus and the bridge numbers are $g=g_1+g_2$ and $b=b_1+b_2-1$. 
\vspace{1em}

\textbf{Connected sum with $(S^1\times S^3, \emptyset)$.} Take two points $p,q$ in $X_1$ away from $F_1$. One can obtain the connected sum $(X_1,F_1)\#(S^1\times S^3,\emptyset)$ by identifying small neighborhoods of $p$ and $q$. A multisection for such pair can be obtained by choosing $p,q$ in $\Sigma$ away from the punctures. This is equivalent to taking the connected sum of $\Tcal_1$ with a genus one multisection for $(S^1\times S^3,\emptyset)$.
\vspace{1em}

\textbf{Self-tubing of $(X_1,F_1)$}. Take two 4-ball neighborhoods of two points $\{p,q\}$ in $F_1$. The self-tubing of $(X_1,F_1)$ is obtained by identifying these neighborhoods. The new 4-manifold is $X\# S^1\times S^3$ and the new surface is the result of adding a product tube to $F_1$. More precisely, if $U$ is an unknot in $S^3$, $(X,F)=(X_1,F_1)\#(S^1\times S^3,S^1\times U)$. In order to obtain a multisection for $(X,F)$, choose $p$ and $q$ to be two punctures of $\Sigma_1$ with the property for each $i$, $p$ and $q$ belong to different arcs of the tangle $T_i$. One can check that $g=g_1+1$ and $b=b_1-1$. 
\vspace{1em}

The multisections built above have curves bounding c-disks in every tangle. Such multisections are called \textit{c-reducible}, and such curves are called \textit{c-reducible}.
By definition, c-reducible curves induce 2-dimensional spheres in $\partial X_i= H_{i-1}\cup H_i$ intersecting $F\cap \partial X_i$ in at most two points. In particular, since $F\cap \partial X_i$ is an unlink, we conclude the following. 

\begin{lemma}
Let $x$ be a c-reducing curve for $\Tcal$. Then $x$ bounds the same kind of c-disk on every tangle $(H_i,T_i)$; that is, either all compressing or all cut disks. 
\end{lemma}

There are four kinds of c-reducible curves $x\subset \Sigma$. 
If $x$ bounds compressing disks and separates $\Sigma$, then $\mc T$ is the multisection of the distant sum of two pairs. 
If $x$ bounds separating cut disks, then $\mc T$ is the multisection of a connected sum of two pairs. 
If $x$ bounds compressing disk and does not separate, then $\mc T$ is the connected sum of a smaller multisection with a genus one multisection for $(S^1\times S^3, \emptyset)$. 
If $x$ bounds non-separating cut disks, then $\mc T$ is the self-tubing of a smaller multisection.

As its name suggests, c-reducible multisections can be simplified. Given a c-reducing curve for $\Tcal$, a \textit{c-reduction} is the process of filling $\overline{\Sigma-N(x)}$ with $m$-punctured disks \hbox{($m\leq 1$)} to obtain one or two multisections, depending on whether $x$ separates $\Sigma$ or not, and whether $x$ cuts or compresses. A multisection $\Tcal$ is \textit{completely decomposable} if it can be c-reduced into a collection of multisections of complexities $(g,b)\in \{(0,1),(0,2),(1,0)\}$. In other words, completely decomposable multisections are c-connected sums and self-tubings of bridge multisections with the lowest complexities.

\begin{lemma}\label{lem:completely_decomposable}
Let $\Tcal$ be a completely decomposable bridge multisection for $(X,F)$. Then 
\[ (X,F)=(S^4,F')\#^l(S^1\times S^3,S^1\times U)\#\left(\#^hS^1\times S^3\#^iS^2\times S^2\#^j\mathbb{CP}^2\#^k\overline{\mathbb{CP}^2},\emptyset\right),
\]
where $F'$ is an unknotted surface, $U$ is an unknot in $S^3$, and $l,h,i,j,k\geq 0$. 
In particular, if $X$ has no $S^1\times S^3$ summands, then $F$ is an unknotted surface in $X\cong\#^iS^2\times S^2\#^j\mathbb{CP}^2\#^k\overline{\mathbb{CP}^2}$.
\end{lemma}

\begin{proof}
In \cite{islambouli2022toric}, the authors proved that closed 4-manifolds having genus one multisections have such form. One can repeat their argument to check that bridge multisections with $(g,b)=(0,2)$ represent unknotted surfaces in $S^4$. The result follows. 
\end{proof}

We end this section by discussing two situations when the multisection is c-reducible. 

\begin{lemma}\label{lem:stabilization}
Let $\mathcal{T}$ be a bridge multisection with central surface $\Sigma$. Let $x,y\subset \Sigma$ be simple closed curves with $|x\cap y|=1$. 
Suppose that 
for every $i\in \{1,\dots, n\}$ one of $\{x,y\}$ bounds a compressing disk in $(H_i,T_i)$. 
Then $\Tcal$ is c-reducible.
\end{lemma}

\begin{proof}
The boundary of the neighborhood of $x\cup y$ is a c-reducing curve for $\Tcal$. 
\end{proof}

\begin{lemma}\label{lem:(12)red_implies_cred}
Let $\Tcal$ be a bridge multisection with central surface $\Sigma$. Let $x\subset \Sigma$ be a simple closed curve bounding a compressing disk for $(H_i,T_i)$ for $i=1,\dots, n-1$. Suppose $x$ bounds a disk in the handlebody $H_n$ that intersects the tangle $T_n$ in two interior points belonging to distinct connected components of $T_n$. Then $\Tcal$ is c-reducible.
\end{lemma}

\begin{proof}
By assumption, $x$ bounds a compressing disk in each handlebody $H_i$. Thus, if we let $x$ pass through some punctures of $\Sigma$ we will still get a disk in each handlebody $H_i$. In what follows, we will do this to ensure that such disks intersect each tangle $T_i$ once. By assumption, there exist shadows $s_j\subset \Sigma$ for $T_j$ satisfying the property that $x\cap s_i$ is empty for $i=1,\dots, n-1$ and two points for $i=n$. We depict this in \autoref{fig:almost_reducible}(a). Let $a\subset s_n$ be an arc connecting a point in $x\cap s_n$ with a puncture of $\Sigma$. Given that the points in $x\cap s_n$ belong to distinct arcs of $T_n$, we obtain that the arc $a$ is properly embedded in $\overline{\Sigma-x}$. If the interior of $a$ intersects other shadows $s_i$, we can slide such shadows over $x$ and remove the intersections as in \autoref{fig:almost_reducible}(b). To end, use the arc $a$ to push $x$ through a puncture in $\Sigma$ as in \autoref{fig:almost_reducible}(c). This will create a new curve $x'$ intersecting every shadow $s_i$ in one point around such puncture. Hence, $x'$ is a c-reducing curve for $\Tcal$.
\end{proof}

\begin{figure}[h]
\centering
\includegraphics[width=14cm]{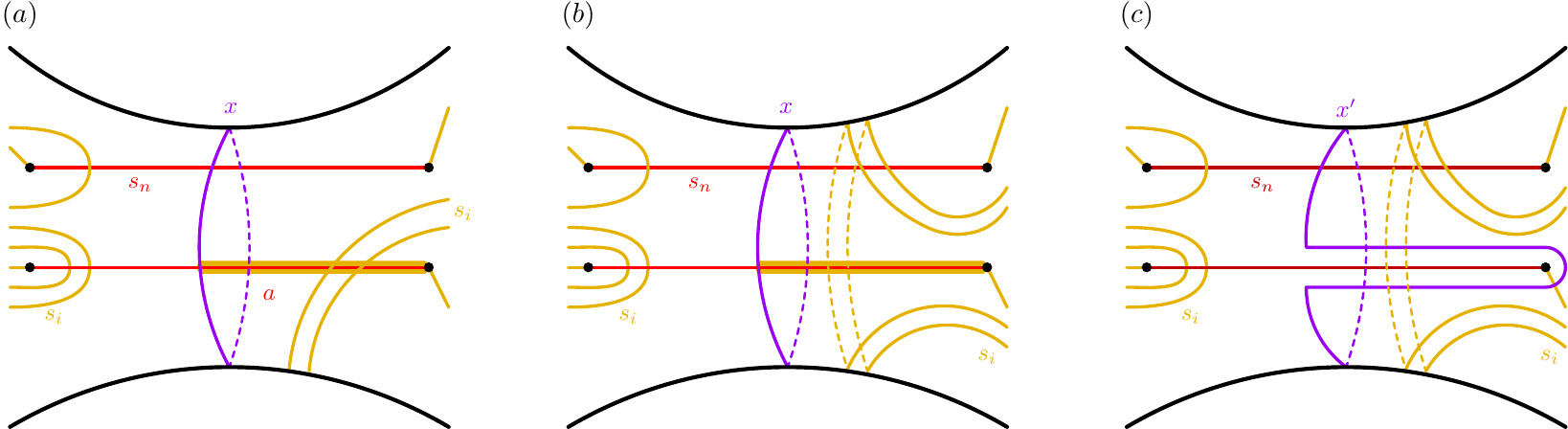}
\caption{How to alter $x$ to find a c-reducing curve.}
\label{fig:almost_reducible}
\end{figure}

\section{New invariants of pairs \texorpdfstring{$(X,F)$}{(X,F)}} \label{sec:invariants}

In this paper, we use the graphs $\mc{C}^*(\Sigma)$ and $\Pcal(\Sigma)$ (see \autoref{subsec:complexes}) to study knotted surfaces in $4$-manifolds. 
This approach can be thought of as a Hempel distance for pairs of pants. In \cite{johnson2006heegaard}, Johnson used it to study Heegaard splittings of irreducible 3-manifolds. Now, we define a similar invariant in four dimensions. 

Let $\Tcal$ be a bridge multisection for a pair $(X,F)$. The disk sets of the tangles $(H_i,T_i)$, denoted by $\Dcal_i$, are subsets of both $\mc{C}^*(\Sigma)$ and $\Pcal(\Sigma)$. A fact we will show in \autoref{sec:efficient} is that the distance between the disk sets of consecutive tangles is the same if computed in $\mc{C}^*(\Sigma)$ or $\Pcal(\Sigma)$; that is, $d^*\left(\Dcal_i, \Dcal_{i+1}\right)=d^{\Pcal}\left(\Dcal_i, \Dcal_{i+1}\right)$, $\forall i=1,\dots, n$. 
With this in mind, a pair of vertices $(u,v)\in \Dcal_i\times \Dcal_{i+1}$ is called an \textit{efficient defining pair} for $(\Dcal_i, \Dcal_{i+1})$ if it achieves the (minimum) distance between $\Dcal_i$ and $\Dcal_{i+1}$. See \autoref{fig:ugly} for an example and non-example of efficient defining pairs. 
We are ready to introduce two notions of distance for $\Tcal$. 
\begin{figure}[h]
\centering
\includegraphics[width=1\textwidth]{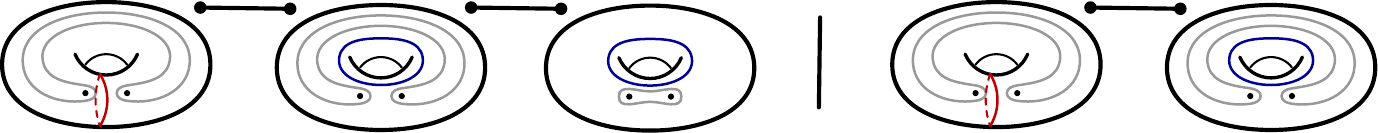}
\caption{On the left, the initial and terminal vertices \textit{do not} form an efficient defining pair. On the right, the initial and terminal vertices \textit{do} form an efficient defining pair.}
\label{fig:ugly}
\end{figure}

\begin{definition}\label{def: L-inv}
The \textit{$\Lcal_n^*$-invariant} of an $n$-section $\Tcal$ is equal to the smallest sum 
$\sum_i d^*(P'_i,P_{i+1})$  
amongst all efficient pairs $(P_i,P'_i)$ for $(\Dcal_i,\Dcal_{i+1})$. 
Analogously, the \textit{$\Lcal_n^\Pcal$-invariant} of an \hbox{$n$-section} $\Tcal$ is defined to be the minimum of 
$\sum_i d^\Pcal \left( P'_i, P_{i+1}\right)$ 
amongst all efficient defining pairs $(P_i,P'_i)$ for $(\Dcal_i, \Dcal_{i+1})$. 
\end{definition}

It is useful to visualize each disk set $\Dcal_i$ as a cloud in the respective complex of curves. This invariant then measures the minimal distance within the clouds, over all diagrams with clouds separated by efficient pairs.
See \autoref{fig:islands} for a graphic representation. In the proofs, we will also often refer to $\Dcal_i$ as an \textit{island}.

\begin{figure}[h]
\labellist 
\pinlabel {$\Dcal_1$}  at 105 240
\pinlabel {$\Dcal_3$}  at 500 240
\pinlabel {$\Dcal_4$}  at 300 40  
\pinlabel {$\Dcal_2$}  at 300 440

\pinlabel {$P_1$}  at 180 265
\pinlabel {$P'_4$}  at 180 210

\pinlabel {$P'_2$}  at 405 260
\pinlabel {$P_3$}  at 405 210

\pinlabel {$P'_3$}  at 355 80 
\pinlabel {$P_4$}  at 235 80 

\pinlabel {$P'_1$}  at 240 390
\pinlabel {$P_2$}  at 330 390

\endlabellist
\centering
\includegraphics[width=7cm]{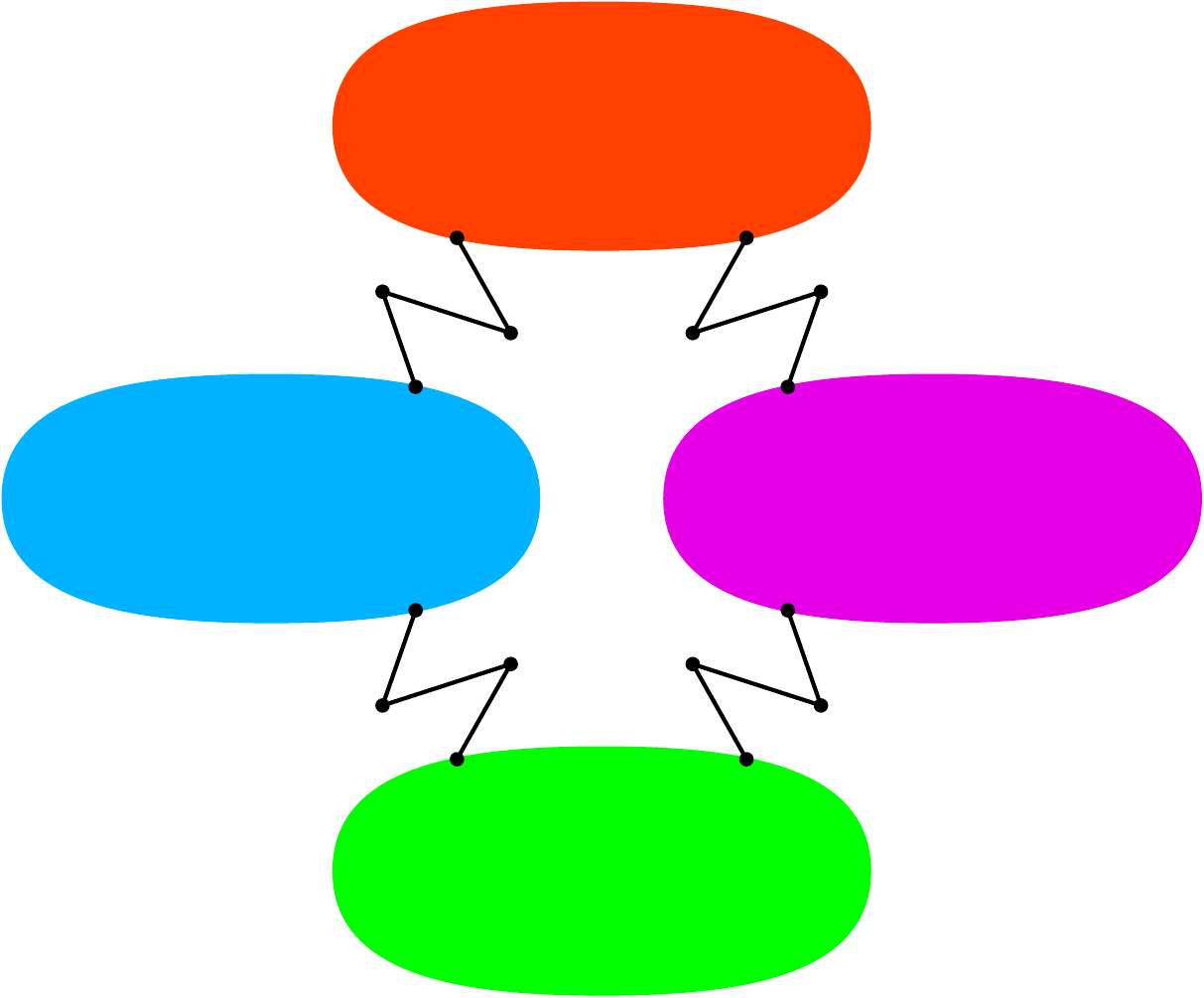}
\caption{$\Lcal_n^\Pcal(\Tcal)$ and $\Lcal_n^*(\Tcal)$ count the number of edges between vertices in the disk sets.}
\label{fig:islands}
\end{figure}

\begin{definition}
Let $F$ be a surface embedded in a closed 4-manifold $X$. The \textit{$\Lcal_n^*$-invariant} of the pair $(X,F)$ is the smallest $\Lcal_n^*(\Tcal)$ amongst all $n$-sections for $(X,F)$ having minimal complexity $(g,b)$. 
The \textit{$\Lcal^\Pcal_n$-invariant} of $(X,F)$ is defined the same way. 
\end{definition}
We now demonstrate these definitions in detail with an example.
\begin{example}[Trisection of $\mathcal{C}_1$ and $\mathcal{C}_2$ in $\mathbb{CP}^2$]\label{exam:deg12_cp2}
Let $\mathcal{C}_d$ denote the degree $d$ complex curve in the complex projective plane. For $d=1,2$, $\mc{C}_d$ admits c-irreducible bridge trisections with complexities $(g,k;b,c)=(1,0;1,1)$. The red, blue, and green curves in \autoref{fig:c1paths} and \autoref{fig:c2paths} on the torus with two marked points combine to make doubly pointed bridge trisection diagrams for $\mc{C}_1$ and $\mc{C}_2,$ respectively. We now provide paths in the dual curve complex and the pants complex, giving the upper estimate for $\mathcal{L}^*_3$ and $\mathcal{L}^\Pcal_3$, respectively. Readers are encouraged to consult \autoref{fig:c1paths} and \autoref{fig:c2paths}. There are two main steps.

\textbf{Step 1}: Each $(X_i,D_i)\cap (X_{i+1},D_{i+1})$ for $i=1,2,3 \pmod 3$ is a $(1,1)$-bridge position of an unknot in $S^3$. For each trivial tangle in $H_i$ choose two pants decompositions $P_i$ and $P_{i+2}'$, where $P_i$ is one S-move away from $P_i'$. This is possible due to \autoref{lem:efficient}.

\textbf{Step 2}: Find a geodesic from $P_1'$ to $P_{2}$, from $P_2'$ to $P_{3}$ and from $P_3'$ to $P_1$. Note that each gray curve that moves once in each disk set is a reducing curve (that is, bounds c-disks in two trivial tangles). 
Therefore, $\mathcal{L}_3^*(\mathbb{CP}^2,\mathcal{C}_1)\leq 3$ and $\mathcal{L}_3^*(\mathbb{CP}^2,\mathcal{C}_2)\leq 3$. 
The lower bound in \autoref{thm:lower_bound_1} also gives that $\Lcal_3^*\geq 3$. Hence, $\mathcal{L}_3^*(\mathbb{CP}^2,\mathcal{C}_i)= 3$ for $i=1,2$.

Observe that when a curve bounding a compressing disk moves inside each colored disk set, the new curve bounding a compressing disk intersects the old curve four times, so these paths do not estimate $\Lcal_3^{\mc P}$. 

To compute the pants version of our complexity, we have to find an intermediate curve between $P_i'$ and $P_{i+1}$. That is, the separating curve always moves twice in each disk set. The intermediate curve can be taken to be parallel to the non-separating curve so that the annulus co-bounded by the parallel copies contains a puncture (see the doubly-pointed diagrams on the right sides of \autoref{fig:c1paths} and \autoref{fig:c2paths}). The lower bound for the latter statement comes from an observation that an A-move in the pants complex occurring in a twice-punctured annulus will connect two curves that intersect at least four times. Therefore, $\mathcal{L}_3^{\mathcal{P}}(\mathbb{CP}^2,\mathcal{C}_i)=6$ for $i=1,2$.

\label{example:c1andc2incp2}
\end{example}

\begin{figure}[ht!]
\centering
\includegraphics[width=8cm]{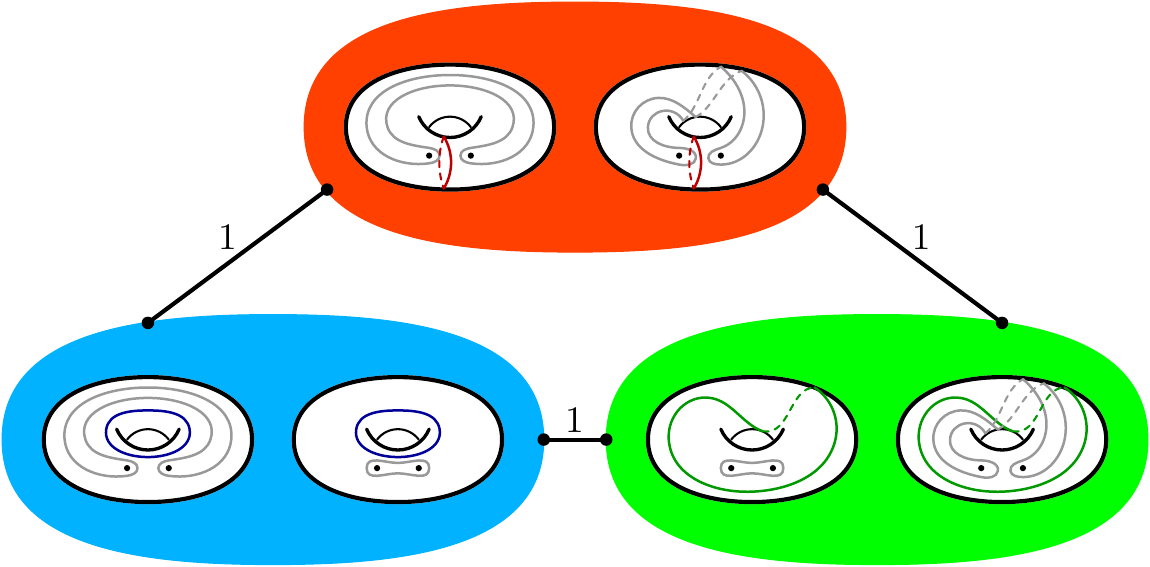}
\includegraphics[width=8cm]{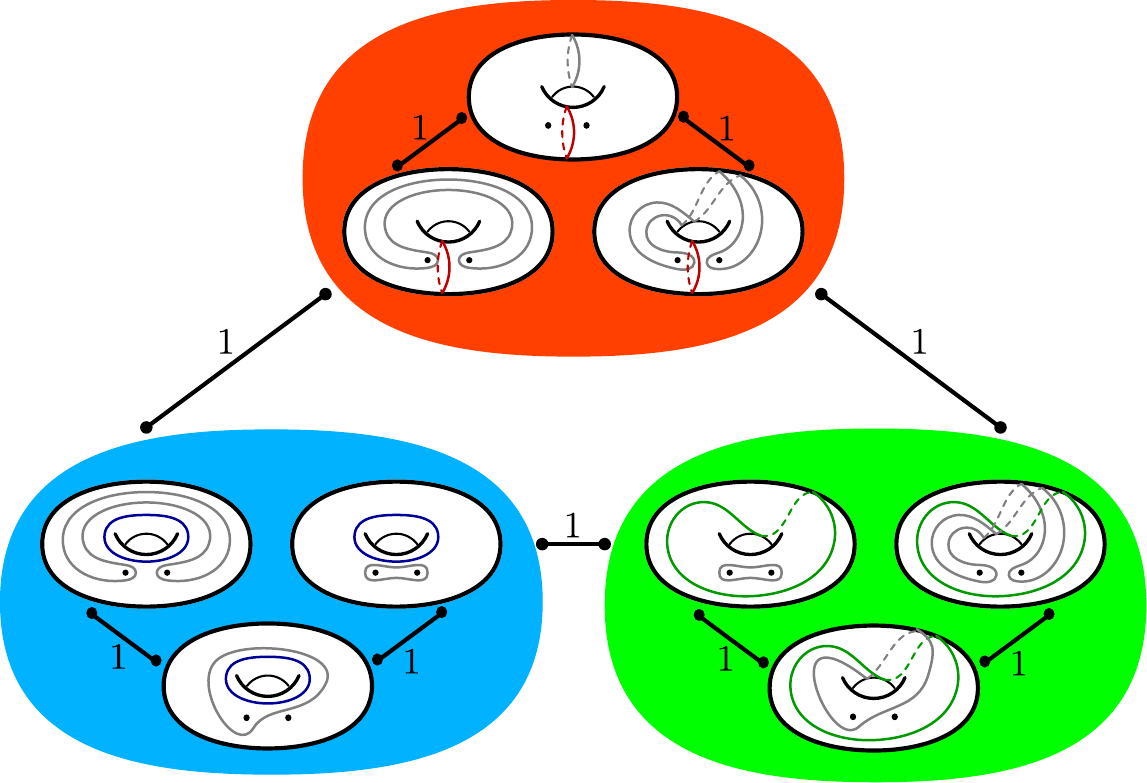} 
\caption{Loops in $\mc{C}^*(\Sigma_{1,2})$ estimating $\Lcal_3^*(\mathbb{CP}^2,\mc{C}_1)$ (left) and $\Lcal_3^{\Pcal}(\mathbb{CP}^2,\mc{C}_1)$ (right).}
\label{fig:c1paths}
\end{figure}

\begin{figure}[ht!]
\centering
\includegraphics[width=8cm]{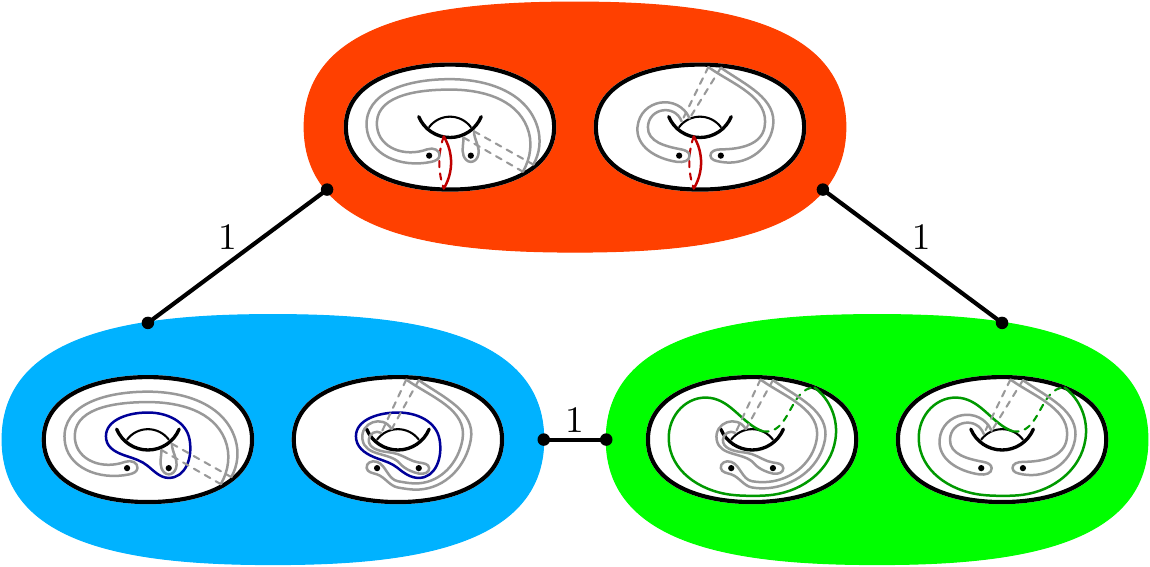} 
\includegraphics[width=8cm]{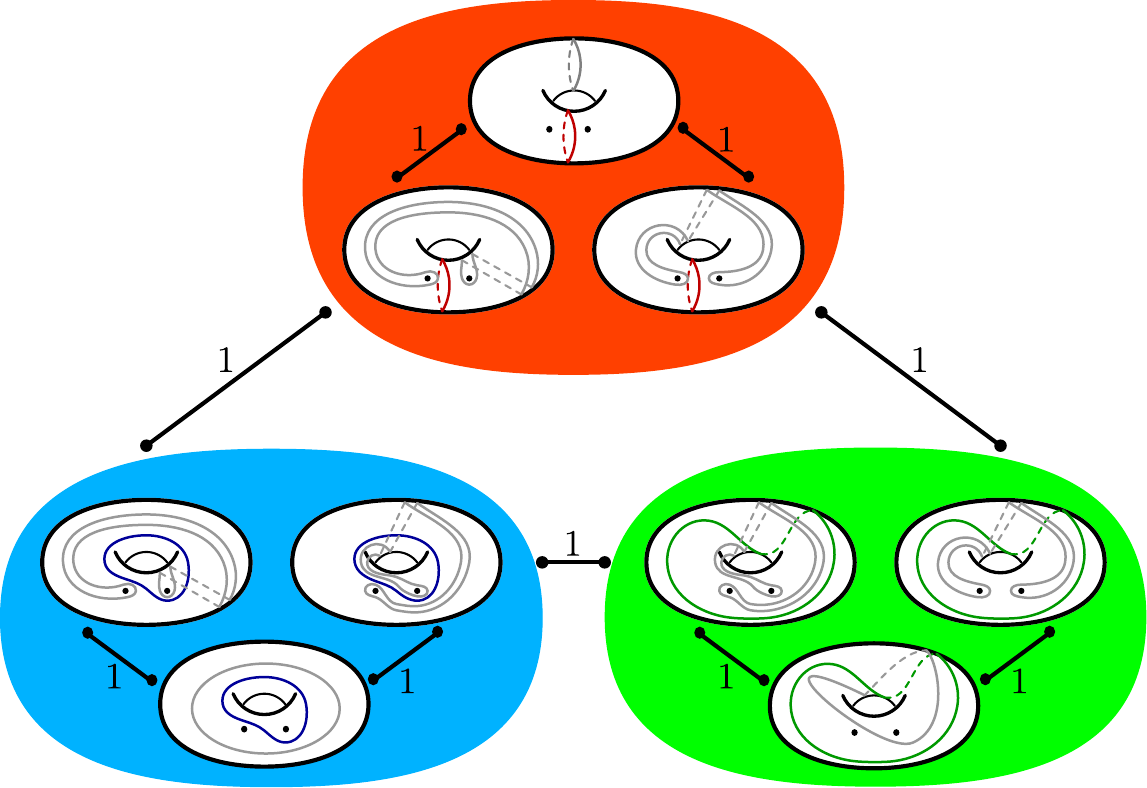} 
\caption{Loops in $\mc{C}^*(\Sigma_{1,2})$ estimating $\Lcal_3^*(\mathbb{CP}^2,\mc{C}_2)$ (left) and $\Lcal_3^{\Pcal}(\mathbb{CP}^2,\mc{C}_2)$ (right).}
\label{fig:c2paths}
\end{figure}

For a bridge trisection of a knotted surface $F$ in $S^4$, the invariant $\mathcal{L}_3^{\mathcal{P}}(S^4,F)$ coincides with $\mathcal{L}(F)$ introduced in \cite{blair2020kirby} and $\mathcal{L}^*_3(S^4,F)$ coincides with $\mathcal{L}^*(F)$ introduced in \cite{aranda2022bounds}. If $F$ is empty, $\Lcal_n^*(X)=\Lcal^*_n(X,\emptyset)$ is an invariant for closed 4-manifolds. 
A summary of the relevant notation for the definition of our invariants can be found in \autoref{tab:notation}. 

\begin{table}[h]
\centering
\begin{tabular}{|rcl|rcl|}
\hline
\multicolumn{3}{|c|}{\textbf{Dual curve complex} $\mc{C}^*(\Sigma)$} & \multicolumn{3}{|c|}{\textbf{Pants complex} $\Pcal(\Sigma)$}  \T\B  \\
\hline
$d^*(-,-)$ & $=$ & combinatorial distance & $d^{\mathcal{P}}(-,-)$ & $=$ & combinatorial distance \T 
\B \\
$\Dcal_i$ & $=$ & disk set of $(H_i,T_i$) & $\Dcal_i$ & $=$ & disk set of $(H_i,T_i)$ \T \B \\
$(P_i,P'_{i})$ & $=$ & efficient pair for $(\Dcal_i,\Dcal_{i+1})$ & $(P_i,P'_{i})$ & $=$ & efficient pair for $(\Dcal_i,\Dcal_{i+1})$ \T 
\B \\
$\mathcal{L}_n^*(\mathcal{T})$ & $=$ & $\min\left \lbrace \sum_{i} d^*(P'_i,P_{i+1}) \right\rbrace$ & $\mathcal{L}_n^{\mathcal{P}}(\mathcal{T})$ & $=$ & $\min\left\lbrace\sum_{i} d^{\mathcal{P}}(P'_i,P_{i+1})) \right\rbrace$ \T \\
& & over all efficient pairs & & & over all efficient pairs \B \\
$\mathcal{L}_n^*(X,F)$ & $=$ & $\min\left\lbrace \mathcal{L}_n^*(\mathcal{T})\right\rbrace$ over all& $\mathcal{L}_n^{\mathcal{P}}(X,F)$ & $=$ & $\min\left\lbrace \mathcal{L}_n^{\mathcal{P}}(\mathcal{T}) \right\rbrace$ over all\T \\
& & minimal $(g,b)$ $n$-sections& & &minimal $(g,b)$ $n$-sections 
\B \\
\hline
\end{tabular}
\caption{A summary of the definition of the invariants $\Lcal_n^*$ and $\Lcal_n^\Pcal$.
\label{tab:notation}
}
\end{table}

\begin{remark} \label{remark:compare}
We remark that $\Lcal_n^*\leq \Lcal_n^{\mc P}$. This is due to the fact that $\Pcal(\Sigma)$ is a subgraph of $\mc{C}^*(\Sigma)$ with the same vertex set. 
Therefore, a path in $\Pcal(\Sigma)$ gives rise to a path in $\mc{C}^*(\Sigma)$, but not necessarily a geodesic. Because of this, many of our results are stated (solely) in terms of the $\Lcal_n^*$ version.
\end{remark}

\subsection{Comparing \texorpdfstring{$\Lcal$}{L}-invariants}

Here we note some comparisons between various versions of the \texorpdfstring{$\Lcal$}{L}-invariant.

\subsubsection{\texorpdfstring{$\Lcal$}{L} and \texorpdfstring{$\Lcal^*_n$}{L*} are different invariants}

The $\Lcal_n^*$-invariant (resp. $\Lcal_n^{\Pcal}$) of a multisection with $n$ sectors is a function of the shortest length of a loop $\lambda$ in $\mc{C}^*(\Sigma)$ (resp. $\Pcal(\Sigma)$) passing through efficient defining pairs for $(\Dcal_i, \Dcal_{i+1})$ in a cyclic order $i=1,\dots, n$. 
This interpretation mimics the definition of the $\Lcal$-invariant introduced by Kirby and Thompson in \cite{kirby2018new} for trisections of closed 4-manifolds. We emphasize that our invariants are variations of the original $\Lcal$-invariant in the sense that they live in distinct complexes of curves -- the dual curve complex ($\Lcal^*_n$) and the pants complex ($\Lcal^\Pcal_n$) -- versus the cut complex ($\Lcal$). 
Regardless of these complexes, they all attempt to capture the complexity of a multisected pair $(X,F)$. 

The $\Lcal$- and $\Lcal^*_n$-invariants both count the number of edges between vertices in the disk sets, but the edges counted in the Kirby-Thompson invariant always lie in the disk sets, while the edges counted in our definition may go out of the disk sets, like the original Hempel distance of Heegaard splittings \cite{Hempel01}. Nevertheless, it is natural to ask whether the two versions are related (when $n=3$). The following two results show that $\Lcal$ and $\Lcal^*_3$ behave differently. 

\begin{corollary} \label{cor:ComparingKT1}
There is an infinite family of $4$-manifolds $\{X_{p}\}_{p\in \N}$ such that 
\[
\sup_p\Lcal_3^*(X_{p})<\infty, \quad \sup_p\Lcal(X_{p})=\infty.\]
\end{corollary}

\begin{proof}
Let $X_{p}$ be the spin of the lens space $L(p,1).$ By a result of Asano, Naoe, and Ogawa \cite{asano2023some}, $\mathcal{L}(X_{p})\underset{p}{\rightarrow} \infty$. By \autoref{example:trisection_spun_lens}, $\mathcal{L}_3^{*}(X_{p}) \leq 18$ for all $p$. 
\end{proof}

\begin{corollary} \label{cor:ComparingKT2}
There is an infinite family of $4$-manifolds $\{Y_{m}\}_{m\in \N}$ such that 
\[
\sup_m\Lcal_3^*(Y_{m})=\infty, \quad \sup_m\Lcal(Y_{m})=0.\]
\end{corollary}

\begin{proof}
Let $Y_{m}$ be the connected sum of $m$ copies of $S^2\times S^2$. By the work in \cite{kirby2018new}, $\Lcal(Y_m)=0$ for all $m$. The work in \autoref{subsec:S2xS2} implies that $\mathcal{L}_3^{*}(Y_m) \underset{m}{\rightarrow} \infty$.
\end{proof}

\subsubsection{\texorpdfstring{$\Lcal_n^*$}{L*}-invariants are unbounded}

In \autoref{sec:irred_multisections} we give lower bounds for the $\Lcal_n^*$-invariants of c-irreducible multisections in terms of the complexities $(g,k;b,c)$ (see \autoref{thm:lower_boundV1} and \autoref{thm:lower_bound_1}). As our lower bounds are increasing with respect to $(g,b)$, we can conclude that the invariants can be arbitrarily large. For instance, take a prime closed 4-manifold with high $n$-section genus. 

\begin{question}\label{question:L*_unbounded}
Let $n\geq 3$, and $g,b\geq 0$ be fixed. Is there a family of $n$-sections $\{\Tcal_{\ell}\}_{\ell\in I}$ with the same $(g,b)$-complexity satisfying $\sup_{\ell} \Lcal_n^*(\Tcal_{\ell})=\infty$?
\end{question}

We also present a version of \autoref{question:L*_unbounded} for $\Lcal^\Pcal_n$-invariants. The examples and conjectures in \autoref{sec:MiscExamples} may hint at a solution to the following problem. Recall from \autoref{remark:compare} that $\Lcal^*_n\leq \Lcal^\Pcal_n$. 

\begin{problem}\label{problem:large_difference}
    Fix $n\geq 3$. Find an infinite family $Z^n_q$ of closed, oriented $4$-manifolds such that the difference $\Lcal^\Pcal_n(Z^n_q) - \Lcal_n^*(Z^n_q)$ is unbounded as $q\to \infty$.
\end{problem}

\section{Various examples}\label{sec:MiscExamples}

In this section, we provide computations of the distance invariants for examples of surfaces in 4-manifolds. We think of this section as a series of applications of the general estimates proved later in this work (see \autoref{sec:lower_bounds} and \autoref{sec:genus_two_quadrisec}). The upper bounds are constructive. We begin computing $\Lcal_n^*$ and $\Lcal_n^\Pcal$ for multisections with small $(g,b)$-complexity.

\begin{proposition}\label{prop:(0,2)(1,0)_cases}
Let $\Tcal$ be a multisection with $(g,b)=(0,2),(1,0)$. Then $$\Lcal_n^*(\Tcal)=\Lcal_n^{\mc P}(\Tcal)=0.$$ \end{proposition}
\begin{proof}
It is well known that 2-string trivial tangles and genus one handlebodies have unique c-disks up to isotopy. Thus, the disk sets $\Dcal(H_i,T_i)$ are singletons. Hence, by definition of the $\Lcal_n^\Pcal$-invariant, $0\leq \Lcal_n^*(\Tcal)\leq \Lcal_n^{\mc P}(\Tcal)\leq 0$. 
\end{proof}

\subsection{Trisections}

\begin{example}[Trisection of $\mathcal{C}_{1,1}$ in $S^2\widetilde{\times}S^2$]

Consider the sphere $\mathcal{C}_{1,1}$ representing the homology class $(1,1)$ in $H_2(S^2\widetilde{\times}S^2,\Z)$. Such a surface admits a bridge trisection diagram depicted in \autoref{fig:c11path} with complexity $(g,k;b,c)=(2,0;1,1)$. The number of curves in a pants decomposition is $3g+2b-3=5.$ The curves that are not colored grey are the ones that move outside the disk sets and are not contributing to the count of $\mathcal{L}_3^*$. Among these grey curves, there is one curve $\lambda$ that stays fixed along the whole loop. In each disk set, two grey curves that are not $\lambda$ each moves once. Since there are three disk sets, we get the upper bound $\mathcal{L}_3^*(S^2\widetilde{\times}S^2,\mathcal{C}_{1,1})\leq 6.$ 

It is also true that $\mathcal{L}_3^*(\mathcal{C}_{1,1})=6$. To estimate the lower bound, let us focus on the depicted edge connecting red and blue handlebodies. By \autoref{lem:efficient}, there are two separating curves each cutting off a genus one summand bounding in red and blue handlebodies simultaneously. If these two curves stay stationary in the red disk set, then we get two grey curves $\lambda_1$ and $\lambda_2$ that bound in three disk sets simultaneously. This implies that the surface $\mathcal{C}_{1,1}$ is completely contained in a $\mathbb{CP}^2$ summand or a $\overline{\mathbb{CP}^2}$ summand. This cannot happen since $\mathcal{C}_{1,1}$ generates homology in both summands. Thus, $\lambda_1$ and $\lambda_2$ each moves at least once. Making the same argument for the red-green pair and the blue-green pair gives $\mathcal{L}_3^*(\mathcal{C}_{1,1})=6$.
\end{example}

\begin{figure}[h]
\includegraphics[width=11cm]{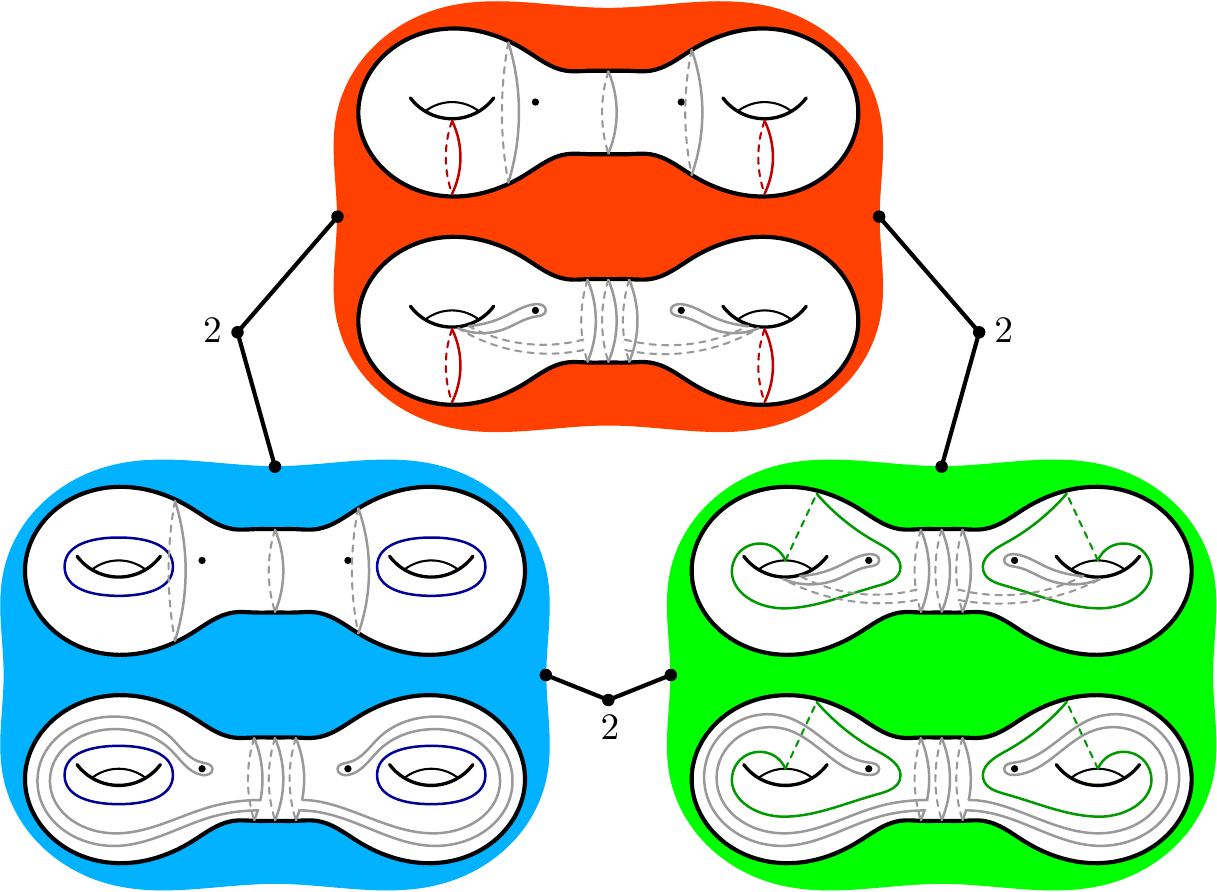}
\centering
\caption{A loop in the dual curve graph giving $\mathcal{L}_3^*(S^2\widetilde{\times} S^2,\mathcal{C}_{1,1}) \leq 6$. \label{fig:c11path}}
\end{figure}

\begin{example}\label{example:trisection_spun_lens}
Let $X_p$ be the 4-manifold obtained by spinning the lens space $L(p,q)$. In \cite{Meier}, Meier showed that $X_p$ admits genus three trisections depicted in \autoref{fig:trisection_spun_lens}(a). Using the path in \autoref{fig:trisection_spun_lens2} one can check that $\Lcal^*_3(X_p)\leq 6+6+6=18$. On the other hand, given that genus two trisections are standard \cite{MZgenustwo}, the genus three trisections for $X_p$ are c-irreducible. Thus, by \autoref{thm:lower_bound_1}, $\Lcal_3^*(X_p)\geq 12$. 
Motivated by the work in \cite{aranda2022bounds}, we conjecture that this lower bound can be improved. 

\begin{conjecture}
The $\Lcal^*_3$-invariant of spun lens spaces is equal to $18$. 
\end{conjecture}

\begin{figure}[h]
\centering
\includegraphics[width=12cm]{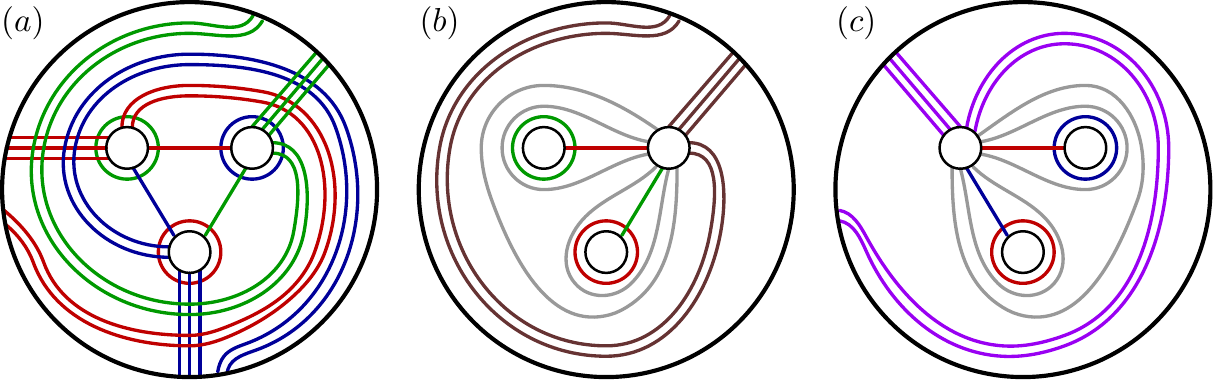}
\caption{(a) A trisection diagram for a spun lens space. (b) Red-green efficient defining pair. (c) Red-blue efficient defining pair.}
\label{fig:trisection_spun_lens}
\end{figure}

\begin{figure}[h]
\centering
\includegraphics[width=13cm]{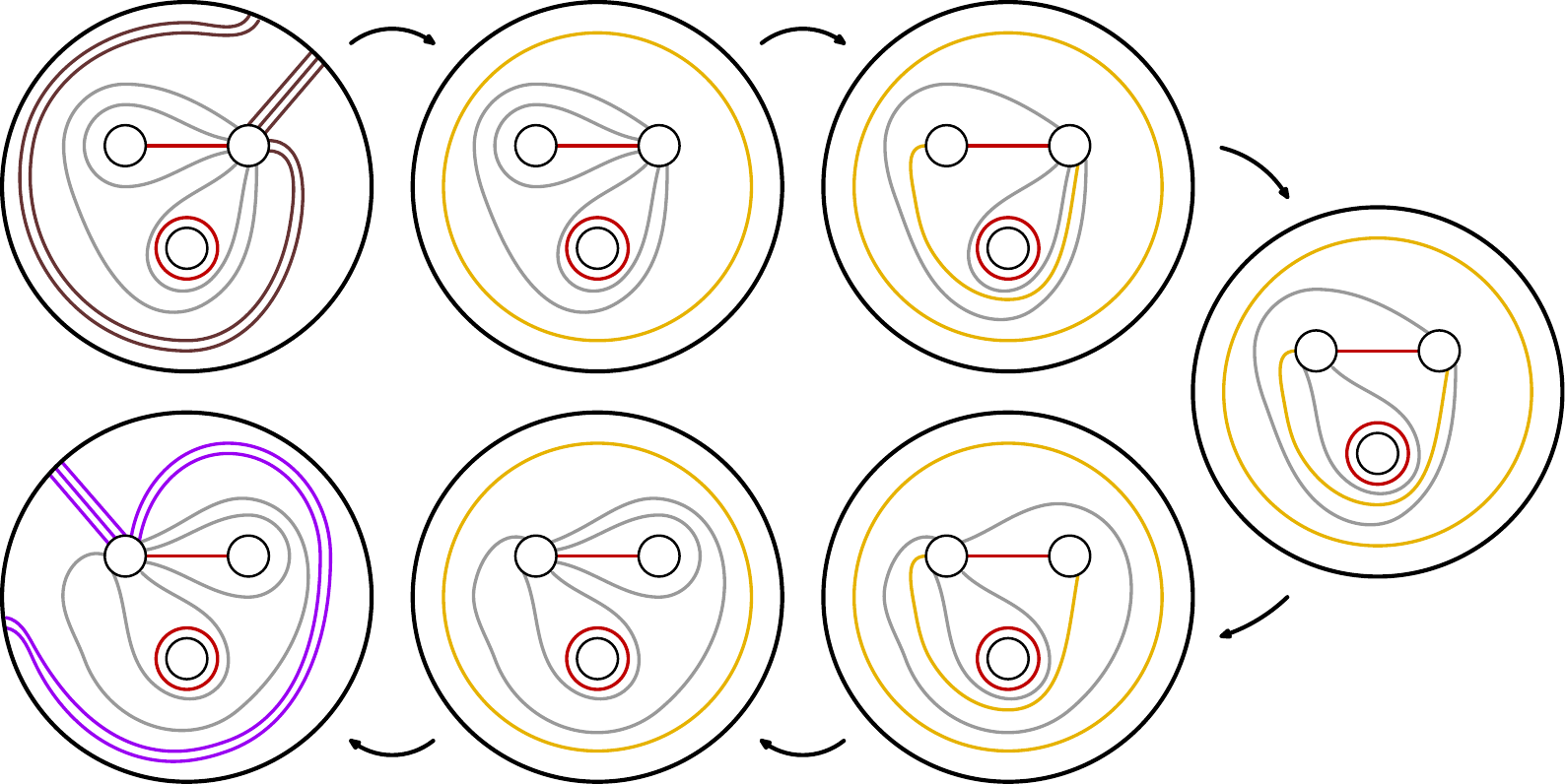}
\caption{A path in $\mc{C}^*(\Sigma_{3,0})$ between the pairs of pants in the red handlebody of \autoref{fig:trisection_spun_lens}.}
\label{fig:trisection_spun_lens2}
\end{figure}

\end{example}

\begin{example}[Trisections of $\#^m S^2\times S^2$]\label{subsec:S2xS2}
For $m>1$, let $Y_m$ be the connected sum of $m$ copies of $S^2\times S^2$. In what follows, we will see that $6m\leq \Lcal_3^*(Y_m)\leq 9m$. 

First, recall that $\mathcal{L}_3^*$ is defined over all minimal genus trisections of $Y_m$. Since the second Betti number is additive and gives a lower bound for the trisection genus \cite{ChuTillman}, we get that the trisection genus of $Y_m$ is equal to $2m$. \autoref{fig:s2xs21}(a) contains an example of a minimal genus trisection $\Tcal$ of $Y_m$. 
Such trisection is the connected sum of $m$ copies of the genus two trisection of $S^2\times S^2$. \autoref{fig:s2xs21}(b) showcases how to choose curves on each genus two summands of $\Tcal$ to obtain efficient defining pairs for $\Tcal$. The path in \autoref{fig:s2xs22} can be applied to each summand to obtain a path of length $3m$ in the blue disk set. One can find similar paths on the red and green islands and conclude that $\Lcal_3^*(\Tcal)\leq 3(3m)$.  

\begin{figure}[h]
\includegraphics[width=11cm]{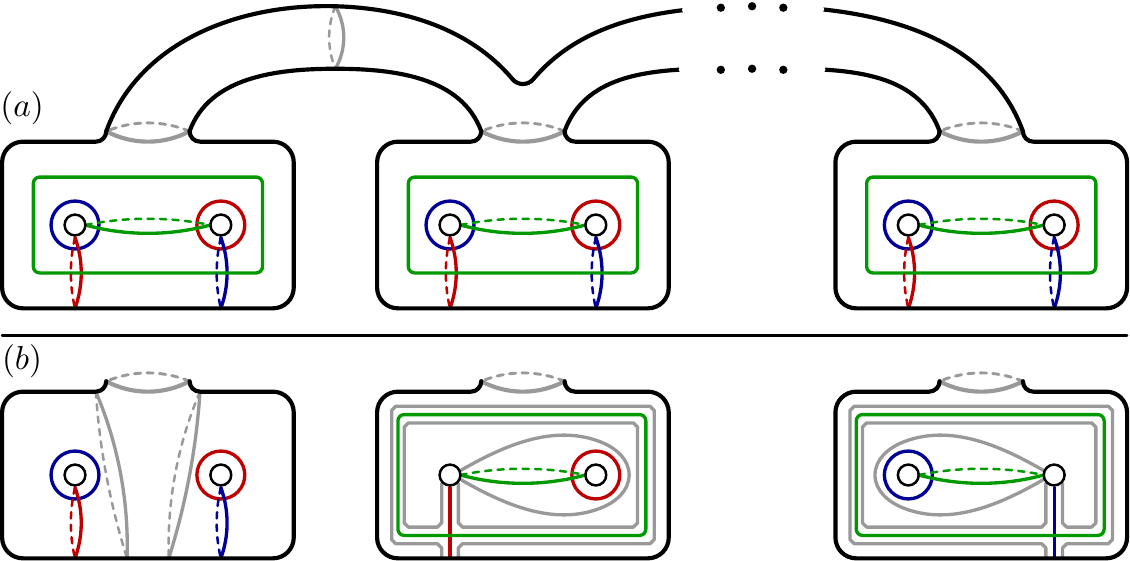}
\centering
\caption{(a) $(2m,0)$-trisection of $\#^m S^1\times S^2$. (b) Efficient defining pairs.}
\label{fig:s2xs21}
\end{figure}

\begin{figure}[h]
\includegraphics[width=9cm]{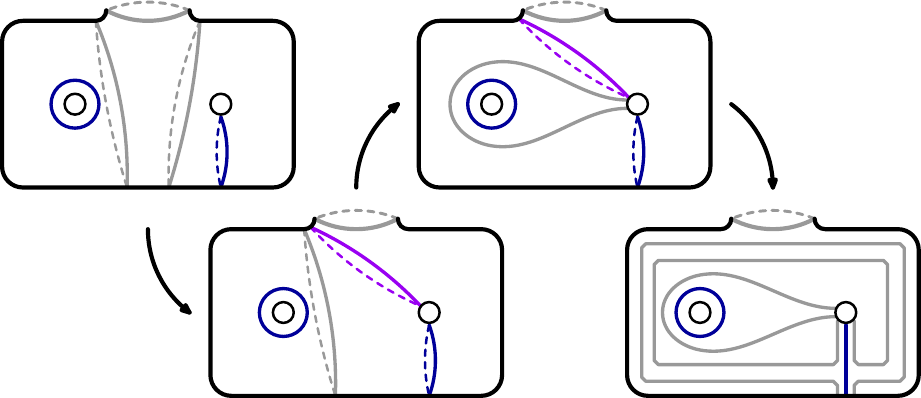}
\centering
\caption{A path in the blue disk set.}
\label{fig:s2xs22}
\end{figure}

Let $\Tcal$ be a minimal genus trisection for $Y_m$. 
By the above, the genus of $\Tcal$ is $2m$ and $k=0$. 
Thus, by \autoref{lem:efficient} there are $g-k=2m$ S-moves in any efficient pair. In each S-move occurring in a one-holed torus, we exchange a curve with a dual curve that intersects the original curve once. Each S-move then gives rise to a separating waist curve (the boundary of the one-holed torus). In conclusion, there are $2m$ such waist curves for a pants decomposition that is a vertex in a disk set of an efficient pair. Each such waist curve that stays stationary in a disk set gives rise to a reducing sphere decomposing our trisection as $\Sigma = \Sigma_1\#\Sigma_2$, where $\Sigma_1$ has genus 1. By the classification of genus one trisections, $Y_m = \#^m S^2\times S^2 = X \# \mathbb{CP}^2$ or $X \#\overline{\mathbb{CP}^2}$ for some 4-manifold $X.$ This cannot happen since $Y_m$ has an intersection form that is not equivalent to that of $X \#\mathbb{CP}^2$ and $X \#\overline{\mathbb{CP}^2}$. This argument can be repeated for any of the three disk sets of the trisection giving the distance of at least $2m$ in each disk set. Hence, $3(2m)\leq \Lcal_3^*(Y_m)$.

As some minimal genus trisections of $Y_m$ are reducible (see \autoref{fig:s2xs21}(a)), the lower bounds from \autoref{sec:lower_bounds} cannot be used to estimate $\Lcal_3^*(Y_m)$. A (hard) interesting problem is to give lower bounds for $\Lcal^*_n$-invariants of reducible manifolds.

\begin{conjecture}
The $\Lcal_3^*$-invariant of $Y_m$ is equal to $9m$. 
\end{conjecture}
\end{example}

\subsection{Quadrisections}

\begin{example}[Quadrisection of $S^2\times \lbrace \text{pt} \rbrace \subset S^2\times S^2$] \label{exmp:s2xpt}
Consider the surface $F=S^2\times \lbrace \text{pt} \rbrace \subset S^2\times S^2$. By work of \cite{islambouli2022toric}, this admits a multisection with $(g,b)=(1,1)$ depicted in \autoref{fig:01curve}. Since this quadrisection is c-irreducible, \autoref{thm:ogawaV1} implies that $\Lcal_4^*\geq 2$. On the other hand, the loop in \autoref{fig:01curve} shows that $\Lcal_4^*\leq 2$ and so $\Lcal_4^*(S^2\times S^2,S^2\times \{\text{pt}\})=2$.
\end{example}

\begin{figure}[h]
\centering
\includegraphics[width=9cm]{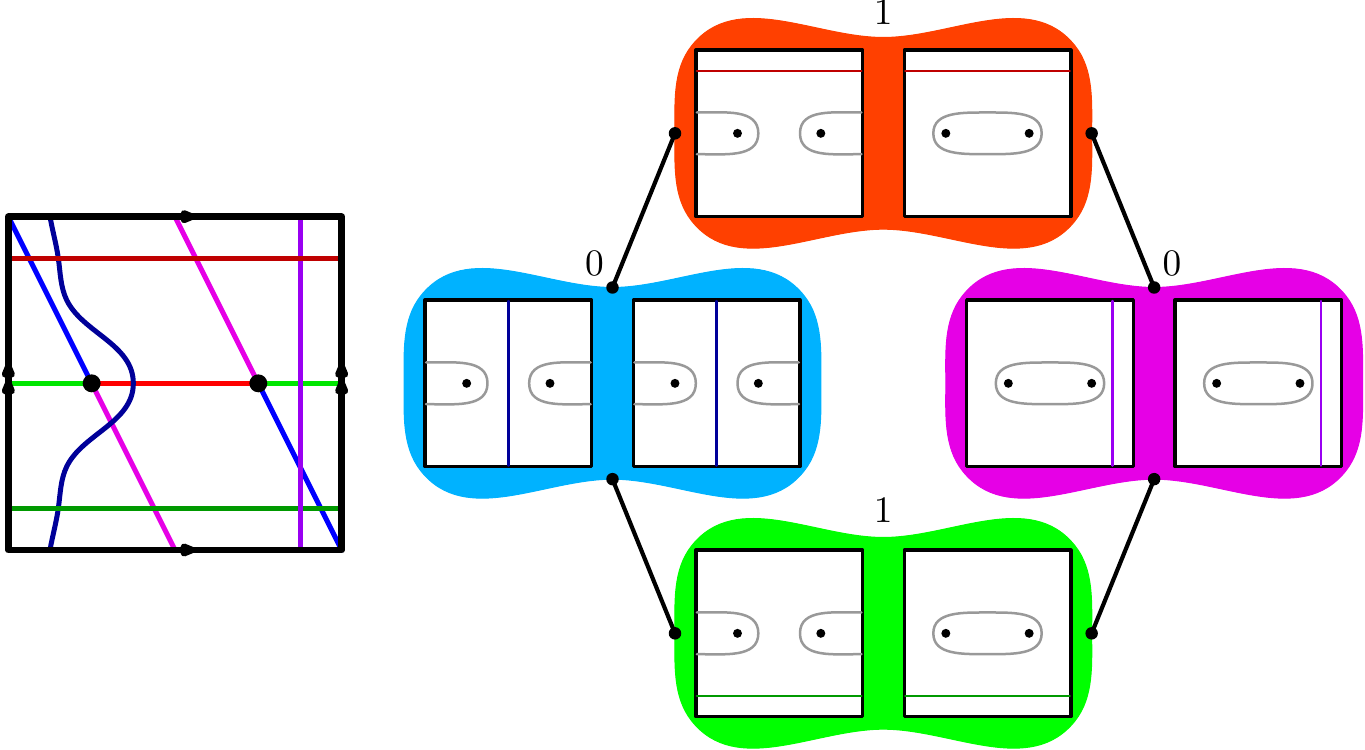}
\caption{A diagram for $(S^2\times S^2,S^2\times\{pt\})$ and a loop showing that $\mathcal{L}^*\leq 2$.}
\label{fig:01curve}
\end{figure}

\begin{example}[Quadrisection of the double of the standard ribbon disk for the knot $6_1$]
We denote the surface of interest by $F$ and by $X$ the 2-fold cover of $S^4$ branched along $F$.  \autoref{fig:steve} shows a path in $\mathcal{L}^*$ of length six for a quadrisection of $F$, thus $\Lcal^*_4(F)\leq 6$.
By \autoref{lem:L_under_covers}, $\mathcal{L}_4^*(X)\leq \mathcal{L}_4^*(F)$. After drawing a Kirby diagram for $X$ (see Figure 3 of \cite{owens2019knots}), we can see that $X$ has a finite non-trivial fundamental group. By the contrapositive of \autoref{thm:L>=6}, $6 \leq \mathcal{L}_4^*(X)$ and so $\Lcal_4^*(X)= \mathcal{L}_4^*(F)=6$.
\end{example}

\begin{figure}[h]
\includegraphics[width=10cm]{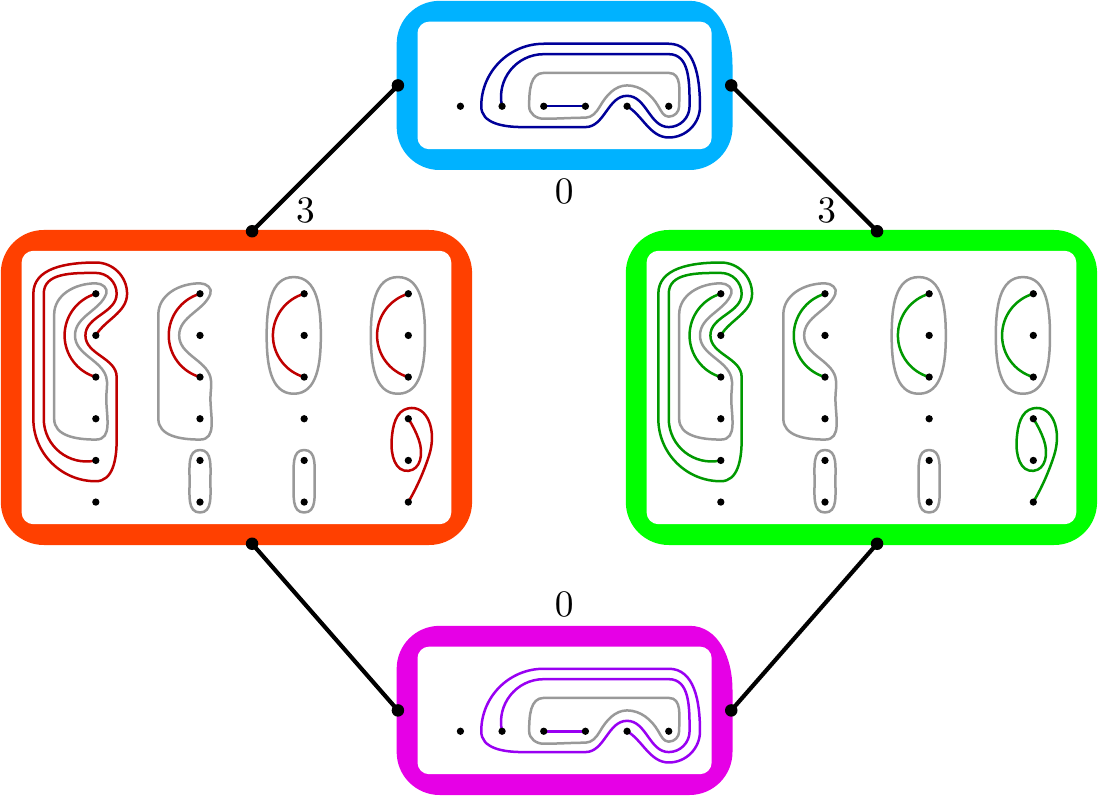}
\centering
\caption{The double of the standard ribbon disk for the Stevedore's knot.\label{fig:steve}}
\end{figure}

\subsection{Quadrisections of spun manifolds}
The goal of this section is to compute $\Lcal_4^*$ for spun knots and spun lens spaces. 

\begin{theorem}
\label{thm:spunlens}
Let $X_{p}$ be the spin of a lens space $L(p,q)$. Then $\mathcal{L}_4^*(X_{p}) = 6.$
\end{theorem}
\begin{theorem}
\label{thm:spunknots}
Let $F\subset S^4$ be the spin of a $2$-bridge knot in $S^3$. Then $\mathcal{L}_4^*(F) = 6.$
\end{theorem}

\noindent We divide the proofs of these results into the following steps, and then tie everything together in \autoref{sec:4.3.4}.

\begin{framed}
\textbf{Step 1}: Find quadrisections for twist-spun knots (\autoref{sec:4.3.1}). 

\textbf{Step 2}: Build an upper bound for $\Lcal_4^*(F)$ (\autoref{sec:4.3.2}). 

\textbf{Step 3}: Relate $\Lcal_4^*$ of $F$ and its 2-fold cover $X$ (\autoref{sec:4.3.3}). 
\end{framed}

\subsubsection{Diagrams of twist-spun knots} \label{sec:4.3.1}

We begin by describing the process of getting a quadrisection from a banded unlink diagram of an $n$-twist spun $b$-bridge link. In \cite{meier2017bridgeS4}, the authors consider the standard banded unlink diagram with $b-1$ bands for the ribbon disks/annuli, whose double is our spun link, and $b-1$ dual bands. Meier and Zupan then perturbed the unlink $2b$ times to turn it into a banded bridge splitting, which is in one-to-one correspondence with a trisection.

For a quadrisection, the banded bridge splitting corresponding to it has fewer requirements. Namely, it is not necessary for all the bands to each have a dual bridge disk simultaneously. Therefore, starting with the standard banded unlink as above, one just needs to perform $b-1$ perturbations instead of $2b-2$ perturbations.

\begin{proposition} \label{prop:quadrisection_spun_knot}
An $n$-twist spun $b$-bridge link admits a $(2b-1,b)$-quadrisection. 
\end{proposition}

\begin{proof}
    It is well-known that any knotted surface in $S^4$ admits a presentation as a banded unlink: an unlink such that after the band surgery is performed, we obtain another unlink. The banded unlink $(L,v)$ corresponding to the $n$-twist spun of links is depicted in \autoref{fig:4secperturb} (left). We think of $(L,v)$ as being contained in $S^3$. We perturb $(L,v)$ to obtain the position depicted in \autoref{fig:4secperturb} (right). Observe that this perturbed form satisfies the following conditions:
    \begin{enumerate}
        \item[(1)] The unlink $L$ is in bridge position.
        \item[(2)] The bands $v$ are described by surface-framed arcs $y^*$.
        \item[(3)] The surface-framed arcs describing the $n-1$ green bands and the shadows of red bridge disks together forms a collection of embedded, pairwise disjoint arcs in the bridge sphere.
        \item[(4)] After smoothing along the green bands, the surface-framed arcs describing the green bands and the surface-framed arcs describing the blue bands together forms a collection of embedded, pairwise disjoint arcs in the bridge sphere.
    \end{enumerate}

\begin{figure}[h]
\centering
\includegraphics[width=9cm]{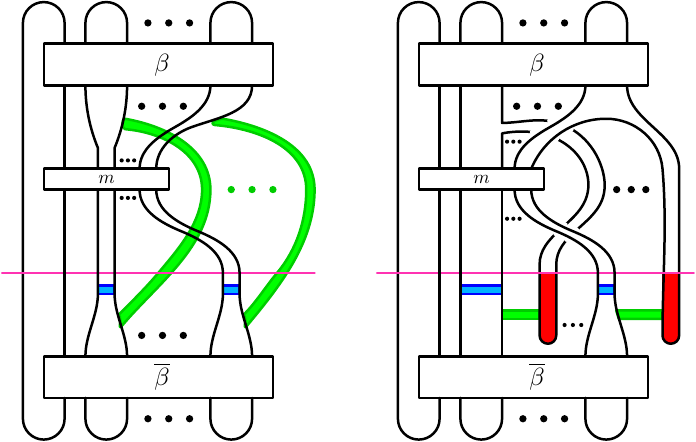}
\caption{From a banded unlink in $S^3$, we do one perturbation for each of the $n-1$ bands coming from the 1-handles of the ribbon surface bounded by the link that we are twist spinning.
\label{fig:4secperturb}}
\end{figure}

   We now argue that this gives a quadrisection with four trivial tangles $T_1$ through $T_4$, whose pairwise unions are unlinks.

    Recall that the unlink $L$ is in bridge position, and so it is divided into two trivial tangles $T_1$ and $T_2$. We push all the bands $v$ to be contained in $T_2$. Performing the band surgeries along the bands that intersect the red bridge disks give $T_3$. Finally, performing the band surgeries on the remaining bands gives $T_4.$ The tangles $T_1$ and $T_2$ are trivial by construction. Each of the remaining tangles comes from performing band surgeries on a trivial tangle where the bands satisfy conditions (2), (3) and (4) above. We get two new bridge disks for each band $v_i$: a shadow of a bridge disk is the surface-framed arc $y_i^*$ describing the band and the other shadow comes from taking $y_i^*$ together with the shadows of bridge disks that intersect the sides of $v_i.$ 

    Finally, $T_1\cup \overline{T_2}$ is already an unlink by definition. To see that the link $T_1 \cup \overline{T_4}$ is also an unlink since it is obtained from the original banded unlink by band surgeries. For the other pairs $T_i\cup \overline{T_{i+1}}$, one can construct the embedded disks that the link bound explicitly since the bands have been isotoped to be in a restricted position (see \autoref{fig:boundeddisks}).
\end{proof}

\begin{figure}[h]
    \centering
    \includegraphics[width=5cm]{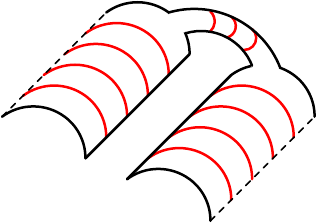}
    \caption{The product $T_i\times [\epsilon,0)$ of a trivial tangle with an interval together with a band $v$ described by a surface-framed arc followed by $T_i$ smoothed along $v$ is a trivial disk system.
    \label{fig:boundeddisks}}
\end{figure}

\subsubsection{Upper bound for \texorpdfstring{$\Lcal_4^*(F)$}{L*(F)}} \label{sec:4.3.2}

Let $F$ be the spun of a 2-bridge knot. By \autoref{prop:quadrisection_spun_knot}, the bridge number of $F$ is at most 3. On the other hand, one can see that surfaces in $S^4$ with bridge number less than three are unknotted. Thus, the bridge trisection described \autoref{prop:quadrisection_spun_knot} for $F$ is minimal. 
\autoref{fig:quadrisection_trefiol} is the 4-plane diagram of the spun trefoil obtained via this process and \autoref{fig:bridgepath} computes an upper bound for the $\Lcal_4^*$-invariant of such quadrisection. The following lemma generalizes this for all two-bridge knots.   

\begin{figure}[h]
\includegraphics[width=14cm]{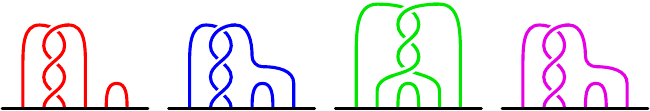}
\centering
\caption{The spine of a quadrisection of the spun trefoil in $S^4$.}
\label{fig:quadrisection_trefiol}
\end{figure}

\begin{lemma}\label{lem:L_at_most_6}
Let $F$ be the spin of a $2$-bridge knot. Then $\Lcal_4^*(F)\leq 6$.
\end{lemma}

\begin{proof}
Consider the quadrisection of $F$ given in \autoref{prop:quadrisection_spun_knot}. We will show that there is a loop in the dual curve complex intersecting $T_2$ (colored blue) in three edges and $T_4$ (colored purple) in three edges. The two paths of length three look identical except for the change in colors. An example of the loop and these paths for the spun trefoil is shown in \autoref{fig:bridgepath}.

Starting from pants decomposition $P_0$, which belongs to an efficient pair for $T_1\cup T_2$, we first move a reducing curve in $T_1\cup T_2$ to a pivot curve bounds a disk in neither $T_1$ nor $T_2.$ This is the pants decomposition $P_1$. For the second edge $P_1\mapsto P_2$, we move a cut-reducing curve in $T_1\cup T_2$ to a cut-reducing curve in $T_2\cup T_3$. Finally, we move the pivot curve to a reducing curve for $T_2$ and $T_3$ to obtain $P_3$, which belongs to an efficent pair for $T_2 \cup T_3.$
\end{proof}

\begin{figure}[h]
\includegraphics[width=10cm]{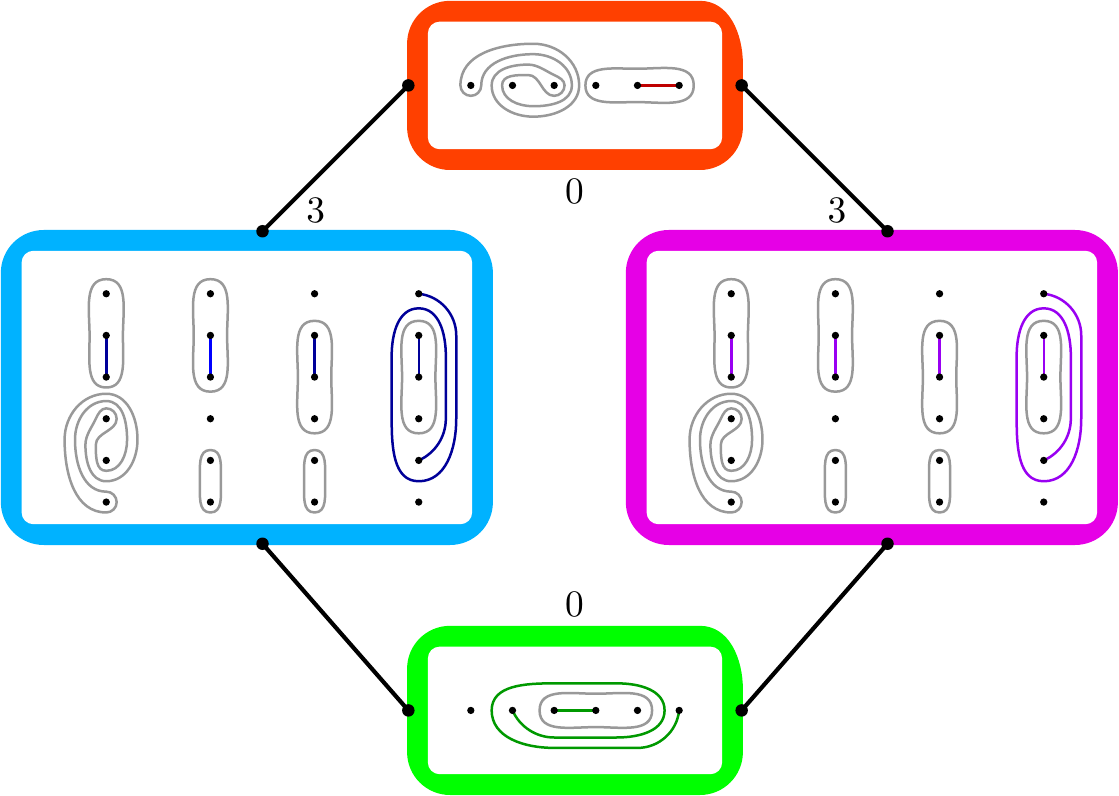}
\centering
\caption{A loop in $\mc{C}^*(\Sigma_{0,6})$ showing that $\mathcal{L}_4^*$ of a spun trefoil is at most 6.}
\label{fig:bridgepath}
\end{figure}

\subsubsection{\texorpdfstring{$F$}{F} and its 2-fold cover \texorpdfstring{$X$}{X}} \label{sec:4.3.3}

There is a close relationship between 3-bridge links and Heegaard splittings of genus two \cite[Theorem~8]{BirmanHilden75}. This relationship extends to the dual curve complexes of a six-punctured sphere and a genus two surface. For example, the loop of \autoref{fig:bridgepath} lies $\mc{C}^*(\Sigma_{0,6})$ and lifts to a loop in $\mc{C}^*(\Sigma_{2,0})$ depicted in \autoref{fig:arriba}. This new loop estimates the $\Lcal_4^*$-invariant of the 2-fold cover of $S^4$ along the spun trefoil. 
The following lemma builds upon this thought. 

\begin{figure}[h]
\includegraphics[width=10.5cm]{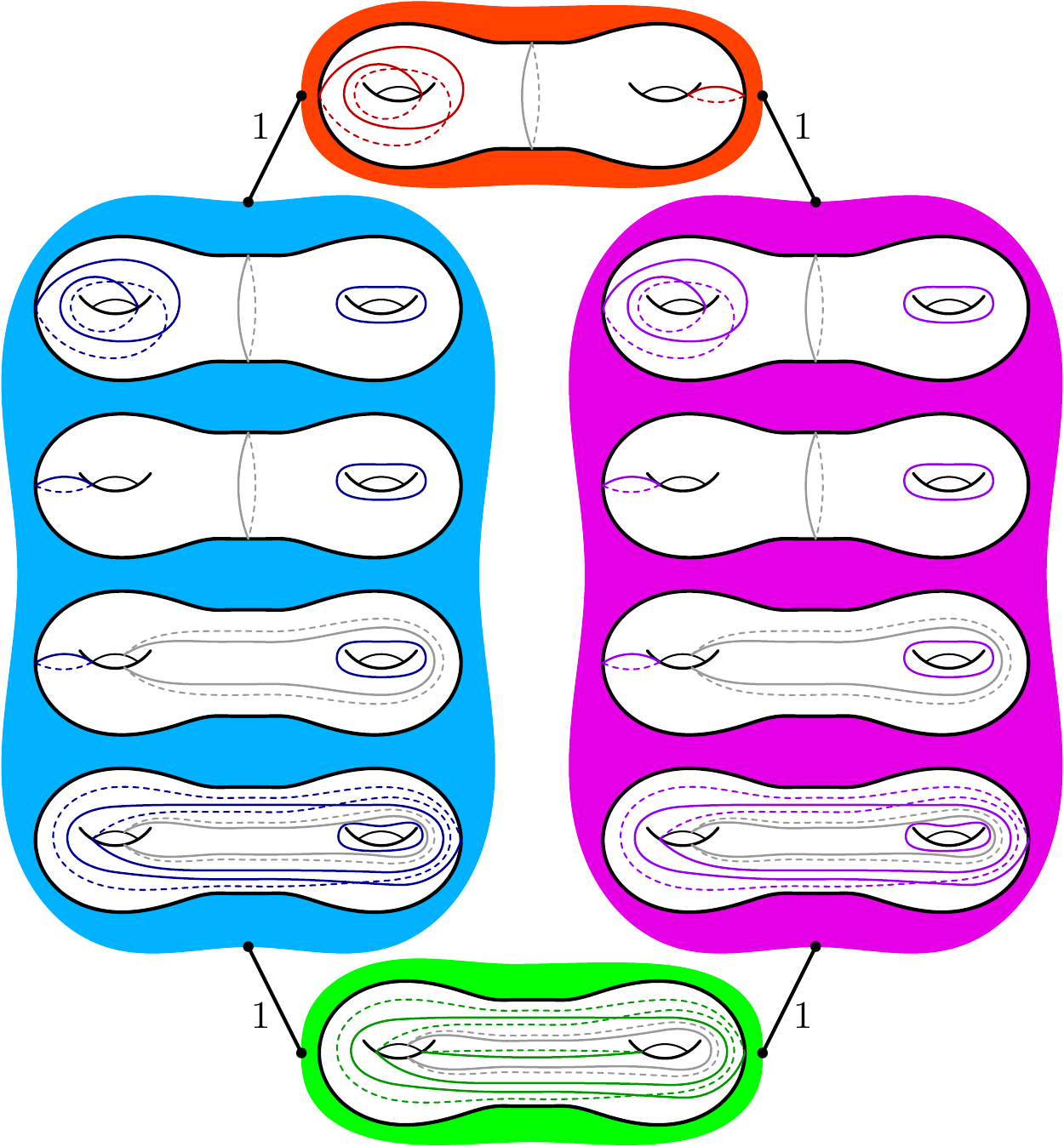}
\centering
\caption{A loop showing that $\mathcal{L}_4^*$ of a spun lens space is at most 6. All curves in this figure come from lifting curves appearing in \autoref{fig:bridgepath} to the 2-fold branched cover.}
\label{fig:arriba}
\end{figure}

\begin{lemma}\label{lem:L_under_covers}
Let $\Tcal$ be a $(g,b)$-multisection of a pair $(S^4,F)$ and let $\widetilde\Tcal$ be the multisection obtained by taking the $2$-fold cover of each piece. If $(g,b)=(0,3)$, then 
$$\mathcal{L}_n^*(\widetilde\Tcal)\leq \mathcal{L}_n^*(\Tcal).$$ 
\end{lemma}

\begin{proof}
Let $\Sigma$ and $\widetilde\Sigma$ be the central surfaces of $\Tcal$ and $\widetilde\Tcal$, respectively. We are assuming that $\Sigma$ is a six-punctured sphere and $\widetilde\Sigma$ is a closed surface of genus two. In particular, the quotient map by the hyperelliptic involution in $\widetilde \Sigma$ induces an isomorphism between $\mc{C}(\Sigma)$ and $\mc{C}(\widetilde\Sigma)$. This, in turn, descends to an isomorphism $\mc{C}^*(\widetilde\Sigma)\ra \mc{C}^*(\Sigma)$. 
We consider all the possible types of paths in the dual curve complex for $\Sigma$ and show that they each lift to a path in the dual curve complex for $\widetilde{\Sigma}.$

If a curve $x$ bounding two punctures moves to $x'$, then it moves to another curve bounding two punctures. The boundaries of the 4-holed sphere involved in the move includes the curve $y$ bounding an odd number of punctures in the pants decomposition, the other curve bounding two punctures, and the two punctures that $x$ surrounds. The curve $y$ lifts to the separating curve in $\widetilde{\Sigma}$ cutting it into two tori, each with one boundary component. The curves $x$ and $x'$ each lifts to a non-separating curve on contained in one of the tori. The corresponding move in $\widetilde{\Sigma}$ is then a dual path complex version of the $s$-move, which involves the 4-holed sphere whose boundaries are $\widetilde{y}$ and $\widetilde{x}$.

The second possibility is the cut curve $y$ moves to $y'$. Suppose first that $y'$ is a cut curve. Then, the boundaries of the 4-holed sphere involved in the move include the two other curves  in the pants the composition and two punctures. The two compressing curves lift to two non-separating curves, which will function as the boundaries of the 4-holed sphere in $\widetilde{\Sigma}$. The non-separability means that one curve gives rise to two components of the 4-holed sphere. The corresponding move $\widetilde{y}\mapsto \widetilde{y'}$ is then from a separating curve to a separating curve. If $y'$ is a compressing curve then, the 4-holed spheres in $\Sigma$ and $\widetilde{\Sigma}$ remain the same as when $y'$ is a cut curve. But now, $\widetilde{y'}$ is a non-separating curve.
\end{proof}

\subsubsection{Proof of \autoref{thm:spunlens} and \autoref{thm:spunknots}} \label{sec:4.3.4}

Let $F$ be a spun two-bridge knot in $S^4$ and let $\Tcal$ be the 3-bridge quadrisection of $F$ with $\Lcal_4^*(\Tcal)=\Lcal_4^*(F)$. Let $\widetilde\Tcal$ its 2-fold cover. By \autoref{lem:L_under_covers} and \autoref{lem:L_at_most_6}, 
\[\Lcal_4^*(\widetilde\Tcal)\leq \Lcal_4^*(\Tcal)=\Lcal_4^*(F)\leq 6.\]
 
On the other hand, $\widetilde\Tcal$ is a quadrisection for a spun lens space $X$ which we know has a fundamental group of order $0<p<\infty$. Thus $X$ is not diffeomorphic to a connected sum of copies of $S^1\times S^3$, $S^2\times S^2$, and $\pm \mathbb{CP}^2$. \autoref{thm:L>=6} implies that $6\leq \Lcal_4^*(\widetilde\Tcal)$. Therefore, $\Lcal_4^*(F)=6$. 

A similar proof shows that $\Lcal_4^*(X)=6$ for $X$ an arbitrary spun lens space. \hfill $\qed$

\section{Properties of efficient pairs} \label{sec:efficient}

Previous results on the Kirby-Thompson invariants relied heavily on how being efficient restricts the types of curves that appear in a pants decomposition. For instance, in the setting of nontrivial knotted surfaces in $S^4,$ each efficient pair contains a common curve that bounds an odd number of punctures. Following this precedent we study the properties of efficient pairs in this section, as a first step towards proving the theorems presented in \autoref{sec:lower_bounds} and \autoref{sec:genus_two_quadrisec}.

Zupan computed the bridge and pants distance of bridge splittings of knots in 3-manifolds $M$ such that $M$ does not contain an essential sphere \cite{zupan2013bridge}. The following lemma is a generalization of \cite[Lemma~5.6]{blair2020kirby} (see also \cite[Lemma~3.4]{aranda2022bounds}) to our setting and gives the calculations for the cases not addressed by Zupan. 
An interesting consequence of \autoref{lem:efficient} is that the set of efficient defining pairs in $\Pcal(\Sigma)$ is the same that the set of efficient defining pairs in $\mc{C}^*(\Sigma)$. One can also derive that the curves participating in efficient defining pairs arise from the standard Heegaard diagrams for unlinks in $\#^k S^1\times S^2$. 

\begin{lemma}\label{lem:efficient}
Let $(H,T)$, $(H',T')$ be two $(g,b)$-trivial tangles whose union is a $c$-component unlink $U_c$ in $Y_k \vcentcolon = \#^k S^1\times S^2$ in $(g,b)$-bridge position, with $\partial H = \partial H' = \Sigma$. Let $(P,P')$ be an efficient pair for $(H,T)$, $(H',T')$. Then the distance in $C^*(\Sigma)$ between $P$ and $P'$ is $d(P,P')=g-k+b-c$. Moreover, there is a geodesic $\lambda$ connecting $P$ to $P'$ satisfying the following properties. 
\begin{enumerate}
\item \label{lem:part_1} The path $\lambda$ consists of $(g-k)$ $S$-moves and $(b-c)$ $A$-moves.
\item \label{lem:part_2} Every curve in $P$ moves at most once, and furthermore:
\begin{enumerate}
\item every curve in $P$ intersects at most one curve in $P'$, and
\item if a curve in $P$ intersects a curve in $P'$, then they intersect in either one or two points.
\end{enumerate}
\item \label{lem:part_3} Only compressing curves move.
\end{enumerate}
\end{lemma}

\begin{proof}
We will proceed by induction on the pair $(g,b)$ using the dictionary order on $\mathbb{Z}_{\geq 0} \times \mathbb{Z}_{\geq 0}$, but first provide a broad overview of the proof strategy. 

\begin{framed}
\textbf{Step 1}: We prove the lemma for the base case $(g,b)=(1,1)$. 

\textbf{Step 2}: After curve counting, we conclude that there exist curves which stay fixed as we travel from one disk set to another. Such a curve bounds disks in two handlebodies simultaneously, which means that we get a sphere coming from two such disks. 

\textbf{Step 3}: This sphere allows us to simplify the (3-manifold, link)-pair into new (3-manifold, link)-pairs with lower complexity. This allows us to use the inductive hypothesis across several cases, which are outlined in \autoref{fig:lemma_overview}. 
\end{framed}

\textbf{Step 1}:
Recall that $\Sigma$ is a genus $g$ surface with $2b$ punctures, so $\Sigma$ has $3g+2b-3$ curves in any pants decomposition. In particular, there are no pants for the cases $(g,b) = (0,0), (0,1), (1,0)$. Furthermore, the cases where $g=0$ have been addressed in \cite[Lemma~5.6]{blair2020kirby} (see also \cite[Lemma~3.4]{aranda2022bounds}), so the lemma holds for $(g,b)<(1,1)$.

However for the sake of illustration, and to include a non-trivial case with genus in the base case, we also give a proof for the case $(g,b)=(1,1)$. In this case $\Sigma$ is a twice-punctured torus, which requires two curves for a pants decomposition, and $g=1$, $b=1$, $c=1$, and $k=0$ or $1$, since both $Y_0$ and $Y_1$ admit genus one Heegaard splittings. By a generalization of Waldhausen's theorem (see \cite[Theorem~2.2]{zupan2013bridge}), there exists a standard diagram of $(Y_k, U_c)$ such that \[d(P,P')\leq g-k+b-c = 1-k+1-1 = 1-k.\] If $k=1$, then $d(P,P')=0$ and we are done. If $k=0$, then $d(P,P') \leq 1$. We claim additionally that $d(P,P') \geq 1$, as any genus one Heegaard splitting of $S^3$ consists of two dual curves which intersect once. Thus $d(P,P')=1$.

Next we turn to parts (\ref{lem:part_1}), (\ref{lem:part_2}), and (\ref{lem:part_3}) for the case $k=0$. (For $k=1$ these parts are trivially satisfied.) 
We classify the possible pants decompositions of $(g,b)=(1,1)$ into two types: one where the two punctures are in the same pants, and one where they are in different pants. See \autoref{fig:lemma_base_case} for examples of each type. Recall that the curves we are considering each bound a \textit{c-disk}, which is either a \textit{compressing disk} or \textit{cut disk} (see \autoref{subsec:complexes}).
Note that when the punctures are in the same pants, both curves bound compressing disks. When the punctures are in different pants, one curve bounds a cut disk and the other bounds a compressing disk.

\begin{figure}[h]
\centering
\includegraphics[width=13cm]{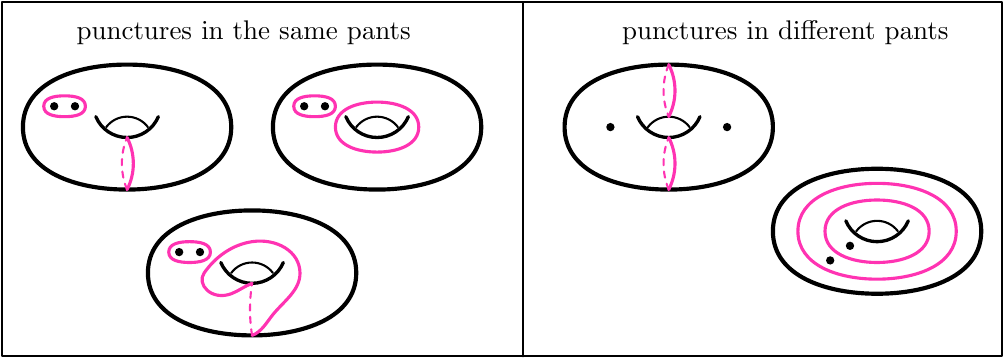}
\caption{Some examples of the two types of pants decompositions for the case $(g,b)=(1,1)$.}
\label{fig:lemma_base_case}
\end{figure}

When $k=0$, we know $d(P,P')=1$, so exactly once curve moves.
First assume both $P$ and $P'$ are pants decompositions with both punctures in the same pants. Then the curve which moves must be the non-separating curve. This move must be an $S$-move, since this curve must intersect the new curve exactly once in order for the curves to form a Heegaard diagram for $S^3$, and a quick check confirms that parts (\ref{lem:part_2}) and (\ref{lem:part_3}) hold as well.
Now assume (without loss of generality) that $P$ is a pants decomposition with punctures in the same pants, and $P'$ is a pants decomposition with punctures in different pants. In order for this to happen, the curve which moves must be the separating curve, and it must move to a non-separating curve, which is parallel to the other curve in $P$. However this produces a Heegaard diagram for $S^1 \times S^2$, and so this case cannot happen.
Finally, assume both $P$ and $P'$ are pants decompositions with punctures in different pants. Here, $P$ and $P'$ consist of two pairs of non-separating curves in $\Sigma$. Since $P$ and $P'$ differ by one edge, they must share a curve and additionally produce a Heegaard diagram for $S^1\times S^2$, but this is a contradiction, so this case also cannot happen. Thus the lemma holds for $(g,b)\leq(1,1)$.

\vspace{1em}
\textbf{Step 2}:
Now suppose that the lemma is true for any $(g_0,b_0)$-bridge splitting, where $(g_0,b_0) < (g,b)$. Again, by a generalization of Waldhausen's theorem (see \cite[Theorem~2.2]{zupan2013bridge}), there exists a standard diagram of $(Y_k, U_c)$ such that \[d(P,P')\leq g-k+b-c.\]
We will focus on proving equality, and prove parts (\ref{lem:part_1}), (\ref{lem:part_2}), and (\ref{lem:part_3}) along the way. By the above, at most $g-k+b-c$ curves in $P$ move, so at least \[(3g+2b-3)-(g-k+b-c)=2g+k+b+c-3\] curves stay fixed. Notice that $2g+k+b+c>3$ holds for $(g,b) \geq (1,1)$, so in particular $P$ and $P'$ have a common curve, which we call $x$. 
Recall that $x$ bounds either a cut disk or a compressing disk in each of $P$ and $P'$, and observe that $x$ bounds a cut disk in $P$ if and only if it bounds a cut disk in $P'$. If the sphere $S$ created by the union of the c-disks bounded by $x$ is separating in $Y_k$, then this is clear. If $S$ is non-separating, this follows since any unlink intersecting a sphere only once will not be nullhomotopic, which is a contradiction.

\vspace{1em}
\textbf{Step 3}:
The general strategy for the inductive step will be to cut along this common curve $x$, attach punctured disks to the newly created boundary components, and use induction on the new, smaller surface(s). A summary of this strategy is given in \autoref{fig:lemma_overview}, although some additional special cases within each subcase will need to be considered.

\begin{figure}[h]
\centering
\includegraphics[width=1\linewidth]{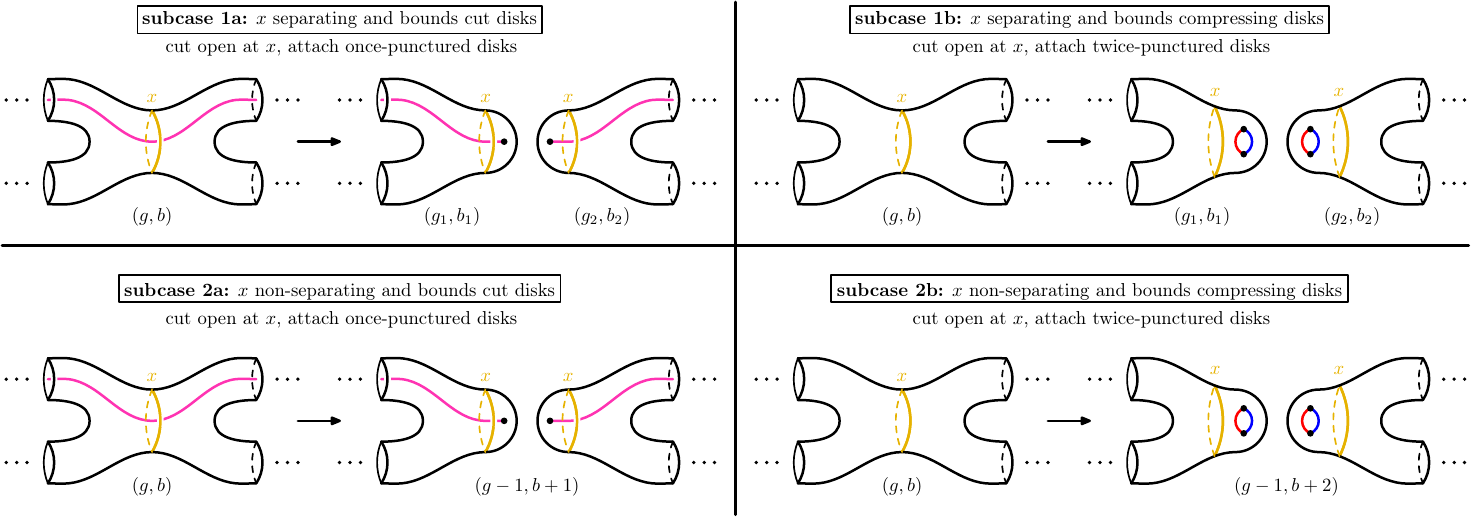}
\caption{A summary of the general strategy for the inductive step.}
\label{fig:lemma_overview}
\end{figure}

\textit{Case 1:} First assume $S$ is separating. Then $Y_k$ decomposes as a connected sum $Y_{k_1}\#Y_{k_2}$ along $S$, where $Y_{k_i}= \#^{k_i} S^1\times S^2$ and $k_1+k_2=k$, by the uniqueness of prime decompositions of $3$-manifolds. 

\indent \indent \textit{Subcase 1a:} If $x$ bounds cut disks, then each disk making up $S$ intersects a tangle once, so one of the unknots is split into two tangles across the sphere, which we can instead think of as two unknots connect summed together. So we have the decomposition \[(Y_k, U_c) = (Y_{k_1}\#Y_{k_2}, U_{c_1}\# U_{c_2}),\]  where the $U_{c_i}$ are $c_i$-component unknots and $c_1+c_2-1=c$. 

Let $\Sigma_1$, $\Sigma_2$ be the new Heegaard surfaces created by adding a punctured disk to each component of $\overline{\Sigma-x}$. Let $\Sigma_1=(g_1,b_1)$ and $\Sigma_2=(g_2,b_2)$, where $g_1+g_2=g$ and $b_1+b_2-1=b$. Note that because $x$ bounds cut disks, we have $b_i >0$. Identify $x$ with the punctures in the new disks added to create $\Sigma_1$ and $\Sigma_2$. (See \autoref{fig:lemma_overview}.) The curves in $P-\{x\}$ and $P'-\{x\}$ induce multicurves $P_i$ and $P_i'$ in $\Sigma_i$, for $i=1,2$. If both $\Sigma_1$ and $\Sigma_2$ are admissible, then $P_i$ and $P_i'$ are pants decompositions bounding c-disks in their respective tangles, and moves on $\Sigma_i$ between $P_i$ and $P_i'$ correspond to moves on $\Sigma$ between $P$ and $P'$ (fixing the curves on the other component of $\overline{\Sigma-x}$). In particular, because $x$ does not move, we can express the distance between $P$ and $P'$ as the sum of the distances between $P_i$ and $P_i'$ on each piece: \[d(P,P') = d(P_1,P_1')+d(P_2,P_2').\] 

Furthermore, if both $\Sigma_1$ and $\Sigma_2$ are admissible, then for each $i$ either $g_i<g$ or $b_i<b$. To see this, assume without loss of generality that $g_1=g$ and $b_1=b$ (note that because $b_i>0$ and $b_1+b_2-1=b$, this means $b_i$ can be at most $b$). This forces $g_2=0$ and $b_2=1$, which is not admissible. So by induction each pair $(P_i,P_i')$ satisfies the lemma. Therefore:
\begin{align*}
d(P,P') &=  d(P_1,P_1')+d(P_2,P_2')\\
&= g_1-k_1+b_1-c_1+g_2-k_2+b_2-c_2\\
&= g_1+g_2-k_1-k_2+b_1+b_2-c_1-c_2\\
&= g - k + (b+1) - (c+1)\\
&= g - k + b - c.
\end{align*}
For each $i$, let $\lambda_i$ be the geodesic connecting $P_i$ to $P_i'$ which satisfies parts (\ref{lem:part_1}), (\ref{lem:part_2}), and (\ref{lem:part_3}). The moves in $\lambda_1$ together with the moves in $\lambda_2$ create a geodesic $\lambda$ taking $P$ to $P'$, which consists of $(g_1-k_1)+(g_2-k_2)=(g-k)$ $S$-moves and $(b_1-c_1)+(b_2-c_2)=b+1-(c+1)=(b-c)$ $A$-moves, and again by induction, parts (\ref{lem:part_2}) and (\ref{lem:part_3}) hold as well. 

Now assume without loss of generality that only $\Sigma_1$ is admissible. As $\Sigma_2$ is not admissible and $x$ bounds cut disks, we know $(g_2,b_2) = (0,1)$. But this cannot happen, since we require curves bounding c-disks to be essential; in particular the curve cannot bound a once-punctured disk in $\Sigma$. Similarly, if both $\Sigma_1$ and $\Sigma_2$ are not admissible, then both $(g_i,b_i)=(0,1)$, which also cannot happen.

\indent \indent \textit{Subcase 1b:} If $x$ bounds compressing disks, then the unlink components do not intersect $S$, so we have the decomposition \[(Y_k, U_c) = (Y_{k_1}\#Y_{k_2}, U_{c_1}\sqcup U_{c_2}),\] where the $U_{c_i}$ are $c_i$-component unknots with $c_1 + c_2 = c$.

Let $\Sigma_1$, $\Sigma_2$ be the new Heegaard surfaces created by adding a disk with two punctures to each component of $\overline{\Sigma-x}$. Add to each side a new unlink made up of two trivial tangles, one in $H$ and one in $H'$, bounded by these punctures. Note that the number of unlink components on each side is now $c_i+1$. Let $\Sigma_1=(g_1,b_1)$ and $\Sigma_2=(g_2,b_2)$, where $g_1+g_2=g$ and $b_1+b_2-2=b$. (See \autoref{fig:lemma_overview}.) The curves in $P$ and $P'$ induce multicurves $P_i$ and $P_i'$ in $\Sigma_i$, for $i=1,2$. (Note that this time we are including $x$ in $P_i$ and $P_i'$.) If both $\Sigma_1$ and $\Sigma_2$ are admissible, then $P_i$ and $P_i'$ are pants decompositions bounding c-disks in their respective tangles, and moves on $\Sigma_i$ between $P_i$ and $P_i'$ correspond to moves on $\Sigma$ between $P$ and $P'$ (fixing the curves on the other component of $\overline{\Sigma-x}$). Because $x$ does not move, the newly added punctures do not affect the moves, so we can again express the distance between $P$ and $P'$ as the sum of the distances between $P_i$ and $P_i'$ on each piece: \[d(P,P') = d(P_1,P_1')+d(P_2,P_2').\] 

We claim that $\Sigma_1$ and $\Sigma_2$ will always be admissible. The only inadmissible surfaces are $(g,b) = (0,0), (0,1), (1,0)$. Since we added two punctures to each $\Sigma_i$, the cases $(0,0)$ and $(1,0)$ do not occur as possibilities for $\Sigma_i$. The case $(0,1)$ also cannot occur, as this would imply that $x$ bounds a disk in $\Sigma$. As both $\Sigma_1$ and $\Sigma_2$ are admissible, if both $g_i<g$, or if (without loss of generality) $g_1=g$ and $b_2>2$, then by induction each pair $(P_i,P_i')$ satisfies the lemma. Therefore:
\begin{align*}
d(P,P') &=  d(P_1,P_1')+d(P_2,P_2')\\
&= g_1-k_1+b_1-(c_1+1)+g_2-k_2+b_2-(c_2+1)\\
&= g_1+g_2-k_1-k_2+b_1+b_2-(c_1+1)-(c_2+1)\\
&= g - k + (b+2) - (c+2)\\
&= g - k + b - c.
\end{align*}
For each $i$, let $\lambda_i$ be the geodesic connecting $P_i$ to $P_i'$ which satisfies parts (\ref{lem:part_1}), (\ref{lem:part_2}), and (\ref{lem:part_3}). The moves in $\lambda_1$ together with the moves in $\lambda_2$ create a geodesic $\lambda$ taking $P$ to $P'$, which consists of $(g_1-k_1)+(g_2-k_2)=(g-k)$ $S$-moves and $(b_1-(c_1+1))+(b_2-(c_2+1))=b+2-(c+2)=(b-c)$ $A$-moves, and again by induction, parts (\ref{lem:part_2}) and (\ref{lem:part_3}) hold as well. 

Now assume without loss of generality that $g_1=g$ and $b_2=2$. This means $\Sigma_1=(g,b)$ and $\Sigma_2=(0,2)$, and thus we are stuck because we cannot use induction on the $\Sigma_1$ component, so now we will use a different approach. Recall that the $\Sigma_i$ were formed by adding disks with two punctures to each side, so before adding these disks we had surfaces $(g,b-1)$ and $(0,1)$, each with one boundary component. This time instead of adding a disk with two punctures to each component of $\overline{\Sigma-x}$, add a disk without punctures to create (abusing notation) $\Sigma_1=(g,b-1)$ and $\Sigma_2=(0,1)$, as in \autoref{fig:lemma_parallel1}. Note that now $k_1=k$, $k_2=0$, $c_1=c-1$, and $c_2=1$. There are two cases in which $\Sigma_1$ is not admissible, $\Sigma_1=(1,0)$ and $\Sigma_1=(0,1)$, but both were included in the base case. The former is part of the case $(g,b)=(1,1)$, and the latter was one of the cases addressed in \cite{blair2020kirby} (but also note that there are no possible moves, so the lemma holds trivially). So assume $\Sigma_1$ is admissible.

\begin{figure}[h]
\centering
\includegraphics[width=12cm]{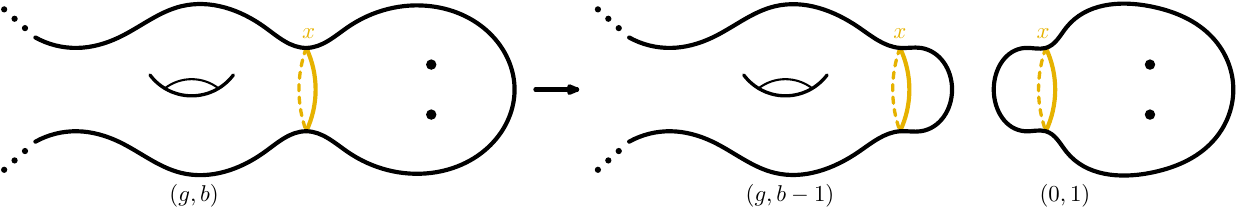}
\caption{The case where $x$ is separating, bounds compressing disks, and cuts off a $(0,1)$ surface on one side.}
\label{fig:lemma_parallel1}
\end{figure}

The curves in $P$ and $P'$ induce multicurves $P_1$ and $P_1'$ in $\Sigma_1$, which are almost pants decompositions, except the pants in $\Sigma$ formed by $x$ and the curves $y,z$ in $P_1$ (resp. $y',z'$ in $P_1'$) become an annulus between $y$ and $z$ (resp. $y'$ and $z'$). Here we consider punctures also to be curves, so in some cases the ``curves'' $y,z,y',z'$ may be punctures. Thus in $\Sigma_1$ the curves $y$ and $z$ (resp. $y'$ and $z'$) are now parallel, so identify them and call the resulting curve $w$ (resp. $w'$); see \autoref{fig:lemma_parallel2} for an example. (If a puncture and curve are being identified, think of the result as a puncture.) After this identification, we now have pants decompositions on $\Sigma_1$ coming from $P$ and $P'$, which we call again $P_1$ and $P_1'$. Although $\Sigma_2$ is not admissible, because $x$ does not move, and the two punctures in $\Sigma_2$ do not affect the moves in $\Sigma$, we still have that $d(P,P') = d(P_1,P_1')+d(P_2,P_2')$, and in this case $d(P_2,P_2')=0$. Thus by induction we have:
\begin{align*}
d(P,P') &=  d(P_1,P_1')\\
&= g_1-k_1+b_1-c_1\\
&= g - k + (b-1) - (c-1)\\
&= g-k+b-c.
\end{align*}
Let $\lambda$ be the geodesic connecting $P_1$ to $P_1'$ which satisfies parts (\ref{lem:part_1}), (\ref{lem:part_2}), and (\ref{lem:part_3}). If $w$ does not move in $\lambda$, that is, $w$ and $w'$ are the same, then the conclusions (\ref{lem:part_1}), (\ref{lem:part_2}), and (\ref{lem:part_3}) for $\Sigma$ follow immediately.

\begin{figure}[h]
\centering
\includegraphics[width=12cm]{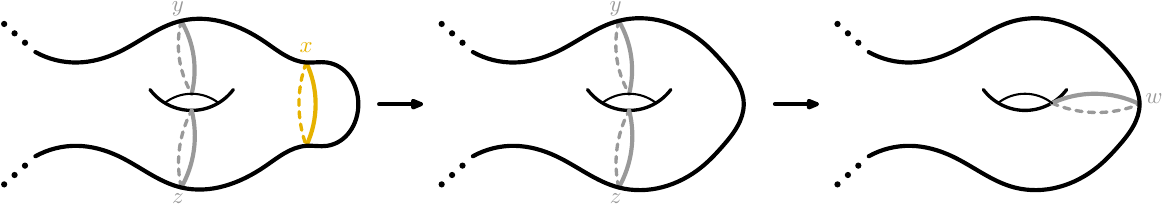}
\caption{Identifying the parallel curves $y$ and $z$.}
\label{fig:lemma_parallel2}
\end{figure}

Now suppose that $w$ does move in $\lambda$. By induction, we know $w$ moves exactly once and intersects exactly one other curve in $P_1'$ in either one or two points. Furthermore, the curve that $w$ intersects in $P_1'$ is $w'$ as they lie in the same four-holed sphere. This means that $y$ and $z$ each intersect each of $y'$ and $z'$ in either one or two points (depending on whether the move from $w$ to $w'$ is an S-move or an A-move). Thus it will take at least two moves to send $y$ to $y'$ and $z$ to $z'$. However, since all other curves in $\Sigma_1$ correspond to the same curves in $\Sigma$, and since $w$ is sent to $w'$ in one move, this means $d(P,P') > d(P_1,P_1')$, which is a contradiction. Thus $w$ does not move, and we are done with the case where $S$ is separating. 

\textit{Case 2:} Now we turn to the case where $S$ is non-separating, that is, $x$ is non-separating. Let $Q_{\pm}$ be the two pants in $\Sigma$ bounded by $x$ and other curves in $P$, keeping in mind that $Q_{\pm}$ could be the same pants.

\indent \indent \textit{Subcase 2a:} First assume $x$ bounds cut disks. Either $Q_+=Q_-$ or $Q_+ \neq Q_-$. If $Q_+=Q_-$, then the curve $x$ appears twice in the pants, so only one other curve in $P$ cobounds $Q_+=Q_-$ with $x$, which we denote by $y$. (Note that $y$ cannot be a puncture.) But since $x$ bounds a cut disk, there exist two (distinct) arcs in the handlebody which connect the disk bounded by $x$ to the disk bounded by $y$, so the disk bounded by $y$ must intersect a tangle twice, which contradicts the definition of c-disks.

So assume $Q_+ \neq Q_-$. Let $y_{\pm}$ and $z_{\pm}$ be the curves (or possibly punctures, which we think of here as curves) which cobound $Q_{\pm}$ (resp.) with $x$. Cut $\Sigma$ along $x$, and just as in Subcase 1a, fill in each newly created boundary component in $\Sigma$ with a disk with one puncture, and identify each copy of $x$ with the puncture. (See \autoref{fig:lemma_overview}.) Now we have a $(g-1,b+1)$ surface, which we denote by $\hat{\Sigma}$. Note that we have decreased the number of $S^1 \times S^2$ summands in $Y_k$ by one and increased the number of unlink components in $U_c$ by one. This new surface $\hat{\Sigma}$ must be admissible; the only way we could have created an inadmissible surface is if we started with $(1,0)$, which is itself inadmissible. The curves in $P-\{x\}$ and $P'-\{x\}$ induce multicurves $\hat{P}$ and $\hat{P'}$ in $\hat{\Sigma}$ which remain pants decompositions, and moves on $\hat{\Sigma}$ between $\hat{P}$ and $\hat{P'}$ correspond to moves on $\Sigma$ between $P$ and $P'$ (as $x$ is fixed). Thus by induction we have: 
\begin{align*}
d(P,P') &=  d(\hat{P},\hat{P'})\\
&= (g-1)-(k-1)+(b+1)-(c+1)\\
&= g-k+b-c.
\end{align*}
Let $\hat{\lambda}$ be the geodesic connecting $\hat{P}$ to $\hat{P'}$ which satisfies parts (\ref{lem:part_1}), (\ref{lem:part_2}), and (\ref{lem:part_3}). As $x$ does not move, $\hat{\lambda}$ induces a geodesic $\lambda$ connecting $P$ to $P'$ with the same properties.

\indent \indent \textit{Subcase 2b:} Finally assume $x$ bounds compressing disks. Either $Q_+=Q_-$ or $Q_+ \neq Q_-$. If $Q_+=Q_-$, then the curve $x$ appears twice in the pants, so only one other curve in $P$ cobounds $Q_+=Q_-$ with $x$, which we denote by $y$. (Note that $y$ cannot be a puncture.) The part of the surface containing $Q_+=Q_-$ must look the left side of \autoref{fig:lemma_subcase2b}. Cut $\Sigma$ along $x$, fill in each newly created boundary component in $\Sigma$ with a disk with one puncture, and identify each copy of $x$ with the puncture. Additionally add a new unlink made up of two trivial tangles, one in $H$ and one in $H'$, bounded by these punctures, so now the surface looks like the right side of \autoref{fig:lemma_subcase2b}. 

\begin{figure}[h]
\centering
\includegraphics[width=10cm]{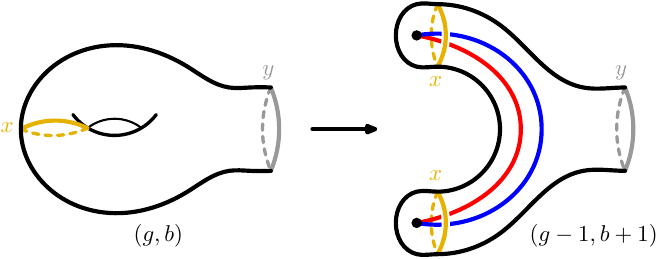}
\caption{The case where $x$ is non-separating, bounds compressing disks, and bounds the same pairs of pants.}
\label{fig:lemma_subcase2b}
\end{figure}

Now we have a $(g-1,b+1)$ surface, which we denote by $\hat{\Sigma}$. Note that we have decreased the number of $S^1 \times S^2$ summands in $Y_k$ by one and increased the number of unlink components in $U_c$ by one. As before, this new surface $\hat{\Sigma}$ must be admissible. The curves in $P-\{x\}$ and $P'-\{x\}$ induce multicurves $\hat{P}$ and $\hat{P'}$ in $\hat{\Sigma}$ which remain pants decompositions, and moves on $\hat{\Sigma}$ between $\hat{P}$ and $\hat{P'}$ correspond to moves on $\Sigma$ between $P$ and $P'$ (as $x$ is fixed). Thus by induction we have: 
\begin{align*}
d(P,P') &=  d(\hat{P},\hat{P'})\\
&= (g-1)-(k-1)+(b+1)-(c+1)\\
&= g-k+b-c.
\end{align*}
Let $\hat{\lambda}$ be the geodesic connecting $\hat{P}$ to $\hat{P'}$ which satisfies parts (\ref{lem:part_1}), (\ref{lem:part_2}), and (\ref{lem:part_3}). As $x$ does not move, $\hat{\lambda}$ induces a geodesic $\lambda$ connecting $P$ to $P'$ with the same properties.

If $Q_+ \neq Q_-$, let $y_{\pm}$ and $z_{\pm}$ be the curves (or possibly punctures, which we think of here as curves)  which cobound $Q_{\pm}$ (resp.) with $x$. Cut $\Sigma$ along $x$, and just as in Subcase 1b, fill in each newly created boundary component in $\Sigma$ with a disk with two punctures. Along with the punctured disk, add a new unlink made up of two trivial tangles, one in $H$ and one in $H'$, bounded by these punctures. (See \autoref{fig:lemma_overview}.) Now we have a $(g-1,b+2)$ surface, which we denote by $\hat{\Sigma}$. Note that we have decreased the number of $S^1 \times S^2$ summands in $Y_k$ by one and increased the number of unlink components in $U_c$ by two. As before, this new surface $\hat{\Sigma}$ must be admissible. The curves in $P$ and $P'$ induce multicurves $\hat{P}$ and $\hat{P'}$ in $\hat{\Sigma}$ which remain pants decompositions, and moves on $\hat{\Sigma}$ between $\hat{P}$ and $\hat{P'}$ correspond to moves on $\Sigma$ between $P$ and $P'$ (as $x$ is fixed). Thus by induction we have: 
\begin{align*}
d(P,P') &=  d(\hat{P},\hat{P'})\\
&= (g-1)-(k-1)+(b+2)-(c+2)\\
&= g-k+b-c.
\end{align*}
As before, see that parts (\ref{lem:part_1}), (\ref{lem:part_2}), and (\ref{lem:part_3}) are satisfied as well.
\end{proof}

We now present two additional lemmas detailing further properties of efficient pairs. 

\begin{lemma}\label{lem:Amoves}
Let $(H,T)$ and $(H',T')$ be two trivial tangles whose union is an unlink in bridge position embedded in $\#^k S^1\times S^2$. Let $(P,P')$ and $\lambda$ be as in \autoref{lem:efficient}. Suppose that $x\mapsto x'$ is an A-move in $\lambda$ for some $x\in P$ and $x'\in P'$. Then $x$ bounds embedded disks in both handlebodies $H$ and $H'$. Furthermore, $x$ does {\em not} bound a c-disk in $P'$.
\end{lemma}
\begin{proof}
Suppose $x\mapsto x'$ is an A-move along the geodesic $\lambda$ in \autoref{lem:efficient}. By property (\ref{lem:part_3}) of \autoref{lem:efficient}, the edges in $\lambda$ commute so we can make $x\mapsto x'$ the last move. Focus on the 4-holed sphere where the A-move occurs, the boundaries of this bound compressing disks for $H'$. They separate a $3$-ball from $H'$ and $x$ is a simple closed curve in its boundary. In particular, $x$ bounds a disk in $H'$. Hence, $x$ bounds a compressing disk in both handlebodies $H$ and $H'$. If $x$ bounds a c-disk in $P'$, then it bounds a c-reducing disk for $T\cup \overline{T'}$. In particular, the move $x\mapsto x'$ is redundant. 
\end{proof}

\begin{lemma}\label{lem:cut}
Let $(P,P')$ be an efficient defining pair as in  \autoref{lem:efficient} 
with $(g,b)\neq (0,2)$. If $b>c$, then there is a cut curve $x\in P\cap P'$. 
Moreover, the number of cut curves in $P\cap P'$ is at least $(b-c)$ if $g>0$ and $c>1$, or $(b-2)$ otherwise.
\end{lemma}

\begin{proof}
Part (1) of \autoref{lem:efficient} states that there are exactly $(g-k)$ S-moves and $(b-c)$ A-moves between $P$ and $P'$. From Part (2) of the same result, these move commute. This means that the 4-holed spheres where the A- and S-moves occur have disjoint interiors. In particular, the boundaries of such subsurfaces of $\Sigma$ are fixed between $P$ and $P'$. 

Let $x\mapsto x'$ be one such A-move and let $E$ be the 4-holed sphere supporting the move. Denote by $\{\partial_1, \partial_2, \partial_3, \partial_4\}$ the boundaries of $E$. We can relabel the boundaries so that $x$ bounds a pair of pants with boundaries $\{x,\partial_1,\partial_2\}$ and $x'$ bounds a pair of pants with boundaries $\{x',\partial_1,\partial_3\}$. 
By the previous paragraph, each $\partial_i$ bounds a (possibly inessential) c-disk in both $T$ and $T'$. 
Moreover, since $x$ bounds a compressing disk for $T$ (\autoref{lem:efficient}(3)), we know that $\partial_1$ and $\partial_2$ must bound the same kind of disk. 
If $\partial_1$ and $\partial_2$ bound compressing disks, then $x$ will be forced to bound a compressing disk for both $T$ and $T'$. This will contradict \autoref{lem:Amoves}. Similarly, $\partial_3$ and $\partial_4$ cannot bound compressing disks simultaneously. Therefore, all $\partial_i$ must bound a (possibly inessential) cut disk in both $T$ and $T'$. However, if all $\partial_i$ are inessential, then the pants decomposition consists of two pairs of pants, where each pair of pants has $x$ and two of the $\partial_i$'s as boundaries. As $(g,b)\neq (0,2)$, we conclude that 
at least one boundary of $E$ must be an essential cut curve in $P\cap P'$. 

We now explain the lower bound for the number of essential cut curves in $P\cap P'$. 
Let $\mc{E}=\{E_1,\dots, E_{b-c}\}$ be the collection of 4-holed spheres supporting the $(b-c)$ A-moves between $P$ and $P'$. Define $F\subset \Sigma$ to be the surface obtained by gluing neighboring $E_i$ along their boundaries. Observe that $F$ may not be connected. Let $\Gamma$ be the dual graph of the cover of  $F=\cup_i E_i$. In other words, $\Gamma$ has vertex set $\mc{E}$ and edges corresponding to the circle components of $E\cap E'$ for distinct vertices $E,E'\in \mc{E}$. 
Denote by $m\geq 1$ the number of connected components of $F$. Since $\Gamma$ has $m$ connected components, the Euler characteristic of $\Gamma$ (vertices minus edges) is at most $m$. As $|\mc{E}|=b-c$, we obtain that $\Gamma$ has at least $b-c-m$ edges. And by definition of an edge, we conclude that the number of essential curves in $P\cap P'$ of the form $E\cap E'$ for some $E,E'\in \mc{E}$ is at least $b-c-m$.  

Let $G$ be the complement of $F$ in $\Sigma$. 
Suppose that the number of connected components $|F|+|G|$ of $F$ and $G$ is at least two; that is, either $F$ is disconnected or $G$ is non-empty. Then every connected component of $F$ will contain a boundary circle that is essential in $\Sigma$. This is true as $F$ and $G$ are complementary subsurfaces of a connected surface $\Sigma$. By definition of $F$, such boundary circle will belong to a unique 4-holed sphere $E\in \mc{E}$, so it is a cut curve in $P\cap P'$ that has not been counted in the previous paragraph. Hence, the number of essential cut curves in $P\cap P'$ is at least $b-c$ whenever $|F|+|G|\geq 2$. 

In order to finish the lemma, it is enough to discuss the case $|F|+|G|=1$. 
Here, $G$ is empty, $F$ is equal to $\Sigma$, and $m=1$. Thus, all essential curves in $P\cap P'$ are cut and correspond to all the edges of $\Gamma$. 
Now, since $\Gamma$ is the dual graph of a cover for $\Sigma = \cup_i E_i$, the Euler characteristic of $\Gamma$ is $1-g$. On the other hand, by definition, the Euler characteristic of $\Gamma$ is equal to $|\mc{E}|-|P\cap P'|$ which yields the equation 
$|P\cap P'|=b+g-c-1$. One can check that this number is at least $b-c$ for $g>0$ or $c>1$, and $b-2$ else. 
\end{proof}

\section{Lower bounds for \texorpdfstring{$\Lcal_n^*(\Tcal)$}{L(T)}} \label{sec:lower_bounds}

Lower bounds for $\Lcal_n^*(\Tcal)$ are difficult to calculate in general, but in this section we make some inroads towards this goal, categorized into two subsections.
In \autoref{sec:ogawa} we study multisections with small $\Lcal_n^*(\Tcal)$-invariant and prove \autoref{thm:ogawaV1} and \autoref{thm:ogawaV2a}, which are the main technical inputs for the proofs of \autoref{thm:leq1} and \autoref{thm:leq2}, respectively. In \autoref{sec:irred_multisections}, we give lower bounds for c-irreducible multisections in terms of the parameters $(g,k;b,c)$, and in particular prove \autoref{thm:lower_bound_1_V2}, which is the main technical input for the proof of \autoref{thm:gb}.
Our characterizations throughout will use the terminology from \autoref{subsec:new_from_old}. 

\subsection{Multisections with small \texorpdfstring{$\Lcal_n^*$}{L*}-invariant}\label{sec:ogawa}

In this section, we consider multisections with $\Lcal_n^*$-invariant at most two. 
The main idea is that the condition $\Lcal_n^*\leq 2$ forces $\Tcal$ to have c-reducing curves. Moreover, the multisections obtained by c-reducing $\Tcal$ also have small $\Lcal_n^*$-invariant. 
\autoref{thm:ogawaV1} shows that the pairs $(X,F)$ having $\Lcal_n^*\leq 1$ are simple. 
When $\Lcal_n^*(\Tcal)=2$, \autoref{thm:ogawaV2a} shows that we can c-reduce $\Tcal$ to a collection of multisections $(g,b)=(0,3),(1,1),(2,0)$. 
In light of this,
the authors of this work are interested in understanding multisections with small $(g,b)$-complexity.
\begin{problem}
List all pairs $(X,F)$ admitting a multisection with $(g,b)\in \{(0,3)$, $(1,1)$, $(2,0)\}$ and $\Lcal_n^*=2$.
\end{problem}

\subsubsection{Multisections with \texorpdfstring{$\Lcal_n^* \leq 1$}{L* leq 1}}

First we classify all pairs $(X,F)$ with $\Lcal_n^*$-invariant less than one. The following theorem is the main ingredient in the proof of \autoref{thm:leq1} (see \autoref{subsec:detection}).

\begin{theorem}
\label{thm:ogawaV1}
Let $\mathcal{T}$ be a multisection with $\mathcal{L}^*_n(\mathcal{T})\leq 1$. Then $\Tcal$ is completely decomposable.
\end{theorem}

\begin{proof}
Let $\lambda$ be the loop in the dual curve complex of $\Sigma$ passing through all the disk sets $\Dcal_i=\Dcal(H_i,T_i)$ and realizing $\Lcal_n^*(\Tcal)$. Since $\Lcal_n^*(\Tcal)\leq 1$, $\lambda$ is a union of efficient defining pairs between $\Dcal_i$ and $\Dcal_{i+1}$ and at most one ``internal'' edge inside some $\Dcal_{i_0}$. In the following argument, we will find curves fixed along $\lambda$. Then we will ``compress'' $\Tcal$ along such curves and get new multisections with $\Lcal_n^*(\Tcal')\leq 1$ and $(g',b')<(g,b)$. The theorem will follow by induction on $(g,b)\geq (0,2)$; the base case $(0,2)$ is completely decompsable by definition.

Let $x$ be an essential curve in $\Sigma$ that separates a once-holed torus in $\Sigma$ or bounds a cut disk in some $\Dcal_i$. By \autoref{lem:efficient}, if $x$ belongs to a pants decomposition $P$ for some efficient defining pair $(P,Q)$, then $x$ also belongs to $Q$. Thus, since there is at most one edge that does not belong to a geodesic between an efficient defining pair, any such $x$ must be fixed along the loop $\lambda$. In particular, the multisection is c-reducible along $x$. If $x$ bounds a once-holed torus, $x$ must bound a compressing disk in all tangles. Then, we can cut $\Sigma$ along $x$ and cap off the two new boundary components with twice-punctured disks. This will produce two central surfaces for two multisections with complexities $(g_i,b_i)\leq (g,b)$, $i=1,2$. Given that $x$ is a reducing curve for $\Tcal$, curves bounding c-disks for a handlebody in some tangle of $\Tcal$ correspond to curves bounding c-disks in some tangle of one the multisections $\Tcal_1$ or $\Tcal_2$. Moreover, since $x$ is fixed along $\lambda$, $\Lcal_n^*(\Tcal_i)\leq 1$ for $i=1,2$. If both $(g'_i,b'_i)<(g,b)$ for both $i=1,2$, we can proceed by induction. If not, then $(g,b)=(1,1)$ and $\Tcal$ is the connected sum of a $(1,0)$-multisection and a $(0,1)$-bridge multisection; thus completely decomposable. 
A similar argument can be used if $x$ bounds a cut disk: here, we cut $\Sigma$ along $x$ and cap off the new boundary components with once-puncture disks. 

We end the proof of this theorem by explaining why curves $x$ as in the previous paragraph must exist. If $g>k_i$ for some $i$, by \autoref{lem:efficient}, there is an S-move in the efficient defining pair between $\Dcal_i$ and $\Dcal_{i+1}$. By definition, S-moves occur in a one-holed torus. We then choose $x$ to be the boundary. On the other hand, if $b<c_i$ for some $i$, by \autoref{lem:cut}, there is a cut curve $x$ in the efficient defining pair between $\Dcal_i$ and $\Dcal_{i+1}$. The remaining option is that $g=k_i$ and $b=c_i$ for all $i$ which, in turn, is completely decomposable. 
\end{proof}

\subsubsection{Multisections with \texorpdfstring{$\Lcal_n^*=2$}{L*=2}}

Our goal now is to classify all pairs $(X,F)$ with $\Lcal_n^*$-invariant equal to two. The following theorem is the main ingredient in the proof of \autoref{thm:leq2} (see \autoref{subsec:detection}).

\begin{theorem}\label{thm:ogawaV2a}
Let $\mathcal{T}$ be a multisection for $(X,F)$. If $\mathcal{L}^*_n(\mathcal{T})= 2$, then $\Tcal$ is the c-connected sum and self-tubing of bridge multisections $\mathcal{T}'$ with complexity \[(g,b)\in \{(0,1),(0,2),(0,3),(1,0),(1,1),(2,0)\}\] and $\Lcal_n^*(\Tcal')\leq 2$. 
\end{theorem}

\begin{proof}
We begin with a broad overview of the proof, which is broken into three main cases based on the parameters $(g,k,b,c)$ of the multisection.
The theorem will then follow by induction on $(g,b)\geq (0,3)$; the base cases are established by assumption.

\begin{framed}
\textbf{Setup}: We suppose that the $\mathcal{L}^*_n$-invariant is precisely two so that two curves move; furthermore, because we are dealing with a \textit{loop} in the dual curve complex, we know that one curve has to move twice (rather than two different curves moving once). Call this curve $x$, and let $y,y'$ be the curves $x$ moves to: $x\mapsto y$ and $x\mapsto y'$.

\textbf{Case 1}: The case $g=k_i$ and $b=c_i$ for all $i$ is immediate.

\textbf{Case 2}: We then consider the case when there exists $i_0$ with $b>c_{i_0}$, that is, the number of arcs is more than the number of disks in some disk system. We analyze what the curves that are fixed around the elementary pants move $x\mapsto y$ can look like. More precisely, the move happens in a $4$-holed sphere with boundaries $z_1,z_2,z_3$ and $z_4$, so one can characterize these $z_i$'s. We will see in \autoref{Claim:y_is_cut} that $y$ must bound a cut disk in $T_1$. This property of $y$ is an ingredient in \autoref{Claim:z2_moves} to either conclude that we started with a $(0,3)$-multisection in the first place, or produce a curve that is fixed along the loop $\lambda$ realizing the $\mathcal{L}^*$-invariant. The latter possibility allows us to conclude that the multisection is a c-connected sum and self-tubing of multisections with lower complexities.

\textbf{Case 3}: The last case we have yet to consider is when there exists $i_0$ with $g>k_{i_0}$ and $b=c_i$ for all $i$. The difference in this case compared to the previous one is that the $4$-holed sphere with boundaries $z_1,z_2,z_3$ and $z_4$ has two of the curves (say $z_3$ and $z_4$) being parallel. \textbf{Subcase 3a} and \textbf{Subcase 3b} go through the possibilities of disks that $z_1$ and $z_2$ bound. From these analyses, we conclude that either there is a curve that is fixed along $\lambda$ allowing us to use induction, or we have started with \[(g,b)\in \{(0,1),(0,2),(0,3),(1,0),(1,1),(2,0)\}\] in the first place.
\end{framed}

Let $\lambda$ be the loop in the dual curve complex of $\Sigma$ passing through all the disk sets $\Dcal_i=\Dcal(H_i,T_i)$ and realizing $\Lcal_n^*(\Tcal)$. Since $\Lcal_n^*(\Tcal)=2$, $\lambda$ is a union of efficient defining pairs between $\Dcal_i$ and $\Dcal_{i+1}$ and two ``internal'' edges inside some disk sets $\Dcal_{i}$. Moreover, since $\lambda$ is a loop, some curve $x$ must move at least twice along $\lambda$ in two distinct disk sets $\Dcal_{i_0}$ and $\Dcal_{i_1}$. Given that $\Lcal_n^*(\Tcal)=2$, we conclude that at most one such $x$ can exist. 

Suppose that a curve $w$ is fixed along $\lambda$ and either separates a once-holed torus in $\Sigma$ or bounds a cut disk in some $\Dcal_i$. If this is the case, we can modify $\Sigma$ by cutting along $w$ as in the proof of \autoref{thm:ogawaV1} to get multisections of smaller complexity $(g,b)$ and induct. In what follows, we show that we can always find such a $w$ to cut along. 

\vspace{1em}
\textbf{Case 1}: For all $i$, $g=k_i$ and $b=c_i$. Here, all tangles are the same and the multisection is completely decomposable with $\Lcal_n^*(\Tcal)=0$. 

\vspace{1em}
\textbf{Case 2}: There exists $i_0$ with $b>c_{i_0}$. By \autoref{lem:cut}, there exists a cut curve $x$ in some efficient defining pair. As discussed above $x$ moves twice along $\lambda$. Without loss of generality, $x\mapsto y$ in $\Dcal_1$ and $x\mapsto y'$ in $\Dcal_i$ (see \autoref{fig:case2}). We focus on the 4-holed sphere, depicted in \autoref{fig:case2bc}(a), where the move $x\mapsto y$ occurs. Since $x$ bounds a cut disk for $T_1$, exactly one of $\{z_1,z_2\}$ (resp. $\{z_3,z_4\}$) bounds a cut disk for $T_1$ and the other curve compresses due to parity. If one of these cut curves is essential in $\Sigma$, it will be fixed along $\lambda$ and we can apply the inductive step. Thus, without loss of generality, we assume that $\{z_1,z_3\}$ bound once-punctured disks in $\Sigma$ and $\{z_2,z_4\}$ bound compressing disks in $T_1$. In particular, there is a shadow $s$ of $T_1$ passing through $z_1$ and $z_3$ once and disjoint from $z_2$ and $z_4$. 

\begin{figure}[h]
\centering
\includegraphics[width=7cm]{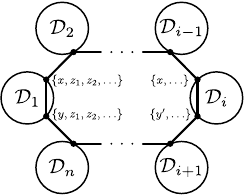}
\caption{In Case 2 of \autoref{thm:ogawaV2a}, the loop $\lambda$ realizing $\mathcal{L}^*_n$ intersects each disk set in a vertex except for $\mathcal{D}_1$ and $\mathcal{D}_i$, where the intersection is an edge representing the moves $x\mapsto y$ in $\Dcal_1$ and $x\mapsto y'$ in $\Dcal_i$, respectively.}
\label{fig:case2}
\end{figure}

In what follows, we will study the behavior of $y$, $z_2$, and $z_4$ along $\lambda$. These sub-paths are formed by efficient defining pairs so they only contain $A$-moves and $S$-moves. By construction, $z_2$ and $z_4$ cannot participate in $S$-moves as each of them cobound a once-punctured annulus with $x$.

\begin{figure}[h]
\centering
\includegraphics[width=8cm]{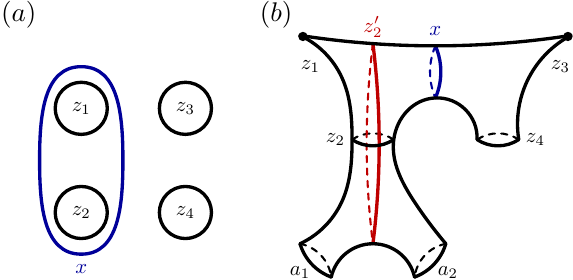}
\caption{(a) The T-shirt where the dual move $x\mapsto y$ in Case 2 of \autoref{thm:ogawaV2a} happens. (b) The local picture of the move $z_2\mapsto z_2'$ in the proof of \autoref{Claim:z2_moves}.}
\label{fig:case2bc}
\end{figure}

\begin{shaded}
\begin{claim}\label{Claim:y_is_cut}
The curve $y$ bounds a cut disk in $T_1$.
\end{claim}

\begin{proof}[Proof of \autoref{Claim:y_is_cut}]
Suppose $y$ bounds a compressing disk in $T_1$. 
Recall that $y$ lies in the 4-holed sphere in \autoref{fig:case2bc}(a). 
We have two possibilities: either $y$ separates the punctures $z_1$ and $z_3$ or not. If the former happens, then $y$ will always intersect the shadow $s$, forcing $y$ to bound a cut disk for $T_1$. Hence, $y$ cannot separate $z_1$ and $z_3$ and $\{y,z_2,z_4\}$ cobound a pair of pants. 

We will focus on the efficient defining pair between $\Dcal_1$ and $\Dcal_n$. Since moves in this efficient pair commute and $\{y,z_2,z_4\}$ cobound a pair of pants, at most one of these curves can participate in an $A$-move. In particular, two of the three curves will bound compressing disks in both tangles $T_1$ and $T_n$. This forces the third curve to also bound a reducing curve in both tangles $T_1$ and $T_n$. By the contrapositive of \autoref{lem:Amoves}, all three curves $\{y,z_2,z_4\}$ must be fixed between $\Dcal_1$ and $\Dcal_n$. By the same argument, we obtain that $\{y,z_2,z_4\}$ must be fixed along $\lambda$ between $\Dcal_i$ and $\Dcal_n$. Thus, $y$ and $y'$ must be equal and the moves $x\mapsto y$ and $x\mapsto y'$ is redundant: we can leave $\{x,z_2,z_4\}$ fixed between $\Dcal_i$ and $\Dcal_1$. This contradicts the assumption that $\Lcal_n^*(\Tcal)=2$. 
\end{proof}
\end{shaded}

From \autoref{Claim:y_is_cut}, $y$ is a cut curve for $T_1$ and so it is fixed in every efficient defining pair between $\Dcal_i$ and $\Dcal_1$. In particular, $y=y'$ bounds a cut disk in the tangles $T_i, T_{i+1}, \dots, T_n, T_1$. 

\begin{shaded}
\begin{claim}\label{Claim:z2_moves}
If $z_2$ (resp. $z_4$) moves in some efficient defining pair between $\Dcal_j$ and $\Dcal_{j+1}$, then we can assume that $z_2$ (resp. $z_4$) bounds a twice-punctured disk. 
\end{claim}

\begin{proof}[Proof of \autoref{Claim:z2_moves}]
Say $z_2\mapsto z'_2$ is an $A$-move between the disk sets $\Dcal_1$ and $\Dcal_2$. Since $b>c_{i_0}$, the curves $z_2$ and $z_4$ are distinct and the surface $\Sigma$ near $x$ can be depicted as in \autoref{fig:case2bc}(b). 
Here, the two new curves $a_1$ and $a_2$ bound c-disks in both tangles $T_1$ and $T_2$ as edges in an efficient defining pair commute by \autoref{lem:efficient}. We know that $z'_2$ bounds a compressing disk for $T_2$, so there is a shadow for $T_2$ starting at $z_1$ and disjoint from $z'_2$. Thus, $a_1$ must bound a cut disk for $T_2$ and so for $T_1$. By the same reason, replacing $z_1$ with $x$, $a_2$ must bound a cut disk in both tangles $T_1$ and $T_2$. If any of $\{a_1, a_2\}$ is essential in $\Sigma$, we proceed by induction. Hence, we can assume that both $\{a_1,a_2\}$ bound once-punctured disks and $z_2$ bounds a twice-punctured disk. 

We can repeat the argument above to prove \autoref{Claim:z2_moves} for $1\leq j <i$. By replacing $x$ with $y$ above, we can show the result for the case $i\leq j \leq n$.
\end{proof}
\end{shaded}

By \autoref{Claim:z2_moves}, if $z_2$ (or $z_4$) does not bound a twice-punctured disk, then it is fixed along $\lambda$. We can then fill $\Sigma-z_2$ (or $\Sigma-z_4$) with twice-punctured disks to get multisection(s) $\Tcal'$ with $\Lcal_n^*(\Tcal')\leq 2$ and complexities $(g',b')< (g,b)$. The proof of \autoref{thm:ogawaV2a} will follow by induction. 
It remains to discuss the case when both $z_2$ and $z_4$ bound twice-punctured disks. However, note that this is the base case $(g,b)=(0,3)$, so we are done.

\vspace{1em}
\textbf{Case 3}: There exists $i_0$ with $g>k_{i_0}$. By the work on Case 2, we can assume that $b=c_i$ for all $i$. By \autoref{lem:efficient}, each efficient defining pair in $\lambda$ has no A-moves. 
By the same result, the efficient defining pair between $\Dcal_{i_0}$ and $\Dcal_{i_0 + 1}$ has at least one S-move $a\mapsto b$ which, in turn, yields the existence of a curve $x$ in the same efficient defining pair that separates a once-holed torus in $\Sigma$. As discussed at the beginning of the proof, $x$ moves twice along $\lambda$. Without loss of generality, $x\mapsto y$ in $\Dcal_1$ and $x\mapsto y'$ in $\Dcal_i$ for some $i\neq 1$ (see \autoref{fig:case3}(a)). 
The move $x\mapsto y$ occurs in a four-holed sphere with boundaries equal to two copies of $a$ and two other curves $z_1$ and $z_2$ as in \autoref{fig:case3}(b). 
Recall that $x$ bounds a compressing disk for $T_1$ and $\{z_1,z_2\}$ bound c-disks in $T_1$. Thus, $z_1$ and $z_2$ bound the same kind of c-disk, cut or compressing, in $T_1$. 

\indent \indent \textbf{Subcase 3a}:  Suppose they both bound cut disks. Then they cannot be essential curves in $\Sigma$ since that would create a cut curve fixed along $\lambda$. Thus, we can assume that $z_1$ and $z_2$ bound once-punctured disks. In this case, $(g,b)=(1,1)$.

\begin{figure}[h]
\centering
\includegraphics[width=12cm]{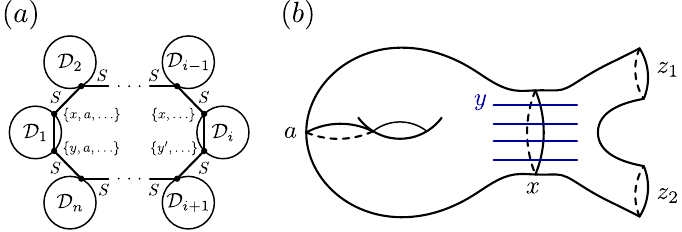}
\caption{Setup for Case 3.}
\label{fig:case3}
\end{figure}

\indent \indent \textbf{Subcase 3b}: Both $z_1$ and $z_2$ bound compressing disks in $T_1$. 
We now discuss the behavior of $y$ along $\lambda$. From \autoref{fig:case3}(b), 
we observe that $y$ should bound a compressing for $T_1$ as $a$, $z_1$, and $z_2$ do. 
Suppose first that $y$ is non-separating in $\Sigma$. Given that we are assuming that efficient defining pairs along $\lambda$ only have S-moves, the curves $\{a,y,z_1,z_2\}$ are forced to be fixed along $\lambda$ in the efficient defining pairs between $\{\Dcal_1,\Dcal_n\}$, $\{\Dcal_n, \Dcal_{n-1}\}$, $\dots$, $\{\Dcal_{i+1},\Dcal_i\}$. This is because none of these curves lie in a once-holed torus. In particular, $y=y'$ and the move $x\mapsto y$ in $\Dcal_1$ is unnecessary, contradicting the assumption that $\Lcal_n^*(\Tcal)=2$. Hence, $y$ must be separating in $\Sigma$. 
Moreover, $y$ (resp. $x$) cuts a once-holed torus in $\Sigma$ where $a$ lies, and the curve $y$ (resp. $x$) is fixed along every efficient defining pair in $\lambda$ between $\{\Dcal_j, \Dcal_{j+1}\}$ for $1\leq j<i$ (resp. $i\leq j\leq n$). In particular, $y=y'$.

If $z_1=z_2$, then $(g,b)=(2,0)$ as in the conclusion. So we may assume that $z_1\neq z_2$. 
In particular $\{z_1,z_2\}$ do not lie inside a once-holed torus in $\Sigma$ since they cobound a pair of pants with $x$. Thus, $\{z_1,z_2\}$ do not participate in S-moves and $\{z_1,z_2\}$ are fixed in the efficient defining pair between $\Dcal_j$ and $\Dcal_{j+1}$ for $1\leq j<i$. By replacing $x$ with $y$ we obtain the same conclusion for $i\leq j \leq n$. Hence, $\{z_1,z_2\}$ are fixed along $\lambda$. 
Here, we can cap off the boundaries of $\Sigma-(z_1\cup z_2)$ with twice-punctured disks and still get multisections $\Tcal'$ with $\Lcal_n^*(\Tcal')=2$ and complexities $(g',b')\leq (g,b)$. If the last inequality is strict, we can proceed by induction. The inequality is not strict only when $z_1$ and $z_2$ bound twice punctured disks; that is $(g,b)=(1,2)$. The multisection $\Tcal$ is the connected sum of multisections with complexities $(g,b)=(1,0),(0,1)$; hence completely decomposable. 
\end{proof}

\subsection{\texorpdfstring{$\Lcal_n^*$}{L*}-invariants of c-irreducible multisections}\label{sec:irred_multisections}

The goal of this section is to give a lower bound for c-irreducible multisections $\Tcal$ in terms of its invariants $(g,k;b,c)$. We give two bounds: one for multisections (\autoref{thm:lower_boundV1}) and an improvement for trisections (\autoref{thm:lower_bound_1}). We then rephrase the latter to obtain classical genus bounds for $\Lcal_3^*(\Tcal)$ (\autoref{thm:lower_bound_1_V2}) in the spirit of \cite[Theorem~6.4]{blair2020kirby}. 
The main idea for all lower bounds is that, for c-irreducible multisections, efficient defining pairs between consecutive disk sets do not have too many common curves. 

\begin{theorem}\label{thm:lower_boundV1}
Let $\Tcal$ be a $(g,k;b,c)$-multisection such that $(g,b)\neq (0,2),(1,0),(2,0)$. Let 
\[M=
\begin{cases}
2\max_i (g-k_i+b-c_i), & \text{if $g> 0$ or $c_i>1$ $\forall i$}\\
2(b-2), & \text{otherwise.}
\end{cases}
\]
If $\Tcal$ is c-irreducible, then $\Lcal_n^*(\Tcal)\geq M$.
\end{theorem}

\begin{proof}
Let $\lambda$ be a loop in the dual curve complex of $\Sigma$ passing through all the disk sets $\Dcal_i=\Dcal(H_i,T_i)$ and realizing $\Lcal_n^*(\Tcal)$. By definition, $\lambda$ is the union of efficient defining pairs between $\{\Dcal_i, \Dcal_{i+1}\}$ and $\Lcal_n^*(\Tcal)$-many extra edges connecting these pairs. In what follows, we will find a curve $x\subset \Sigma$ is fixed along $\lambda$, implying by definition, that $\Tcal$ is c-reducible. 

For each $1\leq i \leq n$, let $l_i$ be the number of curves $x$ in an efficient defining pair between $\{\Dcal_i, \Dcal_{i+1}\}$ bounding cut disks in $T_i$ or separating a once-holed torus from $\Sigma$. By \autoref{lem:efficient}, such curves do not move along efficient defining pairs. If one of these curves participates in at most one edge of $\lambda$, then it must be fixed along $\lambda$, and $\Tcal$ will be c-reducible. Therefore, if $\Lcal_n^*(\Tcal)<2l_i$ for some $i$, then $\Tcal$ is c-reducible. 

On the other hand, we can give a lower bound on the number of such curves $x$. 
Since $(g,b)\neq (2,0)$, \autoref{lem:efficient} implies that there are at least $g-k_i$ curves in the pair $\{\Dcal_i,\Dcal_{i+1}\}$ separating a once-holed torus from $\Sigma$. \autoref{lem:cut} states that the number of cut curves is at least $b-c_i$ if $g>0$ or $c_i>1$ $\forall i$, or $b-2$ otherwise. Hence, $2l_i\leq M$ for all $i$ and the result holds. 
\end{proof}

We know from \autoref{prop:(0,2)(1,0)_cases} that $\Lcal_n^*=0$ for multisection with complexities $(g,b)=(0,2), (1,0)$. The following propositions discuss the missing case $(g,b)=(2,0)$. 

\begin{proposition}
Let $\Tcal$ be a multisection with complexity $(g,b)=(2,0)$. If $\Tcal$ is irreducible, then $\Lcal_n^*(\Tcal)\geq 2$.
\end{proposition}

\begin{proof}
As in the proof of \autoref{thm:lower_boundV1}, it is enough to give a lower bound on the number of curves separating a once-holed torus from $\Sigma$. Since $(g,b)=(2,0)$, there is exactly one such curve, and $l_i=1$. Hence, $\Lcal_n^*(\Tcal)\geq 2$. 
\end{proof}

\subsubsection{Lower bound for trisections (\texorpdfstring{$n=3$}{n=3})}
The lower bound above can be improved if we specialize in trisections.

\begin{proposition}\label{prop:no_reducing_curves_V2}
Let $\mc T$ be a $(g,k;b,c)$-bridge trisection. Let $(P,P')$ and $(Q,Q')$ be efficient defining pairs for $(T_1,T_2)$ and $(T_1,T_3)$, respectively. Suppose that $(P\cap P')\cap Q$ is not empty. Then $\mc T$ is c-reducible. 
\end{proposition}

\begin{proof}
Let $x$ be a curve in $P\cap P'\cap Q$. Denote by $\lambda$ a fixed geodesic between $Q$ and $Q'$ from \autoref{lem:efficient}. 
If $x$ is fixed along $\lambda$, then it bounds a c-disk in $T_1$, $T_2$, and $T_3$ and $\Tcal$ is c-reducible.
Thus, we may assume that $x$ moves along $\lambda$. By \autoref{lem:efficient}, $x$ moves once via an S-move or an A-move $x\mapsto x'$. 
By the same lemma, $x$ and $x'$ bound compressing disks in $T_1$ and $T_3$, respectively. Since $T_1\cup \overline{T}_2$ is an unlink and $x$ is a c-reducing curve, $x$ must bound a compressing disk in $T_2$. Hence, if $x\mapsto x'$ is an S-move, then $x$ and $x'$ satisfy the hypothesis of \autoref{lem:stabilization} and $\mc T$ is c-reducible. 

Suppose that $x\mapsto x'$ is an A-move. By \autoref{lem:Amoves}, we know that both $x$ and $x'$ bound disks in both handlebodies $\{H_1,H_3\}$, $x$ does not bound a c-disk in $T_3$, and $x'$ does not bound a c-disk in $T_1$.  
Since $Q\in \Dcal(T_1)$ (resp. $Q\in \Dcal(T_3)$), we can find a shadow system $s_1$ for $T_1$ (resp. $s_3$ for $T_3$) such that $s_1$ (resp. $s_3$) intersects once every cut curve of $Q$ (resp. $Q'$) and is disjoint from every compressing curve of $Q$ (resp. $Q'$). 
We focus on the four-holed sphere $E$ where the move $x\mapsto x'$ occurs. It has four boundaries $\{z_1,z_2,z_3,z_4\}$ in $Q\cap Q'$ all of which bound (possibly trivial) cut disks for both $T_1$ and $T_3$ (see \autoref{fig:no_red_curves}(a)). In particular $s_1\cap E$ (resp. $s_3\cap E$) is exactly two subarcs $a_1\cup a_2$ (resp. $b_1\cup b_2$). Since $x'\in Q'$ bounds a compressing disk in $T_3$, $x'\cap (b_1\cup b_2)=\emptyset$. Thus the condition $|x\cap x'|=2$ forces $b_1$ and $b_2$ to intersect $x$ in exactly one interior point. An analogous condition holds between $a_1\cup a_2$ and $x'$. 

\begin{figure}[h]
\centering
\includegraphics[width=10cm]{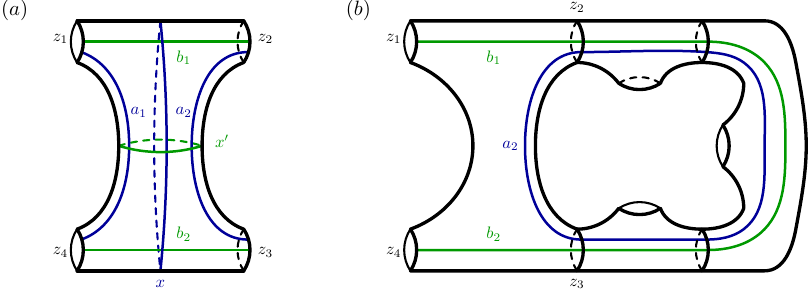}
\caption{A local model for $\Sigma$ near $x\cup x'$.}
\label{fig:no_red_curves}
\end{figure}

Suppose that $b_1$ and $b_2$ belong to the same arc of $T_3$. The connected component of $s_3$ containing $b_1$ is denoted by $t$ and connects an endpoint of $b_1$ with an endpoint of $b_2$. Without loss of generality, suppose that $t$ connects the boundaries $z_2$ and $z_3$. By construction of $s_3$, there is a (possibly empty) path of pairs of pants of $\Sigma-Q'$ starting in $z_2$ and ending in $z_3$ as in \autoref{fig:no_red_curves}(b) with the property that $t$ intersects each pair of pants in exactly one arc. The curves intersecting $t$ must bound cut disks in $T_3$ and so in $T_1$. In particular, by the construction of $s_1$, only one connected component of $s_1$ intersects these curves. This implies (see \autoref{fig:no_red_curves}(b)) 
that the subarc $a_2$ belongs to a loop component of $s_1$. This is a contradiction since shadow systems are collections of embedded arcs in $\Sigma$. 

The paragraph above showed that $b_1$ and $b_2$ must belong to distinct arcs of $T_3$. In particular, $x$ bounds a disk in $H_3$ intersecting two arcs of $T_3$ in exactly one interior point. Since, $x$ already bounds compressing disks for $T_1$ and $T_2$, we obtain that $\Tcal$ satisfies the hypothesis of \autoref{lem:(12)red_implies_cred}. By this lemma, $\Tcal$ is c-reducible, as desired. 
\end{proof}

\begin{theorem}\label{thm:lower_bound_1}
Let $\mc T$ be a c-irreducible $(g,k;b,c)$-bridge trisection. Then, 
\[ \Lcal_3^*(\mc T) \geq 6g+3b+(k_1+k_2+k_3)+(c_1+c_2+c_3)-9.\]
\end{theorem}

\begin{proof}
Let $\lambda$ be a loop in $\mc{C}^*(\Sigma)$ as in the definition of $\Lcal_n^*(\Tcal)$ (see \autoref{fig:islands}).
By \autoref{prop:no_reducing_curves_V2}, each curve in $P_i\cap P'_i$ must move in the $\Dcal_i$ island. Thus, the distance between two consecutive efficient defining pairs is at least $|P_i\cap P'_i|$. By \autoref{lem:efficient}, \[|P_i\cap P'_i|=|P_i|-d(P_i,P'_i)=(3g+2b-3)-(g-k_i+b-c_i).\]
Hence, 
$\Lcal_n^*(\Tcal)\geq \displaystyle \sum_{i=1}^3\left( 2g+k_i+b+c_i-3\right)$.
\end{proof}

We can rephrase the lower bound above in the following theorem. 

\begin{theorem}\label{thm:lower_bound_1_V2}
Let $\mc T$ be a c-irreducible $(g,b)$-bridge trisection for the pair $(X,F)$. 
Then, 
\[ \Lcal_3^*(\mc T) \geq 7(g-1)-\chi(X)+4b+\chi(F), \]
where $\chi(E)$ denotes the Euler characteristic of $E$.
\end{theorem}

\begin{proof}
The trisection surface $\Sigma$ is a $2b$-punctured orientable surface of genus g.
By definition, the unpunctured surface $\hat{\Sigma}$ is the central surface for a $(g,k)$-trisection of $X$. Thus, 
\[ \chi(X)=2+g-(k_1+k_2+k_3) \quad \text{and} \quad \chi(F)=c_1+c_2+c_3-b.\]
The result holds by substituting the equations in \autoref{thm:lower_bound_1}.
\end{proof}

\section{Genus two quadrisections}\label{sec:genus_two_quadrisec}

The goal of this section is to determine which closed 4-manifolds admit quadrisections with complexity $(g,k)=(2,1)$ and $\Lcal_4^*\leq 6$. Towards this end we obtain two main results, \autoref{thm:L>=6} and \autoref{thm:L=6}, proven in \autoref{sec: main theorem 1} and \autoref{sec: main theorem 2}, respectively. In \autoref{section:curves_in_genus_two} we collect preliminary results on curves on genus two surfaces needed for the proof of \autoref{thm:L=6}.

\begin{theorem}\label{thm:L>=6}
Let $\Tcal$ be a $(2,1)$-quadrisection of a closed $4$-manifold $X$. If $\Lcal_4^*(\Tcal)<6$, then $X$ is diffeomorphic to $\#^2 S^1\times S^3$ or $\#^h S^1\times S^3\#^iS^2\times S^2\#^j\mathbb{CP}^2\#^k \overline{\mathbb{CP}^2}$ where $h=\{0,1\}$.
\end{theorem}

\begin{theorem}\label{thm:L=6}
Let $\Tcal$ be a $(2,1)$-quadrisection of a closed $4$-manifold $X$. If $\Lcal_4^*(\Tcal)=6$, then $X$ is diffeomorphic to the spin of a lens space, the twisted-spin of a lens space, $S^1\times S^3 \# S^2\times S^2$ or $S^1\times S^3\# S^2\widetilde{\times} S^2$.
\end{theorem}

\subsection{Multisections with \texorpdfstring{$\Lcal_4^*(\Tcal)<6$}{L*(T)<6}}\label{sec: main theorem 1}

Fix $\mathcal{T}$ a $(2,1)$-quadrisection of a closed 4-manifold $X$. Let $(P,P')$, $(Q,Q')$, $(R,R')$, and $(S,S')$ be efficient defining pairs for $(H_1,H_2)$, $(H_3,H_2)$, $(H_3,H_4)$, and $(H_1,H_4)$, respectively. By \autoref{lem:efficient}, there exist simple closed curves $f$, $f'$, $g$, $g'$, $\psi_1$, $\psi_2$, $\gamma_1$, $\gamma_2$ in $\Sigma$ such that $P=\{g,\psi_1,\psi_2\}$, $P'=\{g',\psi_1,\psi_2\}$, $Q=\{f,\gamma_1,\gamma_2\}$, $Q'=\{f',\gamma_1,\gamma_2\}$, $|g\cap g'|=1=|f\cap f'|$, and $\psi_2,\gamma_2$ are separating circles in $\Sigma$. 

The following lemmas discuss particular quadrisections we will see often. A Heegaard splitting $U\cup_F V$ is said to be \textit{weakly reducible} if there are compressing disks $E\subset U$ and $D\subset V$ with disjoint boundaries. Compressing along $E$, along $D$, and along both $E$ and $D$ yield three new simpler surfaces. This process is called \textit{untelescoping} a weakly reducible Heegaard splitting. See \cite{Schultens3m} for a more detailed discussion on this process. 

\begin{lemma}\label{lem:Y_is_S3}
Suppose that $Y_{13}=H_1\cup_\Sigma \overline{H_3}$ is a weakly reducible Heegaard splitting with $H_1(Y_{13},\Z)=0$. Then, $Y_{13}=S^3$ and $X\cong S^4$. 
\end{lemma}
\begin{proof}
Since $H_1\cup_\Sigma H_3$ is a genus two weakly reducible Heegaard splitting, this can be untelescoped to have a thin level 2-sphere (see \cite{thin_pos_3M}). By Haken's Lemma, $H_1\cup_\Sigma H_3$ is reducible. The 3-manifold $Y_{13}$ is then the connected sum of two lens spaces. Since $S^3$ is the only lens space with trivial first homology over $\Z$, $Y_{13}=S^3$. 

We now focus our attention on the quadrisection of $X$. The 3-manifold $Y_{13}$ separates $X$ into two 4-manifolds $A-\interior(B^4)=X_1\cup X_2$ and $B-\interior(B^4)=X_3\cup X_4$. Since $Y_{13}=S^3$, $X\cong A\#B$ where $A$ and $B$ are closed 4-manifolds with $(2;1,1,0)$-trisections. The latter can be seen by filling in the ``bisections'' $X_i\cup X_{i+1}$ with a 4-ball. By \cite{MSZ_Classifying}, these are stabilizations of the genus zero trisection of $S^4$. Hence, $X\cong A\#B\cong S^4$.
\end{proof}

\begin{proposition}\label{prop:4section_2+1}
Suppose $d(P',Q')\leq 2$ and $d(Q,R)\leq 1$. Then $\mathcal{T}$ is c-reducible or $X\cong S^4$.  
\end{proposition}
\begin{proof}
First, recall that circles in $Q$ and $R$ bound compressing disks in $H_3$ and exactly one circle from $Q$ and one circle in $R$ are separating. In particular, given the separating circle in $Q$ (resp. $R$), the non-separating curves in $Q$ (resp. $R$) are determined by the handlebody $H_3$. Hence, the condition $d(Q,R)\leq 1$ forces $Q$ and $R$ to have the same non-separating circles. Thus, $\{f,\gamma_1\}\subset R$. By \autoref{lem:efficient}, the pants decomposition $R'$ contains exactly two non-separating curves $\{x,y\}$ with $x$ dual to one of $\{f,\gamma_1\}$ and disjoint from the other, and $y\in \{f,\gamma_1\}$. 

Suppose first that $d(P',Q')\leq 1$. By analogous reasoning as before, the non-separating curves in $P'$ and $Q'$ are the same; i.e., $\{f',\gamma_1\}=\{g',\psi_1\}$. If $\gamma_1=\psi_1$, then $\gamma_1$ will compress in the handlebodies $H_1$, $H_2$, and $H_3$. Moreover, $\gamma_1$ will either compress in $H_4$ ($\gamma_1=y$) or be dual to a compressing disk in $H_4$ ($\gamma_1=x$). Thus, $\mathcal{T}$ will be c-reducible by see \autoref{lem:stabilization}. If $\gamma_1=g'$, then $f'=\psi_1$. Here, depending on which one of $\{f,\gamma_1\}$ equals $y$, we will get a pair of curves as in \autoref{lem:stabilization}: if $f=y$ (resp. $\gamma_1=y$), then $f$ (resp. $\gamma_1$) bounds a disk in $H_3$, $H_4$ (resp. $H_2$, $H_3$, $H_4$) and its dual curve $f'$ (resp. $g$) bounds a curve in $H_1$ and $H_2$ (resp. $H_1$). 

Suppose now that $d(P', Q')=2$ and denote by $X$ the pants decomposition between $P'$ and $Q'$. One can see that, $\psi_2$ must move from $P'$ to $X$, and $\gamma_2$ most move from $X$ to $Q'$. If both $\{g',\psi_1\}$ are fixed between $P'$ and $Q'$, then we can proceed as if $d(P',Q')\leq 1$. Thus, we can assume that the dual move $X\rightarrow Q$ corresponds to one of $\{g',\psi\}$ moving to $\gamma_2$. In particular $\{f',\gamma_1\}$ and $\{g',\psi_1\}$ have one common curve. 
If $\gamma_1=\psi_1$, then $\gamma_1$ will bound a disk in $H_1$, $H_2$, $H_3$, and will be either bound a disk in $H_4$ or be dual to a disk in $H_4$. Thus, $\mathcal{T}$ will be c-reducible by \autoref{lem:stabilization}. 
We are left to discuss the situation when $\psi_1=f'$. Here, the pants decomposition $X$ contains the non-separating curves $\{\psi_1=f',g',\gamma_1\}$. As elements of $H_1(\Sigma,\Z)$, we can choose orientations of the curves so that, $[\gamma_1]=[\psi_1]+[g']$. Then the algebraic intersection between $\gamma_1$ and $g$ will equal $g'\cdot g=1$. Since $\psi_1=f'$ and $\psi_1\cdot f=1$,
\[
\begin{pmatrix}
\gamma_1\cdot \psi_1 & f\cdot \psi_1 \\
\gamma_1\cdot g & f\cdot g 
\end{pmatrix}
=
\begin{pmatrix}
0 & 1 \\
1 & * 
\end{pmatrix}.
\]
Hence, the 3-manifold $Y_{13}=H_1\cup_\Sigma \overline{H_3}$ satisfies $H_1(Y_{13},\Z)=0$. Moreover, the Heegaard splitting of $Y_{13}$ is weakly reducible ($\gamma_1\cap \psi_1=\emptyset$). Hence, by \autoref{lem:Y_is_S3}, $Y_{13}=S^3$ and $X$ is diffeomorphic to $S^4$. 
\end{proof}

The following is a consequence of \autoref{prop:4section_2+1}. 

\begin{theorem}\label{thm:(2,1)_irreducible_L>=6}
Let $\mathcal{T}$ be a c-irreducible quadrisection of a closed $4$-manifold $X$ with $(g,k)=(2,1)$. If $X\not \cong S^4$, then $\Lcal_4^*(\Tcal)\geq 6$.
\end{theorem}

\begin{proof}
Suppose $\Lcal_4^*(\Tcal)<6$. By definition, $\Lcal_4^*(\Tcal)=d_1+d_2+d_3+d_4$ where $d_i$ is the non-negative integer $d_i=d^*(P_{i},P'_{i})$ as in \autoref{fig:islands}. Amongst all possible ways to write $\Lcal_4^*<6$ as the above sum, all satisfy the condition $d_i\leq 2$ and $d_{i+\varepsilon}\leq 1$ for some $i=1,2,3,4$ and $\varepsilon=\pm 1$. \autoref{prop:4section_2+1} implies that, for all such sums, either $\Tcal$ is c-reducible or $X\cong S^4$. 
\end{proof}

We are ready to show \autoref{thm:L>=6}.

\begin{proof}[Proof of \autoref{thm:L>=6}]
Suppose $\Lcal_4^*(\Tcal)<6$. By \autoref{thm:(2,1)_irreducible_L>=6}, $\Tcal$ is c-reducible. 
In particular, $\Tcal$ is the connected sum of two genus-one multisections. Theorem 5.1 of \cite{islambouli2022toric} implies that each summand is either diffeomorphic to $S^1\times S^3$ or $\#^iS^2\times S^2\#^j\mathbb{CP}^2\#^k \overline{\mathbb{CP}^2}$ for some integers $i,j,k$. 
\end{proof}

\subsection{Curves in genus two surfaces} \label{section:curves_in_genus_two}

Before discussing the proof of \autoref{thm:L=6}, we need to understand which curves $z$ are dual to a fixed pair of circles $\{x,y\}$ (see \autoref{prop:scc_in_genus_2}). 

Let $\Sigma $ be a genus two surface with curves $\gamma_1$, $\gamma_2$, and let $F_0$ the component of $\Sigma-\gamma_2$ containing $\gamma_1$. Let $F_1=\Sigma - F_0$ be the other one-holed torus with longitude and meridian $l$ and $m$, respectively. Fix $n\in \Z$ and denote by $x$ and $y$ the simple closed curves in $F_1$ of the form $x=nl+m$, $y=m$ (see \autoref{fig:curves}). 
A properly embedded arc $z\subset F_1$ is of type $(i,j)$ is $|z\cap x|=i$ and $|z\cap y|=j$. We think of arcs $z$ up to isotopies that move the endpoints of $z$. Such arcs are in one-to-one correspondence with simple closed curves in the closed torus $F_1/\partial F$. The following lemma classifies arcs of small type; the proof is left as an exercise to the reader. \autoref{fig:arcs_ij} summarizes the possible arcs. 


\begin{lemma}\label{lem:(1,1)_arcs}
Let $n\neq 0$ be an integer. 
\begin{itemize}[leftmargin=.625in]
\item[$(0,0)$] Every arc of type $(0,0)$ is parallel to $\partial F_1$. 
\item[$(1,1)$] If $|n|\geq 3$, then there is a unique isotopy class of arc, $z=l$, of type $(1,1)$. \\ If $|n|=1,2$, there are exactly two such arcs $z=l$ and $z=l+\frac{2}{n}m$.
\item[$(1,0)$] There are no arcs of type $(1,0)$ if $|n|\geq 2$. \\ If $|n|=1$, there is one arc $z=m$. 
\item[$(0,1)$] There are no arcs of type $(0,1)$ if $|n|\geq 2$. \\ If $|n|=1$, there is one arc $z=l+nm$.
\end{itemize}
\end{lemma}

\begin{figure}[h]
\includegraphics[width=7cm]{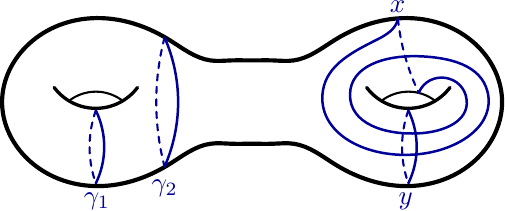}
\centering
\caption{Curves in $\Sigma$ in the setup of \autoref{prop:scc_in_genus_2}.}
\label{fig:curves}
\end{figure}

\begin{figure}[h]
\includegraphics[width=12cm]{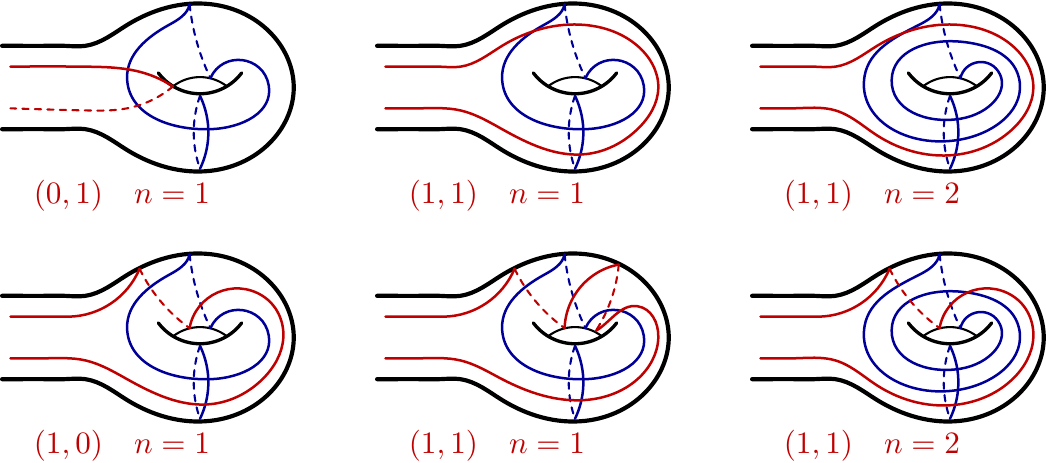}
\centering
\caption{Arcs of type $(i,j)$.}
\label{fig:arcs_ij}
\end{figure}

We now discuss simple closed curves dual to $x$ and $y$. The following result clusters such curves in three families: (0), (2), and (4).  \autoref{fig:curves_ij} gives examples of the arcs in families (2) and (4). 

\begin{proposition}\label{prop:scc_in_genus_2}
Let $\Sigma$, $\gamma_1$, $\gamma_2$, $x$, $y$ be as in \autoref{fig:curves}. Let $z\subset \Sigma$ be a simple closed curve such that $|z\cap x|=|z\cap y|=1$. Then, $|z\cap \gamma_2|\leq 4$ and the curve $z$ has one of the following descriptions:
\begin{itemize}[leftmargin=.5in]
\item[$(0)$] $z$ lies in $F_1$ as torus curve, 
\item[$(2)$] $z$ is obtained by gluing an essential arc in $F_0$ with an arc of type $(1,1)$ in $F_1$, or
\item[$(4)$] $z$ is the union of two parallel copies of an arc in $F_0$ with two arcs in $F_1$, one of type $(0,1)$ and the other of type $(1,0)$. 
\end{itemize}
\end{proposition}

\begin{proof}
Suppose $z$ intersects $\gamma_2$ minimally so the arcs $z\cap F_i$ are essential in $F_i$. By \autoref{lem:(1,1)_arcs}, the type of each arc in $z\cap F_1$ is not $(0,0)$. Thus, since $z$ is dual to both $x$ and $y$, either $z\cap F_1$ consists in one arc of type $(1,1)$ or two arcs (of types $(0,1)$ and $(1,0)$). The former option yields the second conclusion. 
Suppose there is one arc of type $(0,1)$ and another of type $(1,0)$. By \autoref{lem:(1,1)_arcs}, $n=\pm 1$. Observe that the endpoints of such arcs are linked in $\gamma_2=\partial F_1$. In order to make $z$ connected, the endpoints of the two arcs of $z\cap F_0$ must be unlinked in $\gamma_2$. This condition forces them to be parallel in $F_0$. Hence (4) holds. 
\end{proof}

\begin{figure}[h]
\centering
\includegraphics[width=12cm]{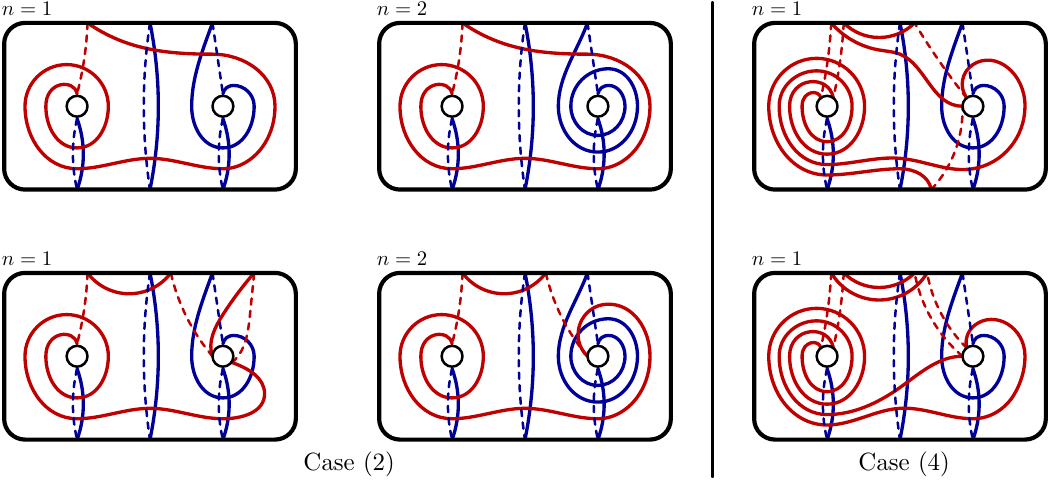}
\caption{Examples of curves $z$ (red) dual to both $x$ and $y$.}
\label{fig:curves_ij}
\end{figure} 

If we think of $\Sigma$ embedded in $S^3$ as in \autoref{fig:curves_ij}, curves of type (2) for distinct arcs of type $(1,1)$ are isotopic in $S^3$. That being said, their surface framings differ by $\pm 1$. 

\begin{remark}\label{remark:family(4)}
Let $z\subset \Sigma$ be a curve in the family (4) and denote by $c$ the isotopy class of the arcs $z\cap F_0$. In the torus $F_0$, $r=|c\cap \gamma_1|$ is equal to the algebraic intersection of $c$ and $\gamma_1$. One can see from \autoref{fig:curves_ij} that $|z \cap \gamma_1|=|z\cdot \gamma_1|=2r$. 
\end{remark}

\subsection{Multisections with \texorpdfstring{$\Lcal_4^*(\Tcal)=6$}{L*(T)=6}}\label{sec: main theorem 2}

\setcounter{theorem}{2}
We now turn to the proof of \autoref{thm:L=6}.

\begin{proof}[Proof of \autoref{thm:L=6}]
Let $\Tcal$ be a $(2,1)$-quadrisection $X=X_1\cup X_2\cup X_3\cup X_4$ with $\Lcal_4^*(\Tcal)=6$.
By the definition of $\Lcal_4^*(\Tcal)$, there exist efficient pairs $(P,P')$, $(Q,Q')$, $(R,R')$, and $(S,S')$ of $(H_1,H_2), (H_3,H_2), (H_3,H_4)$ and $(H_1,H_4)$ respectively, such that $d_1+d_2+d_3+d_4=6$, where $d_1=d^*(S,P)$, $d_2=d^*(P',Q')$, $d_3=d^*(Q,R)$, and $d_4=d^*(R',S')$. We depict this in \autoref{fig:islands_L6}(a). By \autoref{lem:efficient}, $d^*(P,P')=d^*(Q,Q')=d^*(R,R')=d^*(S,S')=1$.

Amongst all possible ways to write six as the above sum, all but one satisfy the condition $d_i\leq 2$ and $d_{i+\varepsilon}\leq 1$ for some $i=1,2,3,4$ and $\varepsilon=\pm 1$. \autoref{prop:4section_2+1} implies that, for all such sums, either $\Tcal$ is c-reducible or $X\cong S^4$. 
The missing sum is $6=3+0+3+0$ which corresponds, without loss of generality, to $d_1=d_3=0$ and $d_2=d_4=3$. We use the notation at the beginning of this section and then obtain that $d(P',Q')=3=d(R',S')$, $P=S=\{g,\psi_1,\psi_2\}$, and $R=Q=\{f,\gamma_1,\gamma_2\}$ as shown in \autoref{fig:islands_L6}(b).

\begin{figure}[h]
\includegraphics[width=\textwidth]{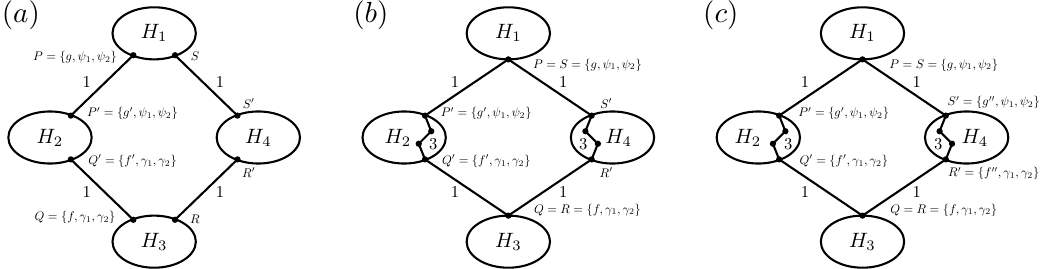}
\centering
\caption{Efficient defining pairs computing $\Lcal_4^*(\Tcal)=6$.}
\label{fig:islands_L6}
\end{figure}

We now divide the proof into three steps.
\begin{framed}
\begin{itemize}[label={},leftmargin=*]
  \item \textbf{Step 1}: Determine $R'=\{f'',r_1,r_2\}$ and $S'=\{g'',\psi_1,\psi_2\}$.
  \item \textbf{Step 2}: Focus on the paths between $(P',Q')$ and $(S',R')$, and claim the following.
  \begin{enumerate}[leftmargin=.9in]
      \item[\autoref{claim:pathmu}.] Either $f'=g'$ or the path between $P'$ and $Q'$ has the following form
\begin{equation*}
P'= \{g',\psi_1,\psi_2\} \mapsto \{g',\psi_1,f'\} \mapsto \{g',\gamma_1,f'\} \mapsto \{\gamma_2,\gamma_1,f'\}=Q'.
\end{equation*}
     \item[\autoref{claim:pathmu'}.] Either $f''=g''$ or the path between $S'$ and $R'$ has the following form
\begin{equation*}
S'= \{g'',\psi_1,\psi_2\} \mapsto \{g'',\psi_1,f''\} \mapsto \{g'',\gamma_1,f''\} \mapsto \{\gamma_2,\gamma_1,f''\}=R'.
\end{equation*}
\end{enumerate}
    \item \textbf{Step 3}: By the previous step, we have the following four cases:
    \begin{itemize}[leftmargin=.9in]
\item[Case (i)] $f'\neq g'$ and $f''=g''$
\item[Case (ii)]$f'\neq g'$ and $f''\neq g''$
\item[Case (iii)] $f'=g'$ and $f''=g''$
\item[Case (iv)]$f'=g'$ and $f''\neq g''$
\end{itemize} However, since (i) and (iv) are symmetric, we only consider (i), (ii), and (iii).
\end{itemize}
\end{framed}

\textbf{Step 1}: Consider the geodesic between the pair $R$ and $R'$. If $\gamma_1\not \in R'$, then $R'=\{f,\gamma_1',\gamma_2\}$ where $|\gamma_1'\cap \gamma_1|=1$. In particular, $Y_{24}=H_2\cup_\Sigma \overline{H_4}$ has a Heegaard diagram given by $\left(\Sigma;\{f',\gamma_1\},\{f,\gamma_1'\}\right)$. The latter is the standard genus two Heegaard splitting for $S^3$, thus weakly reducible. \autoref{lem:Y_is_S3} implies that $X\cong S^4$. Therefore, $R'=\{f'',\gamma_1,\gamma_2\}$, where $f''$ is a simple closed curve on $\Sigma$ such that $|f\cap f''|=1$. By the same argument, $S'=\{g'',\psi_1,\psi_2\}$, where $g''$ is a simple closed curve on $\Sigma$ such that $|f\cap g''|=1$. We are left with the situation depicted in \autoref{fig:islands_L6}(c).

By construction, the curves $f$,$f'$, and $f''$ all lie in the same one-holed torus component of $\Sigma-\gamma_2$. Given that $f$ is dual to both $f'$ and $f''$, we obtain that $|f'\cdot f''|=|f'\cap f''|=n$ for some $n\geq 0$. After a surface diffeomorphism, we can draw these curves as in \autoref{fig:curves}. By the same reasons, $|g'\cdot g''|=|g'\cap g''|=m$ for some $m\geq 0$.

Here we claim that $n=m$. The 3-manifold $Y_{24}=H_2\cup_\Sigma \overline{H_4}$ admits a genus two Heegaard splitting with two distinct Heegaard diagrams, namely $\left(\Sigma;\{f',\gamma_1\},\{f'',\gamma_1\}\right)$, and $\left(\Sigma;\{g',\psi_1\},\{g'',\psi_1\}\right)$. Thus $Y_{24}$ is a connected sum of $S^1\times S^2$ with a lens space with fundamental group $\Z/n\Z \cong \Z/m\Z$. Hence, $n=m$.

\vspace{1em}
\textbf{Step 2}: We now focus on the paths between $(P',Q')$ and $(S',R')$. We prove \autoref{claim:pathmu} and \autoref{claim:pathmu'}.

\begin{shaded}
\begin{claim}\label{claim:pathmu}
Either $f'=g'$ or the path between $P'$ and $Q'$ has the following form
\begin{equation} \tag{1} \label{eq:pathmu}
P'= \{g',\psi_1,\psi_2\} \mapsto \{g',\psi_1,f'\} \mapsto \{g',\gamma_1,f'\} \mapsto \{\gamma_2,\gamma_1,f'\}=Q'.
\end{equation}
\end{claim}

\begin{proof}[Proof of \autoref{claim:pathmu}]
Let $\mu$ be the path of length three connecting $P'$ and $Q'$. We have two possibilities: (1) one curve of $P'$ is fixed along $\mu$, or (2) each curve of $P'$ moves once along $\mu$. In what follows, we will reduce each case to one possibility (Subcases 1c and 2c below). 

\textit{Case 1:} One curve of $P'$ is fixed along $\mu$. 
Equivalently, $P'\cap Q'$ is one curve. As we saw in the proof of \autoref{prop:4section_2+1}, the separating curves $\gamma_2$ and $\psi_2$ must move along $\mu$; i.e., $\{\psi_1,g'\}\cap \{\gamma_1,f'\}=\{x\}$. 

\indent \indent \textit{Subcase 1a:} $\gamma_1=\psi_1$. Here, $\gamma_1$ bounds a disk in $H_1$, $H_2$, and $H_3$. On the other hand, since $\gamma_1\in R'$, $\gamma_1$ bounds a disk in $H_4$ and so $\Tcal$ is c-reducible by \autoref{lem:stabilization}. 

\indent \indent \textit{Subcase 1b:} $\gamma_1=g'$. (This case is analogous to $f'=\psi_1$.) 
Since $\gamma_1$ is fixed ($\gamma_1\in R\cap R'$), then $\gamma_1$ bounds a compressing disk in $H_2$, $H_3$, and $H_4$. Moreover, since $\gamma_1=g'$, $\gamma_1$ is dual to a compressing disk ($g$) in $H_1$. Hence, $\Tcal$ is c-reducible by \autoref{lem:stabilization}. %

\indent \indent \textit{Subcase 1c:} $g'=f'$. This is one of the conclusions.

\textit{Case 2}: Each curve in $P'$ moves once along $\mu$. 
Equivalently, each curve in $Q'$ moves exactly once. We think of $\mu$ as a path $\mu:P'\rightarrow Q'$. 
Notice that all the curves involved in $\mu$ bound compressing disks in $H_2$ and only $\gamma_2$ and $\psi_2$ are separating. Then, one can see that $\psi_2\mapsto x$ and $y\mapsto \gamma_2$ must be the first and third dual moves in $\mu$, where $x\in \{f',\gamma_1\}$ and $y\in \{g',\psi_1\}$. We now discuss the possibilities for $x$ and $y$. 

\indent \indent \textit{Subcase 2a:} $x=\gamma_1$ and $y=g'$. (This case is analogous to $x=f'$ and $y=\psi_1$.) After the dual move $\psi_2\mapsto \gamma_1$, we are left three non-separating curves $\gamma_1$, $\psi_1$, and $g'$. After choosing the correct orientations of these curves, we obtain that $[\gamma_1]=[\psi_1]+[g']$ in $H_1(\Sigma,\Z)$. In particular, $\gamma_1\cdot g =1$. Now, notice that the second move must be $\psi_1\mapsto f'$, leaving us with three non-separating circles in $\Sigma$: $\gamma_1$, $f'$, and $g'$. Thus, $[f']=[\gamma_1]-[g']=[\gamma_1]$ in $H_1(\Sigma,\Z)$ which forces $\gamma_1\cdot f=1$. Hence, the 3-manifold $Y_{13}=H_1\cup_\Sigma \overline{H_3}$ is weakly reducible ($\gamma_1\cap \psi_1=\emptyset$) and its first homology group over $\Z$ is trivial as 
\[
\begin{pmatrix}
\gamma_1\cdot \psi_1 & f\cdot \psi_1 \\
\gamma_1\cdot g & f\cdot g 
\end{pmatrix}
=
\begin{pmatrix}
0 & 1 \\
1 & * 
\end{pmatrix}.
\]
By \autoref{lem:Y_is_S3}, we obtain that $X\cong S^4$. 

\indent \indent \textit{Subcase 2b:} $x=\gamma_1$ and $y=\psi_1$. 
As in Subcase 2a, $[\gamma_1]=[\psi_1]+[g']$ in $H_1(\Sigma,\Z)$. Since $\psi_1\mapsto \gamma_2$ is the last dual move in $\mu$, we also obtain that $[\psi_1]=\pm [\gamma_1]+[f']$. Hence, $\psi_1\cdot f = 1$, $\gamma_1\cdot g=1$, and $\gamma_1\cap \psi_1=\emptyset$. Thus, by \autoref{lem:Y_is_S3}, $X\cong S^4$. 

\indent \indent \textit{Subcase 2c:} $x=f'$ and $y=g'$. This is the second option in the conclusion. This ends the proof of the claim. 
\end{proof}
\end{shaded}

\begin{shaded}
\begin{claim}\label{claim:pathmu'}
Either $f''=g''$ or the path between $S'$ and $R'$ has the following form
\begin{equation}\tag{2}\label{eq:pathmu'}
S'= \{g'',\psi_1,\psi_2\} \mapsto \{g'',\psi_1,f''\} \mapsto \{g'',\gamma_1,f''\} \mapsto \{\gamma_2,\gamma_1,f''\}=R'.
\end{equation}
\end{claim}
\begin{proof}[Proof of \autoref{claim:pathmu'}]
Similar to \autoref{claim:pathmu}.
\end{proof}
\end{shaded}

\textbf{Step 3}:
By \autoref{claim:pathmu} and \autoref{claim:pathmu'}, we now consider the following four cases.
\begin{itemize}[leftmargin=1in,noitemsep]
\item[Case (i)] $f'\neq g'$ and $f''=g''$
\item[Case (ii)]$f'\neq g'$ and $f''\neq g''$
\item[Case (iii)] $f'=g'$ and $f''=g''$
\item[Case (iv)]$f'=g'$ and $f''\neq g''$
\end{itemize}
However, it suffices to consider Cases (i), (ii), and (iii), because (i) and (iv) are symmetric.

\textit{Cases (i) and (ii):}
Assume $f' \neq g'$.
By \autoref{claim:pathmu}, the path between $P'$ and $Q'$ is like \autoref{eq:pathmu}. We focus on the dual move $\psi_1\mapsto \gamma_1$. This occurs in a 4-holed sphere with boundaries corresponding to copies of $f'$ and $g'$ (see \autoref{fig:curves_2}). One can observe that, in this situation ($\psi_1\cap \gamma_1\neq \emptyset$), $\Sigma-(\psi_1\cup \gamma_1)$ is the disjoint union of some polygonal disks and exactly two annuli with cores isotopic to $f'$ and $g'$. This forces any curve $x$ disjoint from $\psi_1\cup \gamma_1$ to be isotopic to one of $\{g',f'\}$. 
Now, by \autoref{claim:pathmu'}, either $g''=f''$ or the path between $S'$ and $R'$ has the following form: 
\[
S'= \{g'',\psi_1,\psi_2\} \mapsto \{g'',\psi_1,f''\} \mapsto \{g'',\gamma_1,f''\} \mapsto \{\gamma_2,\gamma_1,f''\}=R'.
\]
In either case, $f''$ and $g''$ are disjoint from both $\psi_1$ and $\gamma_1$. By the previous paragraph, $\{f'',g''\}$ can be made disjoint from $\{f',g'\}$. Since $\{f',f''\}$ (resp. $\{g',g''\}$) lie in the same one-holed torus of $\Sigma-\gamma_2$ (resp. $\Sigma-\psi_2$), we get that $f'=f''$ (resp. $g'=g''$). 
In particular, $P'=S'=\{g',\psi_1,\psi_2\}$, $Q'=R'=\{f',\gamma_1,\gamma_2\}$, and so the handlebodies $H_2$ and $H_4$ are equal. Thus, $Y_{13}=H_1\cup\overline{H_3}$ cuts $X$ into two copies of the same (bisected) 4-manifold $X_1\cup X_2$. Moreover, since $P=S$ and $Q=R$, the gluing map is the identity in $Y_{13}$, so $X=D(X_1\cup X_2)$ is the double of $X_1\cup X_2$.

\begin{figure}[h]
\centering
\includegraphics[width=7cm]{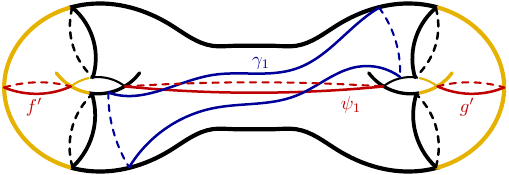}
\caption{A dual move with only non-separating circles.}
\label{fig:curves_2}
\end{figure}

We end this case by drawing a Kirby diagram for $X_1\cup X_2$ and $X$. Recall that $\{g,\psi_1\}$, $\{g',\psi_1\}$, $\{f,\gamma_1\}$ are cut systems for $H_1$, $H_2$, and $H_3$, respectively, $|g\cap g'|=1$, and $f$ is dual to some compressing curve, $g'$, of $H_2$. We use red, blue and green color for $\{g,\psi_1\}$, $\{g',\psi_1\}$ and $\{f,\gamma_1\},$ respectively (see \autoref{fig:abcdef}(a).) By Proposition 3.4 of \cite{islambouli2020multisections}, $X_1\cup X_2$ admits a Kirby diagram with one 0-handle, two dotted 1-handles given by $\psi_1$ and $g$, and 2-handles along $g'$ and $f$ with framing given by $\Sigma$ (see \autoref{fig:abcdef}(b).) 
We cancel the 1/2-pair of handles $\{g,g'\}$ and add a 0-framed unknot linking once around $f$ to obtain a Kirby diagram for $X$ as in \autoref{fig:abcdef}(c). Here, we slide $f$ over the 0-framed unknot to change the crossings and framing of $f$ as needed. 
If $p=|lk(f,\psi_1)|\neq0$, then we can make $f$ to be a $(p,1)$-torus knot as in \autoref{fig:abcdef}(d). This is the Kirby diagram for the spin of a lens space or the twisted spin of a lens space, depending on the framing of $f$. 
If $|lk(f,\psi_1)|=0$, then $f$ and $\psi_1$ can be slid away from each other. Here, we obtain a Kirby diagram for $S^1\times S^3 \# S^2\times S^2$ or $S^1\times S^3\# S^2\widetilde{\times} S^2$, depending on the framing of $f$. 

\begin{figure}[h]
\centering
\includegraphics[width=10cm]{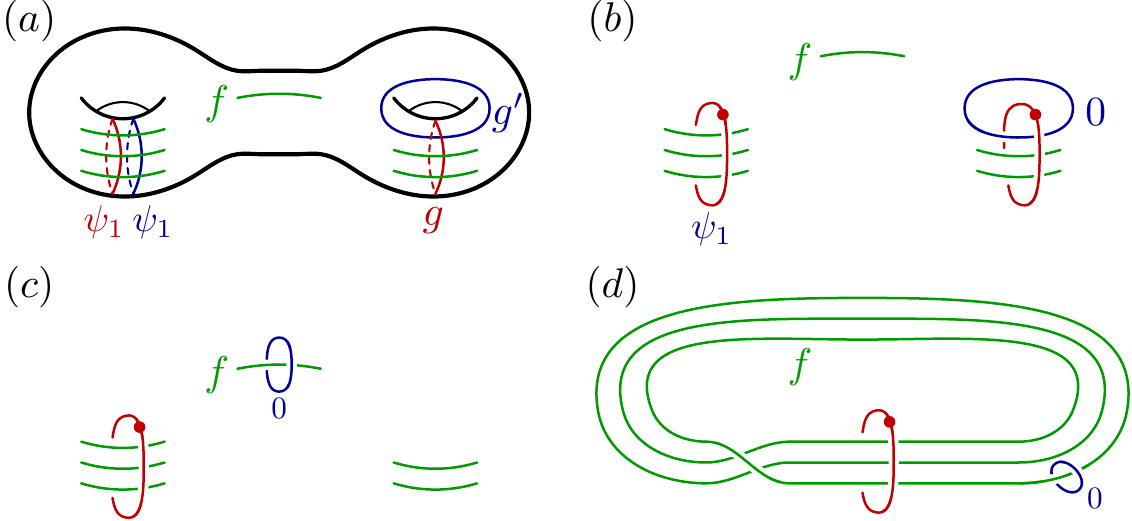}
\caption{(a)-(c) Process of obtaining a Kirby diagram from a multisection diagram. (d) A Kirby diagram of the spin of a lens space or twist-spin of a lens space.}
\label{fig:abcdef}
\end{figure}

\textit{Case (iii):} 
Suppose that $f'=g'$ and $f''=g''$. 
Recall that $\Tcal$ is a $(2,1)$-quadrisection with $\{\gamma_1,f''\}$, $\{\gamma_1,f\}$, $\{\gamma_1,f'\}$, and $\{\psi_1,g\}$ the cut systems of $H_4, H_3, H_2$, and $H_1$, respectively. We also know that $f$ is dual to $f''$, $f'$ is dual to $f$ and $g$ is dual to $f'$. We use red, blue, green, and orange colors for $\{\gamma_1,f''\},\{\gamma_1,f\},\{\gamma_1,f'\}$, and $\{\psi_1,g\}$, respectively. (See \autoref{fig:abcdef4} (a).) 
We use this information to draw a Kirby diagram $\mc{K}$ for $X$ (see Proposition 3.4 of \cite{islambouli2020multisections} and Lemma 4.6 of \cite{MSZ_Classifying}). 
The 0-handle and 1-handles are determined by a neighborhood of $H_4$, the 2-handles are induced by $X_4$, $X_3$, and $X_2$, and the 3-handles and 4-handle are determined by $X_1$. With this in mind, $\mc{K}$ is given as follows: a single $0$-handle, two red dotted $1$-handles given by $\gamma_1$ and $f''$, a blue $2$-handle $f$, a green $2$-handle $f'$, an orange 2-handle $g$, one 3-handle, and one 4-handle. The framing of the 2-handles is given by the surface $\Sigma$. 
\autoref{fig:abcdef4}(a) and \autoref{fig:abcdef4}(b) demonstrate how to create $\mc{K}$ from the curve data. 

\begin{figure}[h]
\centering
\includegraphics[width=12cm]{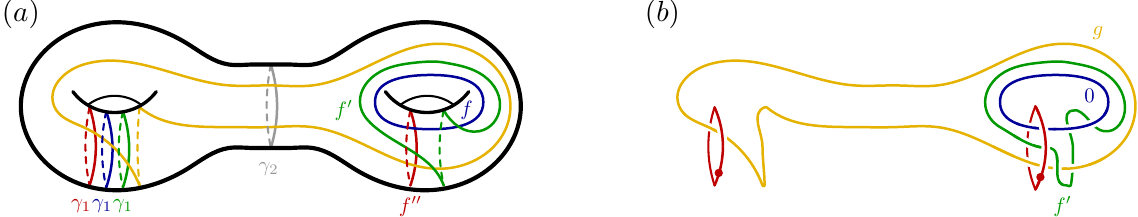}
\caption{(a) Curve data from $\Tcal$. (b) The Kirby diagram $\mc{K}$.}
\label{fig:abcdef4}
\end{figure}

We now use the work in \autoref{section:curves_in_genus_two} to understand the diagram $\mc{K}$. After a surface diffeomorphism, we can depict $f'$ and $f''$ as the curves $x$ and $y$ in \autoref{fig:curves}, respectively. We can also assume that the longitude $l$ of the torus in the right-hand-side of the figure is isotopic to the curve $f$; since $f$ is usual to both $f'$ and $f''$. Here, the intersection of $f'$ and $f''$ is the quantity $n$ in \autoref{section:curves_in_genus_two}. 
In this setup, the curve $z=g$ is dual to both $x=f'$ and $y=f''$ so it satisfies the hypothesis of \autoref{prop:scc_in_genus_2}. Thus, by the proposition, there are three families of curves $z=g$, and so there are families of diagrams, we need to discuss: curves of type (0), (2), or (4) in \autoref{prop:scc_in_genus_2}. 

If $z=g$ is a curve of type (0), then $g$ is disjoint from $\gamma_2$ and lies entirely in the torus on the right side or \autoref{fig:curves}. We can slide $\psi_1$ over $g$ to make $\psi_1$ disjoint from $\gamma_2$. This, will force the curve $\gamma_2$ to bound a disk in $H_4$. Since we knew that $\gamma_2$ bounded disks in $H_1$, $H_2$, and $H_3$ (see \autoref{fig:islands_L6}(c)), the multisection is reducible. Moreover, $\Tcal$ is the connected sum of two genus-one multisections and so $\Lcal_n^*(\Tcal)=0$; this is a contradiction. 

Suppose now that $z=g$ is a curve of type (2) or (4). In order to make this case more intuitive, \autoref{fig:kirby_diagrams_ij} contains the respective Kirby diagrams $\mc{K}$ obtained from the curves in \autoref{fig:curves_ij}. 
By construction, the curves $f$ and $f''$ form a $1/2$-cancelling pair. Since $f$ does not link with any other circle in $\mc K$, removing the pair from $\mc K$ amounts to just erasing the curves $f$ and $f''$ from the diagram. Denote by $
\mc K'$ this new Kirby diagram for $X$. 

\begin{figure}[h]
\includegraphics[width=12cm]{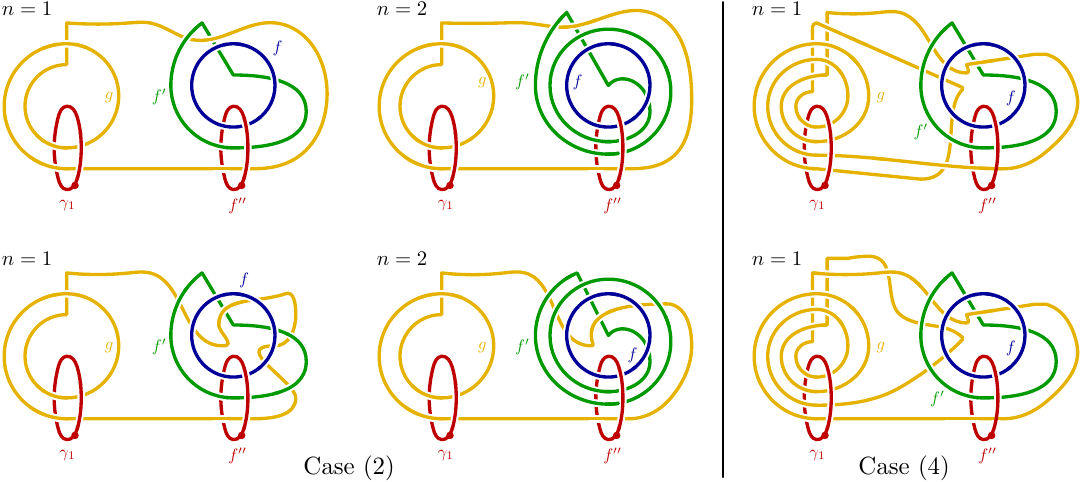}
\centering
\caption{Kirby diagrams $\mc K$ obtained from the curves in \autoref{fig:curves_ij}.}
\label{fig:kirby_diagrams_ij}
\end{figure}

We study the linking matrix of $\mc K'$ with respect to the circles $\{\gamma_1,f',g\}$. By definition, 
$$
A=
\begin{pmatrix}
fr(\gamma_1) & lk(\gamma_1,f') & lk(\gamma_1,g)\\
lk(\gamma_1,f') & fr(f') & lk(f',g)\\
lk(\gamma_1,g) & lk(f',g) & fr(g)
\end{pmatrix}
=
\begin{pmatrix}
0 & 0 & lk(\gamma_1,g)\\
0 & fr(f') & lk(f',g)\\
lk(\gamma_1,g) & lk(f',g) & fr(g)
\end{pmatrix}.
$$
The determinant of the linking matrix must be zero, otherwise, the boundary of $X_1$ is not $S^1\times S^2$. Hence, either $lk(f',g)=0$ or the framing of $f'$ is zero. 

By construction, $f$ is an unknot with framing equal to the quantity $n$ from \autoref{section:curves_in_genus_two}. Moreover, the curve $g$ crosses the disk bounded by $f'$ in exactly one point (see \autoref{fig:kirby_diagrams_ij}) so $lk(f',g)=\pm 1$. 
In particular, if $fr(f')=0$ then $\mc K'$ is the diagram of the double of a 4-manifold with one handle of indices zero, one, and two. The work in Cases (i) and (iii) above (\autoref{fig:abcdef}) can be repeated to conclude that $X$ is diffeomorphic to a (twisted) spun lens space, $S^1\times S^3\#S^2\times S^2$, or $S^1\times S^3\# S^2\widetilde{\times}S^2$.

To end, suppose that $lk(f',g)=0$. By construction, this linking number is equal to the algebraic intersection number of the curves $\gamma_1$ and $g$ inside $\Sigma$. 
If $z=g$ is of type (4), \autoref{remark:family(4)} explains that $g\cdot \gamma_1=0$ if and only if $g$ is disjoint from $\gamma_1$. The same conclusion holds if $g$ is of type (2) (see the statement of \autoref{prop:scc_in_genus_2}). Hence, in either type, the circle $g$ can be isotoped away from $\gamma_1$. In particular, $\mc K'$ has one 0-handle, one 1-handle, and two 2-handles attached along a Hopf link. This diagram represents $S^1\times S^3\#S^2\times S^2$ or $S^1\times S^3\# S^2\widetilde{\times} S^2$. 
This finishes the proof of the theorem. 
\end{proof}


\sloppy
\printbibliography[title={Bibliography}]

$\quad$ \\
Rom\'an Aranda, {Francis Marion University}\\
email: \texttt{romanaranda123@gmail.com}\\ 
$\quad$ \\
Sarah Blackwell, {University of Virginia}\\
email: \texttt{blackwell@virginia.edu}\\
$\quad$ \\
Devashi Gulati, {University of Georgia}\\
email: \texttt{gulati.devashi@gmail.com}\\
$\quad$ \\
Homayun Karimi, {McMaster University}\\
email: \texttt{homayun.karimi@gmail.com }\\
$\quad$ \\
{Geunyoung Kim}, {McMaster University}\\
email: \texttt{kimg68@mcmaster.ca}\\
$\quad$ \\
{Nicholas Paul Meyer}, {Winona State University}\\
email: \texttt{nick.meyer@winona.edu}\\
$\quad$ \\
Puttipong Pongtanapaisan, Pitzer College\\
email: \texttt{puttip@pitzer.edu}\\
$\quad$ \\

\end{document}